\documentclass[a4paper,12pt,leqno]{amsart}
\usepackage{latexsym}
\usepackage[all]{xy}

\usepackage{amssymb} 
\usepackage{amsmath} 
\usepackage{color}
\usepackage{comment}
\usepackage{mathtools}
\usepackage{mathrsfs}

\usepackage{graphicx}
\usepackage{stmaryrd}
\usepackage{url}

\definecolor{gray}{gray}{0.7}
\definecolor{Gray}{gray}{0.3}

\usepackage{amscd}

\numberwithin{equation}{section}

\theoremstyle{break}
 \newtheorem{theorem}{Theorem}[section]
 \newtheorem{proposition}[theorem]{Proposition}
 \newtheorem{corollary}[theorem]{Corollary}
 \newtheorem{lemma}[theorem]{Lemma}

 \theoremstyle{definition}
 \newtheorem{definition}[theorem]{Definition}
 \newtheorem{remark}[theorem]{Remark}
 \newtheorem{example}[theorem]{Example}

\newtheorem*{acknowledgement}{Acknowledgment}

\allowdisplaybreaks[3]

\def\C{\mathbb C}
\def\R{\mathbb R}
\def\Q{\mathbb Q}
\def\Z{\mathbb Z}

\def\O{\mathcal{O}}
\def\Cpp{\mathrm{C}_{++}}
\def\N{N}
\def\Eff{\mathrm{Eff}}
\def\PE{\mathrm{PE}}
\def\Nef{\mathrm{Nef}}
\def\Big{\mathrm{Big}}
\def\Amp{\mathrm{NA}}
\def\NA{\mathrm{NA}}

\DeclareMathOperator{\mult}{mult}
\DeclareMathOperator{\pr}{pr}
\DeclareMathOperator{\length}{length}
\DeclareMathOperator{\maxdeg}{maxdeg}
\DeclareMathOperator{\red}{red}
\DeclareMathOperator{\rank}{rank}
\DeclareMathOperator{\Int}{Int}
\DeclareMathOperator{\WDiv}{WDiv}
\DeclareMathOperator{\CDiv}{CDiv}
\DeclareMathOperator{\Supp}{Supp}

\DeclareMathOperator{\Enu}{\mathcal{N}}

\newcommand{\into}{\hookrightarrow}

\def\t{t}

\renewcommand{\t}[1]{{}^t\!#1}

\begin{document}
\title[An integral version of Zariski decompositions on normal surfaces]{An integral version of Zariski decompositions on normal surfaces}
\author [M. Enokizono]{Makoto Enokizono}
\address{Department of Mathematics, Faculty of Science and Technology, 
Tokyo University of Science, 
2641 Yamazaki, Noda, Chiba 278-8510, Japan}
\email{\url{enokizono_makoto@ma.noda.tus.ac.jp}}


\begin{abstract}
We show that any pseudo-effective divisor on a normal surface decomposes uniquely into its ``integral positive'' part and ``integral negative'' part, which is an integral analog of Zariski decompositions.
By using this decomposition, we give three applications:
a vanishing theorem of divisors on surfaces (a generalization of Kawamata-Viehweg and Miyaoka vanishing theorems), Reider-type theorems of adjoint linear systems on surfaces (including a log version and a relative version of the original one) and extension theorems of morphisms defined on curves on surfaces (generalizations of Serrano and Paoletti's results).
\end{abstract}

\maketitle

\setcounter{tocdepth}{1}

\tableofcontents

\section{Introduction}
\label{sec:Introduction}

In 1962, Zariski showed the following decomposition theorem \cite{Zar}:

\begin{theorem}[Zariski decomposition]
Let $D$ be an effective $\Q$-divisor on a smooth projective surface $X$.
Then there exists the unique decomposition $D=P+N$ such that the following hold.

\smallskip

\noindent
$(\mathrm{i})$ $P$ is a nef $\Q$-divisor on $X$ $($called the positive part of $D$$)$.

\smallskip

\noindent
$(\mathrm{ii})$ $N=0$ or $N>0$ is a negative definite $\Q$-divisor on $X$ $($called the negative part of $D$$)$.

\smallskip

\noindent
$(\mathrm{iii})$ $PC=0$ for any curve $C\subset \Supp(N)$.
\end{theorem}
For a more general setting including the pseudo-effective case \cite{Fuj} and the relative case, see Theorem~\ref{Zardecomp}.
This decomposition $D=P+N$ is called the {\em Zariski decomposition} of $D$.
It is a fundamental tool in algebraic geometry (for higher dimensional generalizations of the Zariski decomposition, see for example \cite{Nak}, \cite{Pro} and references therein).
Note that the positive part $P$ and the negative part $N$ of $D$ are not necessarily $\Z$-divisors even if $D$ is a $\Z$-divisor.
The main theorem in this paper is an integral analog of the Zariski decomposition on surfaces as follows (for a more general setting including relative cases, see Theorem~\ref{IntZardecomp}).
\begin{theorem}[Integral Zariski decomposition]
Let $D$ be a pseudo-effective divisor on a normal complete surface $X$.
Then there exists the unique decomposition $D=P_{\Z}+N_{\Z}$ such that the following hold.

\smallskip

\noindent
$(\mathrm{i})$ $P_{\Z}$ is a $\Z$-positive divisor on $X$.

\smallskip

\noindent
$(\mathrm{ii})$ $N_{\Z}=0$ or $N_{\Z}>0$ is a negative definite divisor on $X$.

\smallskip

\noindent
$(\mathrm{iii})$ $-P_{\Z}$ is nef over $N_{\Z}$.
\end{theorem}
Here a divisor $D$ on $X$ is called {\em $\Z$-positive} if $B-D$ is not nef over $B$ for any effective  negative definite divisor $B>0$ on $X$.
Typical examples of $\Z$-positive divisors are the round-up of nef $\R$-divisors and numerically connected (and not negative definite) divisors.
Integral Zariski decompositions have been naturally appeared informally as a phenomenon in the study of the cohomology of divisors on surfaces (e.g., \cite{Miy}, \cite{Sak2}, \cite{KaMa}, \cite{YZZ}).
As applications of the integral Zariski decomposition, we give generalizations of Kodaira-type vanishing theorems, Reider-type theorems and Serrano's extension theorems on surfaces. 

\subsection{Vanishing theorem}
In the usual Zariski decomposition of $D$ on a surface $X$, the positive part of $D$ measures the asymptotic behavior of the cohomology with respect to $mD$, $m\gg 0$.
For example, the section ring of a big divisor $D$ on $X$ is finitely generated if and only if the positive part $P$ of $D$ is semiample (cf.\ \cite{Laz} Corollary~2.3.23).
Moreover, the cohomology $H^{i}(X,\O_X(K_X+\ulcorner mP \urcorner))$ vanishes for any $i>0$ and $m>0$ by the Miyaoka vanishing theorem \cite{Miy}.
On the other hand, the $\Z$-positive part of $D$ in the integral Zariski decomposition measures the cohomology with respect to $D$ itself.
Indeed, the first cohomology $H^{1}(X,\O_X(K_X+D))$ can be computed by some cohomology on the $\Z$-negative part $N_{\Z}$ for a big divisor $D$ on $X$:

\begin{theorem}[Vanishing theorem]\label{Introvan}
Let $D$ be a big divisor on a normal complete surface $X$ over a field of characteristic $0$.
Let $D=P_{\Z}+N_{\Z}$ be the integral Zariski decomposition of $D$.
Then we have
$$
H^{1}(X,\O_{X}(K_{X}+D))\cong H^{1}(N_{\Z},\mathcal{L}_{D}),
$$
where $\mathcal{L}_{D}$ is the rank $1$ sheaf on $N_{\Z}$ defined by the cokernel of the homomorphism $\O_{X}(K_X+P_{\Z})\to \O_{X}(K_X+D)$ 
induced by multiplying a defining section of $N_{\Z}$.
\end{theorem}
In particular, if $D$ is big and $\Z$-positive, then all the higher cohomologies of $K_X+D$ vanish.
This is a generalization of the Kawamata-Viehweg vanishing \cite{Kaw} \cite{Vie}, Ramanujam's $1$-connected vanishing \cite{Ram} and Miyaoka's vanishing \cite{Miy} on surfaces.
The relative version for Theorem~\ref{Introvan} also holds (see Theorem~\ref{van}), even when positive characteristics,
which is a generalization of the local vanishing theorem due to Sakai~\cite{Sak} and Koll\'ar-Kov\'acs~\cite{KoKo}.

\subsection{Reider-type theorem}
By using Theorem~\ref{Introvan}, we will prove some Reider-type theorems.
For example, the following can be shown:

\begin{theorem}[Reider-type theorem I]\label{Introbpf}
Let $D$ be a big 
 divisor on a normal complete surface $X$ over a field of characteristic $0$.
Let $\zeta$ be a $0$-dimensional subscheme of $X$ along which $K_X+D$ is Cartier.
We further assume that $P^2>\delta_{\zeta}$ $($resp.\ $P_{\Z}^{2}>\delta_{\zeta}$$)$, where $\delta_{\zeta}$ is the invariant of the germ $(X,\zeta)$ $($for the details, see Section 5$)$ and $D=P+N$ $($resp.\ $D=P_{\Z}+N_{\Z}$$)$ the $($resp.\ integral$)$ Zariski decomposition.
Then the restriction map $H^{0}(X,\O_X(K_X+D))\to H^{0}(X,\O_X(K_X+D)|_{\zeta})$ is surjective, or there exists an effective divisor $B$ on $X$ intersecting $\zeta$ such that $(D-B)B\le \delta_{\zeta}/4$ $($resp.\ and $D+N_{\Z}-2B$ is big$)$.
\end{theorem}

It is a generalization of Reider's theorem (\cite{Rei} Theorem~1).
More general statements including the relative case, see Theorem~\ref{sep}.
The relative version of Theorem~\ref{Introbpf} works even for positive characteristics and generalizes the results of Shepherd-Barron \cite{She}, Laufer \cite{Lau} and Sakai \cite{Sak2} for base points of linear systems on resolution spaces of surface singularities.
For example, we obtain as a special case of Theorem~\ref{sep} the following famous results:

\begin{corollary}[Corollaries~\ref{bpfdim0}, \ref{bpfdim1} and \ref{bpfdim2}] \label{Introbpfcor}
Let $f\colon X\to Y$ be a morphism defined over an algebraically closed field satisfying one of the following $(\mathrm{i})$, $(\mathrm{ii})$ and $(\mathrm{iii})$.

\smallskip

\noindent
$(\mathrm{i})$ $X$ is a minimal smooth projective surface of general type with $K_X^2> 4$ and $Y$ is a point and assume the base field is of characteristic $0$.

\smallskip

\noindent
$(\mathrm{ii})$ $f$ is a relatively minimal fiber space from a smooth surface $X$ to a curve $Y$ whose general fiber has arithmetic genus greater than $1$.

\smallskip

\noindent
$(\mathrm{iii})$ $f$ is a minimal resolution of a normal surface singularity $(Y,y)$.

\smallskip

\noindent
Then the natural homomorphism $f^{*}f_{*}\O_X(mK_X)\to \O_X(mK_X)$ is surjective for $m\ge 2$.
\end{corollary}

Moreover, we can prove another variant of Reider-type theorems which says the existence of negative semi-definite curves satisfying the ``Reider-type'' condition:

\begin{theorem}[Reider-type theorem II] \label{Intronegsemidefnthm}
Let $D$ be a nef and big divisor on a normal complete surface $X$ over a field of characteristic $0$. 
Let $\zeta$ be a $0$-dimensional subscheme of $X$ along which $K_X+D$ is Cartier.
Assume that $D^2>\delta'_{\zeta}$ and $H^{0}(X,\O_{X}(K_X+D))\to H^{0}(X,\O_{X}(K_X+D)|_{\zeta})$ is not surjective, where $\delta'_{\zeta}$ is the invariant of the pair $(X,\zeta)$ greater than or equal to $\delta_{\zeta}$ $($for the details, see Section 5$)$.
Then there exists a negative semi-definite effective divisor $B$ on $X$ intersecting $\zeta$ such that 
$0<(D-B)B\le \delta_{\zeta}/4$ and $D-2B$ is big.
\end{theorem}


\subsection{Extension theorem}
For a normal complete surface $X$,
we define the invariant $q_X$ as
$$
q_X:=\mathrm{min}\{E^2\ |\ \text{$E$ is an effective divisor on $X$ with $E^2>0$} \}.
$$
For example, $q_X\ge 1$ if $X$ is smooth and $q_X\ge 2$ if $X$ is smooth and $K_X$ is numerically even.
By using Theorem~\ref{Intronegsemidefnthm},
we can generalize Serrano's extension theorem \cite{Ser} as the following (this is a simplified version of Theorem~\ref{extnthm}):

\begin{theorem}[Extension theorem] \label{introextnthm}
Let $X$ be a normal complete surface over a field of characteristic $0$.
Let $D>0$ be an effective divisor on $X$ such that any prime component $D_i$ of $D$ has positive self-intersection number $D_i^2>0$.
Let $\varphi\colon D\to \mathbb{P}^1$ be a finite morphism of degree $d$
and put $q_{X,d}:=\mathrm{min}\{q_X,d\}$.
If $D^2>q_{X,d}(d/q_{X,d}+1)^2$, then there exists a morphism $\psi\colon X\to \mathbb{P}^1$ such that the restriction of $\psi$ to $D$ is equal to $\varphi$.
\end{theorem}

Historically, the extension theorem on surfaces as above was firstly proved by Reid in \cite{Reid} when $X$ is a $K3$ surface and $D$ is a smooth irreducible curve.
After that, Serrano \cite{Ser} proved this when $X$ is a smooth projective surface and $D$ is smooth and irreducible, and Paoletti \cite{Pao} improved Serrano's extension theorem to the case that $X$ is smooth and $D$ is a prime divisor.
For extension theorems on higher dimensional varieties, see \cite{Som}, \cite{BeIo} and references therein.

Theorem~\ref{introextnthm} has many applications to curves on normal surfaces.
For example, we can give the lower bound of the gonality of curves on a normal surface of Picard number $1$ (e.g., a weighted projective plane) which generalizes the classical result \cite{Nam} that a smooth complex plane curve of degree $m$ has gonality $m-1$.
For other applications, see Section 7.

\subsection{Structure of the paper}
The present paper is organized as follows.
In Section 2, we fix some notations and terminology used in this paper.
In Section 3, we first give another proof of Zariski decompositions in the usual sense, which is as simple as Bauer's one in \cite{Bau}.
Next, we give a proof of integral Zariski decompositions.
We study several properties of $\Z$-positive divisors in the rest of the section.
The $\Z$-positivity is characterized by using the usual Zariski decomposition and connecting chains (Proposition~\ref{connchain}). 
One important property is that the $\Z$-positivity of divisors are preserved by the round-up of the Mumford pull-back by a proper birational morphism (Proposition~\ref{pullback}).
In Section 4, we study vanishing theorems for adjoint divisors $K_X+D$.
Following Sakai's argument in \cite{Sak}, we give a proof of the (semi-)local vanishing theorem (Theorem~\ref{van} (2)).
We prove the absolute version of the vanishing theorem (Theorem~\ref{van} (1)) by using the Kawamata-Viehweg vanishing theorem as a starting point.
In Section 5, we study higher order separations of adjoint (relative) linear systems $|K_X+D|$.
We first define the invariant $\delta_{\zeta}$ for a $0$-dimensional subscheme $\zeta$ and give a proof of the Reider-type theorem I (Theorem~\ref{sep}). 
Next, we prove the Reider-type theorem II (Theorem~\ref{negsemidefnthm}) by using Theorem~\ref{sep} and the Hodge index theorem.
Section 6 is devoted to the proof of the extension theorem (Theorem~\ref{extnthm}) as an application of Theorem~\ref{negsemidefnthm}.
In Section 7, we give some applications of the extension theorems.
In Appendix A, we collect some basic results for Mumford's intersection theory on a normal surface for the convenience of readers.
In Appendix B, we collect definitions and basic properties of some invariants of proper curves used in Section 6 and Section 7.

\begin{acknowledgement}
The author is grateful to Prof.\ Kazuhiro Konno for valuable comments on linear systems on curves.
He would like to thank Tatsuro Kawakami for nice discussions.
He also would like to thank Hiroto Akaike and Takumi Murayama for sending me some helpful papers.
He also appreciates Prof.\ Fei Ye discussing Remark~\ref{zartointzar}.
He was partially supported by JSPS Grant-in-Aid for Research Activity Start-up: 19K23407. 
\end{acknowledgement}

\section{Notations}
\label{sec:Notations}

\begin{itemize}

\item
In this paper, we mainly work in the category of algebraic schemes over a base field $k$.
All results in this paper hold true in the category of complex analytic spaces.

\item
A {\em variety} means an irreducible and reduced scheme which is separated and of finite type over $k$ (or an irreducible, reduced, paracompact and Hausdorff analytic space).

\item
A {\em surface} means a variety of dimension $2$.
But in this paper, a {\em curve} means a $1$-dimensional scheme which is separated and of finite type over $k$ (not necessarily assume irreducible or reduced).
Note that any normal surface $X$ and any curve $C$ on $X$ are both Cohen-Macaulay by the following exact sequence:

\begin{equation} \label{restex}
0\to \O_X(-C)\to \O_X\to \O_C\to 0.
\end{equation}

\item
A {\em divisor} means a Weil divisor (not necessarily $\Q$-Cartier).
We sometimes call it a $\Z$-divisor.
A $\Q$-(resp.\  $\R$-)divisor is a $\Q$-(resp.\  $\R$-)linear combination of divisors.
For two $\R$-divisors $D$ and $E$, we write $D\ge E$ (resp.\ $D>E$) if $D-E$ is effective (resp.\ and non-zero).

\item
For an $\R$-divisor $D=\sum_{i}a_i C_i$ where the $C_i$ are distinct prime divisors,
we define the {\em round-down} $\llcorner D \lrcorner:=\sum_{i}\llcorner a_i \lrcorner C_i$, the {\em round-up} $\ulcorner D \urcorner:=\sum_{i}\ulcorner a_i \urcorner C_i$ and the {\em fractional part} $\{D\}:=D-\llcorner D \lrcorner$ of $D$,
where $\llcorner a \lrcorner$ and $\ulcorner a \urcorner$ are the greatest integer not exceeding $a$ and the least integer not less than $a$, respectively.

\item
For a curve $C$ on a normal surface $X$, we sometimes regard it as an effective divisor $C=\sum_{i}m_iC_i$ by attaching each prime component $C_i$ to the geometric multiplicity $m_i:=\mathrm{length}(\O_{C,\eta_i})$, where $\eta_i$ is the generic point of $C_i$.

\item
For a normal (resp.\ Cohen-Macaulay) variety $X$, let $K_X$ (resp.\ $\omega_X$) denote the canonical divisor (resp.\ dualizing sheaf) on $X$.
Note that for a curve $C$ on a normal surface $X$, we have $\omega_X=\O_X(K_X)$ and $\omega_C=\mathscr{E}xt^{1}(\O_C,\omega_X)$.

\item
Let us denote $h^{i}(\mathcal{F}):=\mathrm{dim}_{k}H^{i}(X,\mathcal{F})$ as the dimension of $H^{i}(X,\mathcal{F})$ as $k$-vector spaces for a proper $k$-scheme $X$ and a coherent sheaf $\mathcal{F}$ on $X$,
and put $\chi(\mathcal{F}):=\sum_{i}(-1)^{i}h^{i}(\mathcal{F})$.

\item
For a Cohen-Macaulay curve $C$ over a field $k$, let us define $p_a(C):=h^{0}(\omega_C)$ and call it the {\em arithmetic genus} of $C$, which coincides with $h^{1}(\O_C)$ by Serre duality.

\item
For a morphism $f\colon X\to Y$ of schemes, the scheme-theoretic image of $f$ is denoted by $f(X)$.
Then $f$ is {\em scheme-theoretically surjective} if $f(X)=Y$ holds.
For a subscheme $Z$ of $Y$, we denote the fiber product $X\times_{Y}Z$ by $f^{-1}(Z)$.

\item
Throughout this paper, $X$ is a normal surface and $f\colon X\to Y$ is a proper surjective morphism to a variety $Y$ unless otherwise stated.
(In the analytic setting, we always assume that $Y$ is compact or $f$ can be extended to a proper surjective morphism $\overline{f}\colon \overline{X}\to \overline{Y}$ from a normal surface $\overline{X}$ to a variety $\overline{Y}$ such that $Y$ is a relatively compact open subset of $\overline{Y}$).

\item
Throughout this paper, we freely use Mumford's intersection theory and some numerical properties of Weil divisors on normal surfaces (see Appendix~A).

\item
An $\R$-divisor $D$ on $X$ is {\em $f$-exceptional} if each irreducible component of $\Supp(D)$ maps to a point by $f$ (note that this is not a standard definition).

\item
A non-zero effective $\R$-divisor $D=\sum_{i}a_iC_i$ on $X$ is {\em negative $($semi-$)$definite}
if the intersection matrix $(C_iC_j)_{ij}$ of prime components of $D$ is negative (semi-)definite.

\item
Let $D$ be an $\R$-divisor and $B$ an effective $f$-exceptional $\R$-divisor on $X$.
Then $D$ is {\em nef over $B$} if $DC\ge 0$ holds for each prime component $C$ of $B$.
\end{itemize}


\section{Zariski decomposition on a normal surface}
\label{sec:Zariski decomposition on a normal surface}


Let $f\colon X\to Y$ be a proper surjective morphism from a normal surface $X$ to a variety $Y$.
Recall the following famous result of Zariski \cite{Zar} (the pseudo-effective case is due to Fujita \cite{Fuj}).
\begin{theorem}[Zariski decomposition]\label{Zardecomp}
Let $D$ be an $f$-pseudo-effective $\R$-$($resp.\ $\Q$-$)$divisor on $X$.
Then there exists the unique decomposition $D=P+N$ such that the following hold.

\smallskip

\noindent
$(\mathrm{i})$ $P$ is an $f$-nef $\R$-$($resp.\ $\Q$-$)$divisor on $X$.

\smallskip

\noindent
$(\mathrm{ii})$ $N=0$ or $N>0$ is a negative definite $f$-exceptional $\R$-$($resp.\ $\Q$-$)$divisor on $X$

\smallskip

\noindent
$(\mathrm{iii})$ $PC=0$ holds for any curve $C\subset \Supp(N)$.
\end{theorem}

First, we define the notion of the positivity of divisors on a surface $X$.

\begin{definition}
Let $R:=\Z$, $\Q$ or $\R$.
For an $\R$-divisor $D$ on $X$, let $\Enu_{R}(D)$ be the set of negative definite $f$-exceptional $R$-divisors $B>0$ on $X$ with $B-D$ nef over $B$.
Then $D$ is called {\em $f$-$R$-positive} if $\Enu_{R}(D)=\emptyset$.
\end{definition}

\begin{lemma}\label{nef=pos}
Let $D$ be an $f$-pseudo-effective $\R$-divisor on $X$.
Then the following are equivalent.

\smallskip

\noindent
$(1)$ $D$ is $f$-nef.
\smallskip

\noindent
$(2)$ $D$ is $f$-$\R$-positive.

\smallskip

\noindent
$(3)$ $D$ is $f$-$\Q$-positive.
\end{lemma}

\begin{proof}
If $\Enu_{\R}(D)$ has an element $B$, then $DB=-(B-D)B+B^2<0$ holds since $B-D$ is nef over $B$ and $B>0$ is negative definite.
Hence (1) implies (2).
It is clear that $f$-$\R$-positive implies $f$-$\Q$-positive.
Suppose that $D$ is $f$-$\Q$-positive.
Let $C$ be an arbitrary $f$-exceptional irreducible curve on $X$.
If $C^{2}\ge 0$, then $DC\ge 0$ since $C$ is nef and $D$ is $f$-pseudo-effective.
Then we may assume that $C^2<0$, that is, $C$ is negative definite.
Let $\varepsilon>0$ be a positive rational number and put $B:=\varepsilon C$.
Since $D$ is $f$-$\Q$-positive, we have $(B-D)C<0$.
Thus $DC>\varepsilon C^2$ holds.
Since $\varepsilon >0$ is arbitrary, we have $DC\ge 0$.
Hence $D$ is $f$-nef.
\end{proof}

\begin{remark}
For not necessarily $f$-pseudo-effective $\R$-divisors on $X$, we have 
$$
\text{$f$-nef}\  \Longrightarrow \  \text{$f$-$\R$-positive} \  \Longrightarrow \  \text{$f$-$\Q$-positive}  \ \Longrightarrow  \ \text{$f$-$\Z$-positive}.
$$
\end{remark}

The main theorem in this section is the following, which is a $\Z$-version of Theorem~\ref{Zardecomp}.
\begin{theorem}[Integral Zariski decomposition]\label{IntZardecomp}
Let $D$ be an $f$-pseudo-effective $\R$-divisor on $X$.
Then there exists the unique decomposition $D=P_{\Z}+N_{\Z}$ such that the following hold.

\smallskip

\noindent
$(\mathrm{i})$ $P_{\Z}$ is an $f$-$\Z$-positive $\R$-divisor on $X$.

\smallskip

\noindent
$(\mathrm{ii})$ $N_{\Z}=0$ or $N_{\Z}>0$ is a negative definite $f$-exceptional $\Z$-divisor on $X$

\smallskip

\noindent
$(\mathrm{iii})$ $-P_{\Z}$ is nef over $N_{\Z}$.
\end{theorem}

\subsection{Proof of Zariski decompositions}

We start with an easy lemma.
\begin{lemma} \label{easy}
Let $D$ be an $\R$-divisor on $X$ and $D=F+E=A+B$ two decompositions of $D$ as $\R$-divisors which satisfy the following conditions.

\smallskip

\noindent
$(\mathrm{i})$ $E$ is effective.

\smallskip

\noindent
$(\mathrm{ii})$ $B$ is  non-zero effective, negative definite and $f$-exceptional.

\smallskip

\noindent
$(\mathrm{iii})$ $F$ and $-A$ are nef over $B$. 

\smallskip

\noindent
Then $B\le E$ holds.
\end{lemma}

\begin{proof}
Let $B=\sum_{i}b_{i}C_{i}$ and $E=\sum_{i}e_{i}C_{i}$ be prime decompositions and put $G:=\sum_{i}\mathrm{min}\{b_{i},e_{i}\}C_{i}$.
Clearly, we have $G\le B$ and $G\le E$.
In order to prove $B\le E$, it suffices to show that $G=B$.
Suppose that $G<B$.
Then $(B-G)^{2}<0$ holds since $B-G$ is non-zero and negative definite.
Moreover, we have $(E-G)(B-G)\ge 0$ and $A(B-G)\le 0$ since $E-G$ and $B-G$ has no common components and $-A$ is nef over $B$.
Thus we have
$$
F(B-G)=(A+B-E)(B-G)=A(B-G)+(B-G)^{2}-(E-G)(B-G)<0,
$$
which contradicts the condition that $F$ is nef over $B$.
Hence we obtain $B\le E$.
\end{proof}

\begin{lemma} \label{existmax}
Let $D$ be an $\R$-divisor on $X$.
If $\Enu_{\R}(D)\neq \emptyset$, then it has a maximal element.
\end{lemma}

\begin{proof}
By Zorn's lemma, it is enough to show that $\Enu_{\R}(D)$ is inductive, that is, any non-empty totally ordered subset $\mathcal{T}\subset \Enu_{\R}(D)$ has an upper bound, 
where we consider the inclusion order on $\Enu_{\R}(D)$.
For such a subset $\mathcal{T}\subset \Enu_{\R}(D)$, we put $B_{\mathcal{T}}:=\sum_{i}b_{i}C_{i}$, where the sum is taken over all $f$-exceptional integral curves $C_i$ and $b_{i}:=\mathrm{sup}\{\mult_{C_{i}}(B)\ |\ B\in \mathcal{T}\}$.
We will show that this is well-defined and gives an upper bound of $\mathcal{T}$.
We first assume that there is an infinite sequence $\{C_{i_{n}}\}_{n}$ with $b_{i_{n}}>0$ for each $n$.
Then for each $n$, there exists an element $B_n$ of $\mathcal{T}$ with $\mult_{C_{i_{n}}}(B_n)>0$.
Since $\mathcal{T}$ is totally ordered, we have $B_1\le B_2$ or $B_2\le B_1$.
Thus we may assume that $\mult_{C_{i_1}}(B_2)>0$.
Similarly, we can take an element $B_n$ of $\mathcal{T}$ for each $n$ such that the number of irreducible components of $B_n$ is not less than $n$.
Taking $n>\dim \N^{1}(X/Y)$, it is a contradiction because $B_n$ is negative definite and then all irreducible components of $B_n$ are linearly independent in $\N^{1}(X/Y)$.
Hence we have $b_i=0$ except for a finite number of $C_i$'s.
Next, we assume that $b_i=\infty$ holds for some $i$.
Then we can take an infinite sequence $\{B_n\}_{n}$ in $\mathcal{T}$ such that $\lim_{n\to \infty}\mult_{C_i}(B_n)=\infty$.
In particular, we have $\lim_{n\to \infty}B_{n}^{2}=-\infty$ since $B_n$ is negative definite and 
any irreducible component of $B_n$ is contained in $\Supp(B_{\mathcal{T}})$.
On the other hand, it follows that $(B_n-D)B_n\ge 0$ since $B_n-D$ is nef over $B_n$.
Thus we have $DB_n\le B_{n}^{2}$, or equivalently, $D(B_n/\sqrt{-B_{n}^{2}})\le -\sqrt{-B_{n}^{2}}$.
Since the multiplication map $D- \colon \N_{1}(X/Y)\to \R$ has the minimum on the compact subset $K:=\{B\in \bigoplus_{C_i\subset \Supp(B_{\mathcal{T}})}\R C_i\ |\ B^{2}=-1\}$, 
we have a contradiction by taking $n\to \infty$.
Hence we have $b_i<\infty$ for each $i$ and then $B_{\mathcal{T}}$ is well-defined as an $\R$-divisor.
Moreover, $B_{\mathcal{T}}$ is negative definite since we can take an element $B\in \mathcal{T}$ with $\Supp(B)=\Supp(B_{\mathcal{T}})$.
For any prime component $C_i$ of $B_{\mathcal{T}}$ and $\varepsilon>0$, we can take $B_{\varepsilon}\in \mathcal{T}$ satisfying $\mult_{C_i}(B_{\varepsilon})>0$ and $0\le b_i-\mult_{C_i}(B_{\varepsilon})<\varepsilon$.
Since $B_{\varepsilon}-D$ is nef over $B_{\varepsilon}$, we have 
$$
(B_{\mathcal{T}}-D)C_i=(B_{\varepsilon}-D)C_i+(B_{\mathcal{T}}-B_{\varepsilon})C_i\ge (b_i-\mult_{C_i}(B_{\varepsilon}))C_{i}^{2}> \varepsilon C_{i}^{2}.
$$
Since $\varepsilon>0$ is arbitrary, we have $(B_{\mathcal{T}}-D)C_i\ge 0$, whence $B_{\mathcal{T}}-D$ is nef over $B_{\mathcal{T}}$.
Thus $B_{\mathcal{T}}$ is an element of $\Enu_{\R}(D)$, which is an upper bound of $\mathcal{T}$ by definition.
\end{proof}

\begin{lemma} \label{D-Bpeff}
Let $D$ be an $f$-pseudo-effective $\R$-divisor on $X$.
Then $D-B$ is $f$-pseudo-effective for any $B\in \Enu_{\R}(D)$.
\end{lemma}

\begin{proof}
We take effective $\R$-divisors $D_n$ on $X$ such that $D_n\to D$ in $\N^{1}(X/Y)$ $(n\to \infty)$.
Let us write $D_n-B=G_{n}^{+}-G_{n}^{-}$, where $G_{n}^{+}$, $G_{n}^{-}$ are effective $\R$-divisors having no common components.
In order to prove $D-B$ is $f$-pseudo-effective, it is enough to show that $G_{n}^{-}\to 0$ in $\N^{1}(X/Y)$ $(n\to \infty)$.
Note that $G_{n}^{-}\le B$ and then $G_{n}^{-}$ is negative definite and $f$-exceptional.
Thus we have
$$
(B-D_n)G_{n}^{-}=(G_{n}^{-}-G_{n}^{+})G_{n}^{-}\le (G_{n}^{-})^{2}\le 0.
$$
Since the sequence of multiplication maps $\{(B-D_{n})-\colon \N_{1}(X/Y)\to \R\}_{n}$ converges uniformly to the non-negative function $(B-D)-$ on the compact subset $K:=\{E\ |\ 0\le E\le B\}$ of $\N_{1}(X/Y)$, we have $(B-D_n)G_{n}^{-}\to 0$ $(n\to \infty)$, whence $(G_{n}^{-})^{2}\to 0$ $(n\to \infty)$.
Since $G_{n}^{-}$ is negative definite, we have $G_{n}^{-}\to 0$ $(n\to \infty)$.
\end{proof}

For an $f$-pseudo-effective $\R$-divisor $D$ on $X$, we take a maximal element $N_{\R}$ in $\Enu_{\R}(D)$.
If $\Enu_{\R}(D)= \emptyset$, we define $N_{\R}:=0$.
Put $P_{\R}:=D-N_{\R}$.
\begin{lemma} \label{Pisnef}
$P_{\R}$ is $f$-nef.
\end{lemma}

\begin{proof}
If $P_{\R}=D$, the claim holds from Lemma~\ref{nef=pos}.
Thus we may assume that $N_{\R}>0$.
By Lemma~\ref{D-Bpeff}, $P_{\R}$ is $f$-pseudo-effective.
Thus it suffices to show that $P_{\R}C\ge 0$ for any $f$-exceptional irreducible curve $C$ on $X$ with $C^{2}<0$.
Suppose that there exists such a curve $C$ on $X$ such that $P_{\R}C<0$.
Take a small number $\varepsilon>0$ such that $(P_{\R}-\varepsilon C)C<0$.
Since $-P_{\R}=N_{\R}-D$ is nef over $N_{\R}$, it follows that $\varepsilon C+N_{\R}-D$ is nef over $\varepsilon C+N_{\R}$.
Hence, it suffices to show that $\varepsilon C+N_{\R}$ is negative definite since it contradicts the maximality of $N_{\R}$.
If $\varepsilon C+N_{\R}$ is not negative definite, then there exists an effective $f$-exceptional nef divisor $Z$ on $X$ with $\Supp(Z)=\Supp(\varepsilon C+N_{\R})$ by Lemma~\ref{linalg}.
Since $P_{\R}$ is $f$-pseudo-effective, it follows that $(D-N_{\R})Z\ge 0$.
On the other hand, since $N_{\R}-D$ is nef over $N_{\R}$ and $(N_{\R}-D)C>0$, we have $(N_{\R}-D)Z\le 0$.
Hence we have $(N_{\R}-D)Z=0$ and $\Supp(Z)\subset \Supp(N_{\R})$, which is a contradiction.
\end{proof}

\begin{corollary}
If $\Enu_{\R}(D)\neq \emptyset$, then $N_{\R}$ is the maximum element of $\Enu_{\R}(D)$.
\end{corollary}

\begin{proof}
We take two maximal elements $N_{\R}$ and $\N'_{\R}$ of $\Enu_{\R}(D)$ and write 
$$
D=P_{\R}+N_{\R}=P'_{\R}+N'_{\R}.
$$
Since $P_{\R}$ and $-P'_{\R}$ are nef over $N'_{\R}$, we have $N'_{\R}\le N_{\R}$ by Lemma~\ref{easy}.
By maximality, we have $N_{\R}=N'_{\R}$.
\end{proof}

\begin{remark}
We can take a maximal element of $\Enu_{\R}(D)$ from Lemma~\ref{existmax} even if $D$ is not $f$-pseudo-effective.
But it is not necessarily unique in general.
For example, let $X\to \Sigma_e$ be a blow-up of the $e$-th Hirzebruch surface $\Sigma_e$ ($e\ge 3$) at a point.
Then we can see that $\Enu_{\R}(K_X)$ consists of two elements (they are prime components of the unique reducible fiber of the rulings on $X$).
\end{remark}

\begin{proof}[Proof of Theorem~\ref{Zardecomp}]
The decomposition $D=P_{\R}+N_{\R}$ satisfies the condition (i), (ii) and (iii) in Theorem~\ref{Zardecomp} from Lemma~\ref{Pisnef} and the definition of $P_{\R}$.
The uniqueness of the decomposition satisfying (i), (ii) and (iii) follows from Lemma~\ref{easy}.
If $D$ is a $\Q$-divisor, then the negative part $N_{\R}=\sum_{i}a_{i}C_{i}$ is also a $\Q$-divisor since $DC_{j}=N_{\R}C_{j}=\sum_{i}a_{i}(C_{i}C_{j})$ is rational for each $j$ and the matrix $(C_{i}C_{j})_{ij}$ is negative definite over $\Q$.
\end{proof}

Let $D$ be an $f$-pseudo-effective $\R$-divisor on $X$ and $D=P_{\R}+N_{\R}$ the Zariski decomposition in Theorem~\ref{Zardecomp}.

\begin{lemma}
If $\Enu_{\Z}(D)\neq \emptyset$, it has the maximum element.
\end{lemma}

\begin{proof}
Note that each element $B$ of $\Enu_{\Z}(D)$ is a subdivisor of $N_{\R}$ since $P_{\R}$ and $B-D$ are nef over $B$ and Lemma~\ref{easy}.
In particular, $\Enu_{\Z}(D)$ is a finite set.
Let $B=\sum_{i}b_{i}C_{i}$ and $B'=\sum_{i}b'_{i}C_{i}$ be two divisors in $\Enu_{\Z}(D)$
and put $B'':=\sum_{i}\max\{b_{i},b'_{i}\}C_{i}$.
For each prime component $C_{i}\le B''$, we have
$$
(B''-D)C_{i}=(B-D)C_{i}+(B''-B)C_{i}\ge 0
$$
when $b'_{i}\le b_{i}$.
Similarly, we also have $(B''-D)C_{i}\ge 0$ in the case of $b_{i}\le b'_{i}$.
Thus $B''$ belongs to $\Enu_{\Z}(D)$, whence the claim holds.
\end{proof}

Let $N_{\Z}$ be the maximum element of $\Enu_{\Z}(D)$ if $\Enu_{\Z}(D)\neq \emptyset$,  or $N_{\Z}:=0$ if $\Enu_{\Z}(D)=\emptyset$.
Put $P_{\Z}:=D-N_{\Z}$.
Note that the decomposition $P_{\Z}=P_{\R}+(N_{\R}-N_{\Z})$ is the Zariski decomposition of $P_{\Z}$ since $N_{\Z}\le N_{\R}$.

\begin{lemma}\label{Pispos}
$P_{\Z}$ is $f$-$\Z$-positive.
\end{lemma}

\begin{proof}
We may assume that $N_{\Z}>0$.
Suppose that there exists an element $B$ of $\Enu_{\Z}(P_{\Z})$.
Then we obtain two decompositions $P_{\R}+(N_{\R}-N_{\Z})=(P_{\Z}-B)+B$ of $P_{\Z}$.
Since $P_{\R}$ and $B-P_{\Z}$ are nef over $B$, we have $B\le N_{\R}-N_{\Z}$ by Lemma~\ref{easy}.
In particular, $B+N_{\Z}$ is negative definite.
Thus it suffices to show that $B+N_{\Z}-D$ is nef over $B+N_{\Z}$ since it contradicts the maximality of $N_{\Z}$.
For any subcurve $C\le B$, we have $(B+N_{\Z}-D)C\ge 0$ since $B-P_{\Z}=B+N_{\Z}-D$ is nef over $B$.
For any irreducible curve $C$ in $N_{\Z}$ not contained in $B$, we have
$$
(B+N_{\Z}-D)C=BC+(N_{\Z}-D)C\ge 0,
$$
since $N_{\Z}-D$ is nef over $N_{\Z}$.
Hence the claim holds.
\end{proof}

\begin{proof}[Proof of Theorem~\ref{IntZardecomp}]
The decomposition $D=P_{\Z}+N_{\Z}$ satisfies the condition (i), (ii) and (iii) in Theorem~\ref{IntZardecomp} by Lemma~\ref{Pispos}.
Let $D=P'_{\Z}+N'_{\Z}$ be another decomposition satisfying (i), (ii) and (iii) in Theorem~\ref{IntZardecomp}.
If $N'_{\Z}=0$, then $D=P'_{\Z}$ is $f$-$\Z$-positive and so $N_{\Z}=0$.
If $N'_{\Z}>0$, then $N'_{\Z}$ belongs to $\Enu_{\Z}(D)$ by the condition (iii).
Thus $N'_{\Z}\le N_{\Z}$ holds.
Suppose $N'_{\Z}<N_{\Z}$.
Then $N_{\Z}-N'_{\Z}$ is an element of $\Enu_{\Z}(P'_{\Z})$ since it is negative definite and $(N_{\Z}-N'_{\Z})-P'_{\Z}=-P_{\Z}$ is nef over $N_{\Z}-N'_{\Z}$, which contradicts the condition (i) on $P'_{\Z}$.
Hence we have $P_{\Z}=P'_{\Z}$ and $N_{\Z}=N'_{\Z}$.
\end{proof}

\begin{remark}
(1) Taking (integral) Zariski decompositions defines self-maps $P_{R}$, $N_{R}$ on $N^{1}(X/Y)$ with $P_{R}^{2}=P_{R}$, $N_{R}^{2}=N_{R}$ and $P_{R}+N_{R}=\mathrm{id}_{N^{1}(X/Y)}$, where $R=\R$ or $\Z$.
The maps $P_{\R}$ and $N_{\R}$ are continuous, but $P_{\Z}$ and $N_{\Z}$ are not continuous in general.

\smallskip

\noindent
(2) Given a property $\mathcal{P}$ of negative definite $f$-exceptional irreducible curves on $X$, we can consider the (resp.\ integral) Zariski decomposition $D=P+N$ (resp.\ $D=P_{\Z}+N_{\Z}$) with the additional condition that each irreducible component of $N$ (resp.\ $N_{\Z}$) has the property $\mathcal{P}$.
Indeed, all the arguments in this subsection work by replacing $\Enu_{R}(D)$ ($R=\Z$ or $\R$) by the subset consisting of $B\in \Enu_{R}(D)$ any component of which has the property $\mathcal{P}$.
\end{remark}

\subsection{$\Z$-positive divisors}

\begin{definition}
Let $A$ and $B$ be $\R$-divisors on $X$ such that $A-B$ is an effective $f$-exceptional $\Z$-divisor.
Then the sequence of divisors $B=D_0<D_1<\cdots<D_m=A$ is called a {\em connecting chain from $B$ to $A$} if $C_i:=D_i-D_{i-1}$ is an integral curve and $D_{i-1}C_i>0$ for any $i=1,\ldots,m$.
Note that $B=A$ is regarded as a connecting chain from $B$ to $A$ ($m=0$ case).
\end{definition}

The following is a characterization of $f$-$\Z$-positive divisors.

\begin{proposition} \label{connchain}
Let $D$ be an $f$-pseudo-effective $\R$-divisor on $X$ with the Zariski decomposition $D=P+N$.
Then the following are equivalent.

\smallskip

\noindent
$(1)$ $D$ is $f$-$\Z$-positive.

\smallskip

\noindent
$(2)$ For any $D-\llcorner N \lrcorner \le D_0\le D$ with $D-D_0$ integral, there exists a connecting chain from $D_0$ to $D$.

\smallskip

\noindent
$(3)$ There exists a connecting chain from $D-\llcorner N \lrcorner$ to $D$.
\end{proposition}

\begin{proof}
Assume that $D$ is $f$-$\Z$-positive.
Let $D_0$ be an $\R$-divisor on $X$ as in (2) and we may assume that $D-D_0>0$.
In order to prove (2), it suffices to show by induction on the number of prime components of $D-D_0$ (counting multiplicity) that there exists an irreducible subcurve $C\le D-D_0$ such that $D_0C>0$ holds.
Since $-D_0=(D-D_0)-D$ is not nef over $D-D_0$ by the $f$-$\Z$-positivity of $D$, the claim follows.
The condition (2) trivially implies (3).
We assume that $D$ satisfies (3) and fix a connecting chain $D-\llcorner N \lrcorner=D_0<D_1<\cdots<D_m=D$.
We will show that $D$ is $f$-$\Z$-positive by induction on $m$.
We first assume $m=0$, that is, $\llcorner N \lrcorner=0$.
Suppose that there is a non-zero effective $f$-exceptional $\Z$-divisor $B$ on $X$ such that $B-D$ is nef over $B$.
From Lemma~\ref{easy}, we have $B\le N$, which contradicts $\llcorner N \lrcorner=0$.
Hence $D$ is $f$-$\Z$-positive.
Assume that $m>0$ and the claim holds when the length of the connecting chain is less than $m$.
In particular, $D_{m-1}$ is $f$-$\Z$-positive.
Assume that there exists a non-zero effective divisor $B$ on $X$ such that $B-D$ is nef over $B$.
Then $B$ is contained in $N$ by Lemma~\ref{easy}.
It is easy to see that $C_m:=D-D_{m-1}$ is not contained in $B$ since $D_{m-1}$ is $f$-$\Z$-positive and $D_{m-1}C_m>0$.
Thus we have $C_mB\ge 0$.
Then $B-D_{m-1}=B-D+C_m$ is nef over $B$, which contradicts that $D_{m-1}$ is $f$-$\Z$-positive.
Hence $D$ is $f$-$\Z$-positive, whence (1) holds.
\end{proof}

\begin{remark}\label{zartointzar}
Proposition~\ref{connchain} is useful for taking the integral Zariski decomposition of a given $f$-pseudo-effective divisor $D$ from the usual Zariski decomposition $D=P+N$.
Indeed, we put $D_0:=D-\llcorner N \lrcorner$  and if $-D_0$ is nef over $N_0:=\llcorner N \lrcorner$, then the decomposition $D=D_0+N_0$ is nothing but the integral Zariski decomposition.
Otherwise, there exists a prime divisor $C_1\le N_0$ such that $D_0C_1>0$.
Then we put $D_1:=D_0+C_1$.
If $-D_1$ is nef over $N_1:=\llcorner N \lrcorner-C$, then the decomposition $D=D_1+N_1$ becomes the integral Zariski decomposition by Proposition~\ref{connchain}.
Repeating this process, we finally obtain the connecting chain $D_0<D_1<\cdots<D_m$ and the integral Zariski decomposition $D=D_m+N_m$ from Proposition~\ref{connchain}.
Such a construction of divisors can be found in many places of the theory of surfaces (e.g., \cite{Miy}, \cite{Sak2}, \cite{KaMa}, \cite{Kon}, \cite{YZZ}) and $2$-dimensional singularities (e.g., \cite{Lau}, Section 7 in \cite{Ish}).
\end{remark}

\begin{corollary}\label{roundup}
Let $D=M+Z$ be an $\R$-divisor on $X$ such that $M$ is an $f$-nef $\R$-divisor and $Z$ is an $f$-exceptional $\R$-divisor with $\llcorner Z \lrcorner= 0$.
Then $D$  is $f$-$\Z$-positive.
In particular, $D=\ulcorner M \urcorner $ is $f$-$\Z$-positive for any $f$-nef $\R$-divisor $M$ on $X$.
\end{corollary}

\begin{proof}
Note that $D$ is $f$-pseudo-effective by Lemma~\ref{inclusion}~(2).
Let $D=P+N$ be the Zariski decomposition of $D$ in Theorem~\ref{Zardecomp}.
Since $M$ and $-P$ is nef over $N$, we have $N\le Z$ by Lemma~\ref{easy}.
In particular, we have $\llcorner N \lrcorner=0$.
Hence $D$ is $f$-$\Z$-positive by Proposition~\ref{connchain}.
\end{proof}

The following property is important.

\begin{proposition} \label{pullback}
Let $\pi\colon X'\to X$ be a proper birational morphism between normal surfaces.
Let $D$ be an $f$-$\Z$-positive $\R$-divisor on $X$ and $Z$ a $\pi$-exceptional $\R$-divisor  on $X$ with $\llcorner Z \lrcorner=0$.
Then $\pi^{*}D+Z$ is $(f\circ\pi)$-$\Z$-positive.
In particular, $\ulcorner \pi^{*}D \urcorner$ is $(f\circ \pi)$-$\Z$-positive for any $f$-$\Z$-positive $\Z$-divisor $D$ on $X$.
\end{proposition}

\begin{proof}
We put $D':=\pi^{*}D+Z$.
Suppose that there exists a negative definite $(f\circ \pi)$-exceptional divisor $B'>0$ on $X'$ such that $B'-D'$ is nef over $B'$.
Let us denote $B'=\pi^{*}B+B_{\pi}$ for some $\pi$-exceptional $\Q$-divisor $B_{\pi}$ and $B:=\pi_{*}B'\ge 0$.
Let $D=P+N$ be the Zariski decomposition. Then $D'=\pi^{*}P+(\pi^{*}N+Z)$ gives the Zariski decomposition of $D'$ since $Z$ is effective and $\pi$-exceptional.
Since $\pi^{*}P$ and $B'-D'$ is nef over $B'$, we have $B'\le \pi^{*}N+Z$ by Lemma~\ref{easy}.
Taking $\pi_{*}$, we obtain $B\le N$.
In particular, $B=0$ or $B>0$ is negative definite and $f$-exceptional.
We write $B_{\pi}-Z=G^{+}-G^{-}$, where $G^{+}$ and $G^{-}$ are effective $\pi$-exceptional $\R$-divisors having no common components.
Clearly we have $\Supp(G^{+})\subset \Supp(B')$.
If $G^{+}>0$, then we can take an integral curve $C$ in the support of $G^{+}$ such that $G^{+}C<0$ since $G^{+}$ is negative definite.
Hence we have $C\le B'$ and 
$$
(B'-D')C=(\pi^{*}(B-D)+B_{\pi}-Z)C=(G^{+}-G^{-})C<0,
$$
which contradicts the nefness of $B'-D'$ over $B'$.
Thus $G^{+}=0$ holds.
Then we have $B>0$ by the assumption of $\llcorner Z \lrcorner=0$.
Since $D$ is $f$-$\Z$-positive, there exists an irreducible curve $C$ in $B$ such that $(B-D)C<0$.
Let $\widehat{C}$ be the proper transform of $C$ on $X'$.
Then $\widehat{C}$ is contained in $B'$ and 
$$
(B'-D')\widehat{C}=(B-D)C-G^{-}\widehat{C}<0,
$$
which contradicts the nefness of $B'-D'$ over $B'$.
Hence $D'$ is $(f\circ \pi)$-$\Z$-positive.
\end{proof}

\begin{definition}[cf.\ \cite{Ram}, \cite{Kon}, \cite{Miy}]
Let $D$ be an effective divisor on a normal complete surface $X$.
Then $D$ is called {\em numerically connected}
(resp.\ {\em chain-connected} or series-connected)
if $D_1D_2>0$ (resp.\ $-D_1$ is not nef over $D_2$) holds for any decomposition $D=D_1+D_2$ with $D_1, D_2>0$.
Clearly, numerically connected implies chain-connected.
\end{definition}

\begin{proposition}
Let $D$ be a chain-connented divisor on a normal complete surface $X$, which is not negative definite.
Then $D$ is $\Z$-positive.
\end{proposition}

\begin{proof}
Following the proof of Proposition~1.2 in \cite{Kon}, $D$ is chain-connected if and only if for any subdivisor $0<D_0\le D$, there exists a connecting chain from $D_0$ to $D$.
Let $D=P+N$ be the Zariski decomposition.
Then we have $P> 0$ since $D$ is not negative definite.
Thus $D-\llcorner N \lrcorner> 0$ and so the claim follows from Proposition~\ref{connchain}.
\end{proof}

\begin{remark}
(1) Ramanujam's connectedness lemma implies that any nef and big effective divisor is numerically connected (cf.\ Lemma~2 in \cite{Ram}, p.\ 242 in \cite{KaMa}).
In summary, the following implication holds for effective and not negative definite divisors:
$$
\text{nef and big}\  \Longrightarrow \  \text{numerically connected} \  \Longrightarrow \  \text{chain-connected}  \ \Longrightarrow  \ \text{$\Z$-positive}.
$$

\smallskip

\noindent
(2) Any chain-connected divisor $D$ is connected since $H^{0}(D, \O_{D})$ is a field (cf.\ \cite{Miy} Corollary~3.6).
On the other hand, $\Z$-positive effective divisors are not necessarily connected.
For example, a finite sum of fibers of a fibration over a curve is $\Z$-positive.
However, if further assume that $D$ is big, it is connected since $H^{1}(X,\O_X(-D))=0$ holds by Theorem~\ref{van} proved in the next section.
\end{remark}


\section{Vanishing theorem on a normal surface}
\label{sec:Vanishing theorem on a normal surface}

Let $f\colon X\to Y$ be a proper surjective morphism from a normal surface $X$ to a variety $Y$.
In this section, we prove the following vanishing theorem.
\begin{theorem}[Vanishing theorem] \label{van}
Let $D$ be an $f$-big divisor on $X$ and $D=P_{\Z}+N_{\Z}$ the integral Zariski decomposition in Theorem~\ref{IntZardecomp}.
Let $\mathcal{L}_{D}$ and $\mathcal{L}'_{D}$ respectively be the rank $1$ sheaves on $N_{\Z}$ defined by the cokernel of the homomorphisms $\O_{X}(K_X+P_{\Z})\to \O_{X}(K_X+D)$ and $\O_{X}(-D)\to \O_{X}(-P_{\Z})$ 
induced by multiplying a defining section of $N_{\Z}$.

\smallskip

\noindent
$(1)$ Assume that $\dim(Y)=0$ and the base field $k$ is of characteristic $0$ if we consider the algebraic setting.
Then we have
$$
H^{1}(X,\O_{X}(K_{X}+D))\cong H^{1}(N_{\Z},\mathcal{L}_{D}),
$$
and
$$
H^{1}(X,\O_{X}(-D))\cong H^{0}(N_{\Z},\mathcal{L}'_{D}).
$$

\smallskip

\noindent
$(2)$ Assume that $\dim(Y)\ge 1$. 
Then we have
$$
R^{1}f_{*}\O_{X}(K_X+D)\cong R^{1}f_{*}\mathcal{L}_{D}.
$$
In particular, we have
$$
\mathrm{length}R^{1}f_{*}\O_{X}(K_X+D)=\dim H^{1}(N_{\Z},\mathcal{L}_{D}).
$$
\end{theorem}

\begin{remark}
In the situation of Theorem~\ref{van}, $R^{i}f_{*}\O_X(K_X+D)=0$ automatically holds for any $i\ge 2$ since $D$ is $f$-big.
\end{remark}

\begin{remark}
Theorem~\ref{van}~(1) is a generalization of the Kawamata-Viehweg vanishing \cite{Kaw} \cite{Vie}, Ramanujam's $1$-connected vanishing \cite{Ram}, Miyaoka's vanishing \cite{Miy} on surfaces.
Langer's vanishing (\cite{Lan} Theorem~3.2 and its remarks) is essentially the same as Theorem~\ref{van} (1) in the case that $X$ is projective over $\C$ and $D$ is $\Z$-positive, which was proved by using a log version of Reider's method \cite{Rei}.
Our proof is more elementary.
Theorem~\ref{van}~(2) is a generalization of the local vanishing theorem due to Ramanujam (Lemma~4 in \cite{Ram}), Sakai (Theorem~(2.2) in \cite{Sak}) and Koll\'ar-Kov\'acs (2.2.5~Corollary in \cite{KoKo}).
\end{remark}

\subsection{Proof of Theorem~\ref{van} (2)}

First, we assume that $\dim(Y)\ge 1$.
We follow Sakai's argument in \cite{Sak} in this subsection.

\begin{lemma}[(Semi-)local vanishing on a regular surface. cf.\ \cite{Sak}, \cite{KoKo}]\label{locvanreg}
Assume that $X$ is regular.
Let $D$ be an $f$-big $f$-$\Z$-positive divisor on $X$.
Then $R^{1}f_{*}\O_{X}(K_X+D)=0$ holds.
\end{lemma}

\begin{proof}
In order to prove the claim, it suffices to show the completion $R^{1}f_{*}\O_{X}(K_X+D)^{\widehat{}}_{y}$ is $0$ for any closed point $y\in Y$.
By the formal function theorem, it suffices to show that for any $f$-exceptional effective divisor $B$ on $X$, the vanishing $H^{1}(B,\O_{B}(K_X+D))=0$ holds, which is equivalent to $H^{0}(B,\O_{B}(B-D))=0$ by the Serre duality.
We will show this by the induction on the number of prime components (counting multiplicity) of $B$.
Note that $B$ is negative definite when $\dim(Y)=2$, or negative semi-definite when $\dim(Y)=1$ by Zariski's lemma.
If $B$ is a prime divisor, then it follows that $(B-D)B<0$ since $D$ is $f$-big and $f$-$\Z$-positive (note that when $\dim(Y)=1$, the $f$-bigness of $D$ implies that $B-D$ is not nef over $B$ if $\Supp(B)$ contains the support of a fiber $F$ of $f$ since $DF>0$, see Corollary~\ref{fib}).
Hence the claim follows.
We assume that $B$ is not prime.
Then there is a prime component $C$ of $B$ such that $(B-D)C<0$ since $D$ is $f$-big and $f$-$\Z$-positive.
Taking $H^{0}$ of the exact sequence
$$
0\to \O_{B-C}(B-C-D)\to \O_{B}(B-D) \to \O_{C}(B-D)\to 0,
$$
we have $H^{0}(B,\O_{B}(B-D))=0$ by the inductive assumption.
Hence the claim holds.
\end{proof}

\begin{lemma}[Projection formula. cf.\ \cite{Sak}]\label{proj}
Let $\pi \colon X'\to X$ be a proper birational morphism between normal surfaces.
Let $D$ be an $\R$-divisor on $X$ and $Z$ a $\pi$-exceptional effective $\R$-divisor on $X'$.
Then $\pi_{*}\O_{X'}(\llcorner \pi^{*}D+Z \lrcorner)\cong \O_{X}(\llcorner D \lrcorner)$ holds.
\end{lemma}

\begin{proof}
We prove this in the algebraic setting. 
In the analytic case, the proof is similar and proved in \cite{Sak}.
Since the claim is local on $X$, we may assume that $X$ is affine and contains only one point $x$ in the center of $\pi$.
Let $E:=\pi^{-1}(x)$ be the $\pi$-exceptional set on $X'$.
Then we have the following exact sequence
$$
0\to H^{0}(X', \O_{X'}(\llcorner \pi^{*}D+Z\lrcorner))\to H^{0}(X'\setminus E, \O_{X'}(\llcorner \pi^{*}D+Z\lrcorner))\to H_{E}^{1}(\O_{X'}(\llcorner \pi^{*}D+Z\lrcorner)).
$$

Because $X$ is normal, we have natural isomorphisms
$$
H^{0}(X'\setminus E, \O_{X'}(\llcorner \pi^{*}D+Z\lrcorner))\cong  H^{0}(X\setminus \{x\}, \O_{X}(\llcorner D\lrcorner))\cong  H^{0}(X, \O_{X}(\llcorner D\lrcorner))
$$
and then obtain the natural injection $H^{0}(X', \O_{X'}(\llcorner \pi^{*}D+Z\lrcorner))\to H^{0}(X, \O_{X}(\llcorner D\lrcorner))$.
One can check easily that this is an isomorphism.
Indeed, since $\pi^{*}\mathrm{div}(\varphi)=\mathrm{div}(\pi^{*}\varphi)$ for any rational function $\varphi$ on $X$ and the Mumford pull-back $\pi^{*}$ preserves the effectivity of divisors, then $\mathrm{div}(\varphi)+D\ge 0$ implies $\mathrm{div}(\pi^{*}\varphi)+\llcorner \pi^{*}D+Z\lrcorner \ge 0$ for any rational function $\varphi$ on $X$.
\end{proof}

\begin{proposition}[(Semi-)local vanishing on a normal surface. cf.\ \cite{Sak}]\label{locvan}
Let $D$ be an $f$-big $f$-$\Z$-positive divisor on the normal surface $X$.
Then $R^{1}f_{*}\O_{X}(K_X+D)=0$ holds.
\end{proposition}

\begin{proof}
Let $\pi\colon X'\to X$ be a resolution.
By using the Leray spectral sequence 
$$
E_2^{p,q}=R^{p}f_{*}R^{q}\pi_{*}\O_{X'}(K_{X'}+\ulcorner \pi^{*}D\urcorner)\Rightarrow E^{m}=R^{m}(f\circ \pi)_{*}\O_{X'}(K_{X'}+\ulcorner \pi^{*}D\urcorner)
$$
and $E^1=0$ by Proposition~\ref{pullback} and Lemma~\ref{locvanreg}, we have $R^{1}f_{*}\pi_{*}\O_{X'}(K_{X'}+\ulcorner \pi^{*}D\urcorner)=0$.
Let $Z$ be an effective $\pi$-exceptional divisor on $X'$.
There is an exact sequence
$$
0\to \pi_{*}\O_{X'}(K_{X'}+\ulcorner \pi^{*}D\urcorner)\to \pi_{*}\O_{X'}(K_{X'}+\ulcorner \pi^{*}D\urcorner+Z)\to \mathcal{T}\to 0,
$$
where $\mathcal{T}$ is a torsion sheaf on $X$ whose support is of dimension $0$.
Taking $R^{1}f_{*}$, we have $R^{1}f_{*}\pi_{*}\O_{X'}(K_{X'}+\ulcorner \pi^{*}D\urcorner+Z)=0$.
Let us write $K_{X'}=\pi^{*}K_X-\Delta$, where $\Delta$ is a $\pi$-exceptional $\Q$-divisor on $X'$.
By taking $Z$ sufficiently effective, we have
$$
\pi_{*}\O_{X'}(K_{X'}+\ulcorner \pi^{*}D\urcorner+Z)=\pi_{*}\O_{X'}(\llcorner \pi^{*}(K_X+D)+Z'\lrcorner)\cong \O_{X}(K_X+D)
$$
by Lemma~\ref{proj}, where $Z':=Z-\Delta+\ulcorner \pi^{*}D\urcorner-\pi^{*}D+\{\pi^{*}(K_X+D)\}$ is an effective $\pi$-exceptional divisor.
Hence $R^{1}f_{*}\O_{X}(K_X+D)=0$ follows.
\end{proof}

\begin{proof}[Proof of Theorem~\ref{van} (2)]
Let $D=P_{\Z}+N_{\Z}$ be the integral Zariski decomposition as in Theorem~\ref{IntZardecomp}.
Then we have an exact sequence
$$
0\to \O_{X}(K_X+P_{\Z})\to \O_{X}(K_X+D)\to \mathcal{L}_D\to 0.
$$
Since $R^{i}f_{*}\O_{X}(K_X+P_{\Z})=0$ for $i\ge 1$ by Proposition~\ref{locvan}, we get
$R^{1}f_{*}\O_{X}(K_X+D)\cong R^{1}f_{*}\mathcal{L}_D$.
\end{proof}

\subsection{Proof of Theorem~\ref{van} (1)}

Next, we assume that $\dim(Y)=0$ and $X$ is a normal proper surface over a field of characteristic $0$.
The following lemma can be proved similarly to Proposition~\ref{locvan} by using Lemma~\ref{proj}.

\begin{lemma}\label{reduction}
Let $\pi\colon X'\to X$ be a proper birational morphism between normal complete surfaces.
Let $M$ be an $\R$-divisor on $X$.
If $H^{1}(X',\O_{X'}(K_{X'}+\ulcorner \pi^{*}M \urcorner))=0$, then $H^{1}(X,\O_{X}(K_X+\ulcorner M\urcorner))=0$ holds.
\end{lemma}

Sakai proved in \cite{Sak} the following vanishing theorem by using the usual Kawamata-Viehweg vanishing theorem and Lemma~\ref{reduction}.

\begin{proposition}[\cite{Sak} Theorem~5.1]\label{KVvan}
Let $M$ be a nef and big $\R$-divisor on $X$.
Then we have $H^{1}(X,\O_{X}(K_X+\ulcorner M \urcorner))=0$.
\end{proposition}

\begin{proof}
Taking a resolution and using Lemma~\ref{reduction}, we may assume that $X$ is smooth.
Shifting the coefficients of prime divisors in $M$, we may assume that $M$ is ample.
Moreover, taking the base change to the algebraic closure of the base field, we may assume that the base field is algebraically closed 
(note that taking the fractional part of $M$ commutes with the base change because all varieties are necessarily geometrically reduced in characteristic $0$).
By taking a log resolution $\pi\colon X'\to X$ of $(X,\{M\})$ and using Lemma~\ref{reduction} again, the assertion follows from the Kawamata-Viehweg vanishing theorem.
\end{proof}

The following corollary in the smooth surface case is a special version of the Miyaoka vanishing theorem in \cite{Miy}.

\begin{corollary}[Miyaoka vanishing theorem on a normal surface, cf.\ \cite{Miy}]\label{Miyvan}
Let $D$ be a big divisor on $X$ with the Zariski decomposition $D=P+N$.
If $\llcorner N \lrcorner=0$, then $H^{1}(X,\O_{X}(K_X+D))=0$.
\end{corollary}

\begin{proof}
Since $P=D-N$ is a nef and big $\Q$-divisor on $X$ and $\ulcorner P\urcorner=D-\llcorner N\lrcorner=D$, the claim holds from Proposition~\ref{KVvan}.
\end{proof}

By using Corollary~\ref{Miyvan}, we will prove the big $\Z$-positive vanishing, the smooth case of which is essentially the same as Theorem~2.7 in \cite{Miy}.
\begin{proposition}\label{bigZposvan}
Let $D$ be a big $\Z$-positive divisor on $X$.
Then we have $H^{1}(X,\O_X(K_X+D))=0$.
\end{proposition}

\begin{proof}
First, we assume that $X$ is smooth.
Let $D=P+N$ be the Zariski decomposition.
Since $D$ is $\Z$-positive, there exists a connecting chain $D-\ulcorner N\urcorner=:D_0<D_1<\cdots<D_m:=D$ by Proposition~\ref{connchain}.
Putting $C_i:=D_i-D_{i-1}$, we have $D_{i-1}C_i>0$.
From Corollary~\ref{Miyvan}, we have $H^{1}(X,\O_{X}(K_X+D_0))=0$.
By taking the cohomologies of the exact sequences
$$
0\to \O_{X}(K_X+D_{i-1})\to \O_{X}(K_X+D_i)\to \O_{C_i}(K_X+D_i)\to 0
$$
and using the vanishing $H^{1}(C_i,\O_{C_i}(K_X+D_i))\cong H^{0}(C_i,\O_{C_i}(-D_{i-1}))^{*}=0$,
the claim holds by the induction on $m$.

Next, we will prove the claim in general.
Take a resolution $\pi\colon X'\to X$.
Then $\ulcorner \pi^{*}D\urcorner$ is big and $\Z$-positive by proposition~\ref{pullback}.
Thus we have $H^{1}(X',\O_{X'}(K_{X'}+\ulcorner \pi^{*}D\urcorner))=0$.
By using Lemma~\ref{reduction}, we get $H^{1}(X,\O_X(K_X+D))=0$.
\end{proof}

\begin{proof}[Proof of Theorem~\ref{van} (1)]
Let $D=P_{\Z}+N_{\Z}$ be the integral Zariski decomposition.
Then $H^{1}(X,\O_{X}(K_X+D))\cong H^{1}(N_{\Z},\mathcal{L}_D)$ holds by using the exact sequence
$$
0\to \O_{X}(K_X+P_{\Z})\to \O_X(K_X+D)\to \mathcal{L}_D\to 0
$$
and $H^{1}(X,\O_{X}(K_X+P_{\Z}))=0$ by Proposition~\ref{bigZposvan}.
The dual case $H^{1}(X,\O_X(-D))\cong H^{0}(N_{\Z},\mathcal{L}'_{D})$ can be proved similarly.
\end{proof}


\section{Reider-type theorems}
\label{sec:Reider-type theorems}

In this section, we apply Theorem~\ref{van} to the criterion of higher order separations of adjoint (relative) linear systems.
Let $f\colon X\to Y$ be a proper surjective morphism from a normal surface $X$ to a variety $Y$.
If $\dim(Y)=0$, we assume that the base field $k$ is of characteristic $0$ if we consider the algebraic setting.

\subsection{The main theorem}


Let $\zeta\subset X$ be a {\em cluster}, that is, a subscheme (or an analytic subset) of dimension $0$.
First, we define invariants $\delta_{\zeta}$, $q_{\zeta}$ and $\delta'_{\zeta}$ of the pair $(X,\zeta)$.

\begin{definition}[Invariants of a cluster] \label{cluster}
We consider resolutions $\pi\colon X'\to X$ of the singularities contained in $\zeta$ and $\pi$-exceptional divisors $Z>0$ satisfying $\pi_{*}\O_{X'}(-Z)\subset \mathcal{I_{\zeta}}$, where $\mathcal{I_{\zeta}}$ is the sheaf of ideals corresponding to $\zeta$.
For such $\pi$ and $Z$, we define $\delta_{\zeta}(\pi,Z)$ as the number $-(\Delta-Z)^2$ if $\Delta-Z$ is not effective, or $0$ otherwise, where $\Delta:=\pi^{*}K_X-K_{X'}$ is the anti-canonical cycle of $\pi$.
The number $\delta_{\zeta}$ is defined by the minimum $\delta_{\zeta}(\pi,Z)$ among all such $\pi$ and $Z$.
Note that the minimum always exists since $Z$ is integral and there is a sufficiently divisible integer $s$ such that $s\Delta$ is integral for any resolution $\pi$.
By definition, we can write $\delta_{\zeta}=\sum_{x}\delta_{\zeta_x}$, where $\zeta_x$ is the subcluster of $\zeta$ supported at the point $x$.

Let us define $q_{\zeta}$ as 
$$
q_{\zeta}:=\mathrm{min}\{E^2, \delta_{\zeta}/4 \ |\ \text{$E$ is an effective divisor on $X$ with $E\cap \zeta\neq \emptyset$ and $E^2>0$}\}
$$
and put $\delta'_{\zeta}:=q_{\zeta}(\delta_{\zeta}/4q_{\zeta}+1)^2$.
Clearly we have $\delta'_{\zeta}\ge \delta_{\zeta}$ with the equality holding if and only if $q_{\zeta}=\delta_{\zeta}/4$.
\end{definition}

The main theorem in this section is the following.

\begin{theorem}[Reider-type theorem I]\label{sep}
Let $D$ be an $f$-big divisor on $X$ and $D=P+N$ $($resp.\ $D=P_{\Z}+N_{\Z}$$)$ the Zariski decomposition $($resp.\ the integral Zariski decomposition$)$ of $D$.
Let $\zeta$ be a cluster along which $K_X+D$ is Cartier.
When $\dim(Y)=0$, we further assume that $P^2>\delta_{\zeta}$ $($resp.\ $P_{\Z}^{2}>\delta_{\zeta}$$)$.
Then the natural map $f_{*}\O_{X}(K_X+D)\to f_{*}(\O_{X}(K_X+D)|_{\zeta})$ is surjective, or
 there exists an $f$-exceptional divisor $B>0$ on $X$ intersecting $\zeta$ such that $(D-B)B\le \delta_{\zeta}/4$ $($resp.\ and $D+N_{\Z}-2B$ is big$)$.
\end{theorem}

\begin{proof}

Assume that $f_{*}\O_{X}(K_X+D)\to f_{*}(\O_{X}(K_X+D)|_{\zeta})$ is not surjective.
Then we have $R^{1}f_{*}\mathcal{I}_{\zeta}\O_{X}(K_X+D)\neq 0$.

First, we consider the case that $D$ is $f$-$\Z$-positive.
We take a resolution $\pi\colon X'\to X$ of the singularities contained in $\zeta$
and a $\pi$-exceptional divisor $Z>0$ on $X'$ such that $\pi_{*}\O_{X'}(-Z)\subset \mathcal{I_{\zeta}}$ and $\delta_{\zeta}=\delta_{\zeta}(\pi,Z)$.
Then $\pi_{*}\O_{X'}(\pi^{*}(K_X+D)-Z)$ is a subsheaf of $\mathcal{I}_{\zeta}\O_X(K_X+D)$ whose cokernel is supported on $\zeta$. 
Hence we have $R^{1}f_{*}\pi_{*}\O_{X'}(\pi^{*}(K_X+D)-Z)\neq 0$.
By using the Leray spectral sequence 
$$
E_2^{p,q}=R^{p}f_{*}R^{q}\pi_{*}\O_{X'}(\pi^{*}(K_X+D)-Z) \Rightarrow E^m=R^{m}f'_{*}\O_{X'}(\pi^{*}(K_X+D)-Z),
$$
we have $R^{1}f'_{*}\O_{X'}(K_{X'}+D')\neq 0$, where we put 
$D':=\pi^{*}D+\Delta-Z$ and $f':=f\circ \pi$.
Note that $\Delta-Z$ is not effective in this case.
Indeed, if $Z\le \Delta$, then we may assume that $D'=\ulcorner \pi^{*}D\urcorner$ by replacing $Z$ with $\llcorner \Delta \lrcorner$, which is $f'$-big and $f'$-$\Z$-positive by Proposition~\ref{roundup}.
By Theorem~\ref{van}, it is a contradiction to the non-vanishing of $R^{1}f'_{*}\O_{X'}(K_{X'}+D')$.
When $\dim(Y)=0$, $D'$ is $f'$-big since $\pi^{*}P+\Delta-Z$ is $f'$-big from the assumption $P^2>\delta_{\zeta}$ and Lemma~\ref{positive}.
When $\dim(Y)\ge 1$, $D'$ is always $f'$-big from the $f$-bigness of $D$.
Hence the non-vanishing of $R^{1}f'_{*}\O_{X'}(K_{X'}+D')$ implies that $D'$ is not $f'$-$\Z$-positive by Theorem~\ref{van}.
Let $B'>0$ be the $\Z$-negative part of $D'$ as in Theorem~\ref{IntZardecomp} and put $B:=\pi_{*}B'$.
Then $B$ is non-zero.
Indeed, if $B'$ is $\pi$-exceptional, then we have $R^{1}f'_{*}\O_{X'}(K_{X'}+D'-B')\neq 0$
by the same argument as above after replacing $Z$ with $Z+B'$, which is a contradiction to Theorem~\ref{van}.
Let us write $B'=\pi^{*}B+B_{\pi}$ for some $\pi$-exceptional $\Q$-divisor $B_{\pi}$ on $X'$.
Then we have $0\le (B'-D')B'=(B-D)B+(B_{\pi}-\Delta+Z)B_{\pi}$.
Hence
\begin{align*}
(D-B)B&\le (B_{\pi}-\Delta+Z)B_{\pi}\\
&=\left(B_{\pi}-\frac{1}{2}(\Delta-Z)\right)^2-\frac{1}{4}(\Delta-Z)^2\\
&\le \frac{1}{4}\delta_{\zeta}.
\end{align*}
We show that $B$ intersects $\zeta$.
Suppose that $B\cap \zeta =\emptyset$.
Then for any subcurve $C\le B$, the proper transform $\widehat{C}\le B'$ of $C$ equals the Mumford pull-back $\pi^{*}C$.
Thus $B$ is negative definite and we have
$$
0\le (B'-D')\widehat{C}=(\pi^{*}(B-D)+B_{\pi}-\Delta+Z)\pi^{*}C=(B-D)C.
$$
Hence $B-D$ is nef over $B$, which contradicts the $f$-$\Z$-positivity of $D$.
When $\dim(Y)=0$ and $D^2>\delta_{\zeta}$, we will show that $D-2B$ is $f$-big.
For this, it is enough to show that $D'-2B'$ is big since $D-2B=\pi_{*}(D'-2B')$ and Lemma~\ref{pushpull} (1).
This follows from Corollary~\ref{reflection}.

For any $f$-big divisor $D$ which is not $f$-$\Z$-positive, we consider the integral Zariski decomposition $D=P_{\Z}+N_{\Z}$ as in Theorem~\ref{IntZardecomp}.
If $N_{\Z}$ intersects $\zeta$, then $B:=N_{\Z}$ satisfies the properties in Theorem~\ref{sep}.
Then we may assume that $N_{\Z}$ and $\zeta$ are disjoint.
Hence $f_{*}\O_{X}(K_X+P_{\Z})\to f_{*}(\O_{X}(K_X+P_{\Z})|_{\zeta})$ is also not surjective.
As shown in the first half of the proof, we can take an $f$-exceptional curve $B$ on $X$ intersecting $\zeta$ such that 
$(P_{\Z}-B)B\le \delta_{\zeta}/4$.
If $N_{\Z}B\le 0$, then we have $(D-B)B\le (P_{\Z}-B)B\le \delta_{\zeta}/4$.
Assume that $N_{\Z}B>0$.
Putting $\overline{B}:=B+N_{\Z}$, we have
$$
(D-\overline{B})\overline{B}=(P_{\Z}-B)(B+N_{\Z})=(P_{\Z}-B)B+P_{\Z}N_{\Z}-BN_{\Z}<(P_{\Z}-B)B\le \delta_{\zeta}/4.
$$
Moreover, $D+N_{\Z}-2\overline{B}$ is big by the bigness of $P_{\Z}-2B$.
Replacing $\overline{B}$ to $B$, we complete the proof.
\end{proof}

\begin{remark}
Reider's original proof \cite{Rei} uses the vector bundle technique, especially, Serre's construction and  the Bogomolov inequality.
Many authors generalize Reider's theorem to singular surfaces (e.g., \cite{Sak2}, \cite{EiLa}, \cite{KaMa}, \cite{Lan}, \cite{YZZ}).
There are mainly two ways of the proof.
The first one is to use the vector bundle technique along the original one and the second one is to use the Kodaira-type vanishing theorem.
Our proof of Theorem~\ref{sep} belongs to the latter one.
Moreover, Theorem~\ref{sep} contains a log version ($D=\ulcorner M \urcorner$, a nef and big $\R$-divisor $M$) and a relative version of Reider-type theorems. 
For the case of normal projective surfaces over $\C$, Theorem~\ref{sep} is slightly weaker than Theorem~3.2 in \cite{Lan} because the invariant $\delta_{\zeta}$ in this paper may be greater than the $\delta_{\zeta}$ defined in \cite{Lan}.
\end{remark}

Next, we prove a variant of Reider-type theorems for adjoint linear systems on normal proper surfaces which will be used in Section~6.

\begin{theorem}[Reider-type theorem II] \label{negsemidefnthm}
Let $D$ be a nef and big divisor on a normal complete surface $X$ over a field $k$ of characteristic $0$. 
Let $\zeta$ be a cluster on $X$ along which $K_X+D$ is Cartier.
Assume that $D^2>\delta_{\zeta}$ and $H^{0}(X,\O_{X}(K_X+D))\to H^{0}(X,\O_{X}(K_X+D)|_{\zeta})$ is not surjective.
Then there exists an effective divisor $B>0$ on $X$ intersecting $\zeta$ such that 
$0<(D-B)B\le \delta_{\zeta}/4$ and $D-2B$ is big.
Moreover, if further assume that $D^2>\delta'_{\zeta}$ $($resp.\ $D^2=\delta'_{\zeta}$$)$,
then the curve $B$ becomes negative semi-definite $($resp.\ or satisfies $B^2=q_{\zeta}$ and $D\equiv (\delta_{\zeta}/4q_{\zeta}+1)B$$)$.
\end{theorem}

\begin{proof}
The first half of the claim for $D^2>\delta_{\zeta}$ is a special case of Theorem~\ref{sep}.
In particular, $H^{0}(X,\O_{X}(K_X+D))\to H^{0}(X,\O_{X}(K_X+D)|_{\zeta})$ is always surjective when $\delta_{\zeta}=0$.
Thus, in order to prove the second half of the claim, we may assume that $D^2>\delta_{\zeta}>0$ and $H^{0}(X,\O_{X}(K_X+D))\to H^{0}(X,\O_{X}(K_X+D)|_{\zeta})$ is not surjective.
Then there exists a curve $B$ on $X$ with $B\cap \zeta\neq \emptyset$ such that $0<(D-B)B\le \delta_{\zeta}/4$ and $D-2B$ is big, 
where $(D-B)B>0$ follows from Ramanujam's connectedness lemma (cf.\ \cite{Ram}, \cite{KaMa}).
Note that the function $\phi(x):=x(\delta_{\zeta}/4x+1)^2$ defined on $\R_{>0}$ decreases monotonically in the range $0<x\le \delta_{\zeta}/4$ and takes the minimum value $\phi(\delta_{\zeta}/4)=\delta_{\zeta}$.

First we consider the case that $D^2> \phi(q_{\zeta})=\delta'_{\zeta}$.
We will show that $B$ is negative semi-definite.
First, we assume that $B^2>0$ and put $x:=B^2$.
Since $D$ is nef and $D-2B$ is big, we have $0< D(D-2B)=(D-B)^2-B^2$ and then $0<x< (D-B)^2$.
Thus we have $B^2(D-B)^2\le( B(D-B))^2\le \delta_{\zeta}^2/16$ by the Hodge index theorem.
It follows that $x< \sqrt{B^2(D-B)^2}\le \delta_{\zeta}/4$ and $(D-B)^2\le \delta_{\zeta}^2/16x$.
Hence we have
$$
D^2=(D-B)^2+2(D-B)B+B^2\le \frac{\delta_{\zeta}^2}{16x}+\frac{\delta_{\zeta}}{2}+x=\phi(x).
$$
By the definition of $q_{\zeta}$, we have $q_{\zeta}\le x< \delta_{\zeta}/4$.
Then we have $D^2\le \phi(x)\le \phi(q_{\zeta})$, which contradicts the assumption $D^2>\phi(q_{\zeta})$.
Hence we conclude $B^2\le 0$.
Next, we assume that there exists a subcurve $0<C\le B$ such that $x:=C^2>0$.
Note that $C\cap \zeta \neq \emptyset$ since $B$ is obtained by the push-forward of a negative definite divisor by a resolution $\pi\colon X'\to X$ of singularities in $\zeta$. 
Since $D-2B$ is big, $D-2C$ is also big and thus we have $0<x< (D-C)^2$ similarly as above.
On the other hand, we have
$$
0<C(D-C)<CD\le BD\le B(D-B)\le \frac{\delta_{\zeta}}{4}.
$$
Thus by the Hodge index theorem, we have $C^2(D-C)^2\le( C(D-C))^2< \delta_{\zeta}^2/16$.
It follows that $x< \sqrt{C^2(D-C)^2}< \delta_{\zeta}/4$ and $(D-C)^2< \delta_{\zeta}^2/16x$.
Hence we have $D^2=(D-C)^2+2(D-C)C+C^2< \phi(x)\le \phi(q_{\zeta})$, which is also a contradiction.
Hence $B$ is negative semi-definite.

Next we consider the case that $D^2=\phi(q_{\zeta})$. 
In this case, we have $q_{\zeta}< \delta_{\zeta}/4$ by the assumption $D^2>\delta_{\zeta}=\phi(\delta_{\zeta}/4)$.
If $B^2\le 0$, then $B$ is also negative semi-definite by the same proof as above.
Thus we may assume that $B^2>0$ and put $x:=B^2$.
Then we have $D^2\le \phi(x)\le \phi(q_{\zeta})$ and $q_{\zeta}\le x<\delta_{\zeta}/4$ by the same argument as the case of $D^2>\phi(q_{\zeta})$. 
Thus $D^2=\phi(q_{\zeta})$ and $x=q_{\zeta}$ hold.
In particular, the curve $B$ computes the invariant $q_{\zeta}$.
Moreover, we have $B^2(D-B)^2=(B(D-B))^2=\delta_{\zeta}^2/16$ and then
$D\equiv (\delta_{\zeta}/4q_{\zeta}+1)B$ holds by the Hodge index theorem.
\end{proof}

\begin{remark}
In Theorem~\ref{negsemidefnthm}, we may replace $\delta_{\zeta}$ by $\delta_{\zeta}(\pi,Z)$ for any pair $(\pi,Z)$ as in Definition~\ref{cluster} (see the proof of Theorem~\ref{sep}).
\end{remark}

\begin{remark} \label{subcurve}
When $D^2=\phi(q_{\zeta})>\delta_{\zeta}$, $B^2=q_{\zeta}$ and $D\equiv (\delta_{\zeta}/4q_{\zeta}+1)B$ hold, we can show that every proper subcurve $0<C<B$ which intersects $\zeta$ satisfies $C^2\le 0$ by the definition of $q_{\zeta}$ since $B^2-C^2=B(B-C)+C(B-C)>0$ holds by the nef and bigness of $B$.
\end{remark}

\subsection{Upper bound of $\delta_{\zeta}$} 
Before stating corollaries of Theorem~\ref{sep}, we collect upper bounds of the invariants $\delta_{\zeta}$ when $\zeta$ has small length for the readers' convenience (more details, see \cite{KaMa}, \cite{LanII}).

First, we consider the case that $\zeta=x$ is a closed point of a normal surface $X$.
Then the length of $\zeta$ is the degree of the field extension $[k(x):k]$, where $k(x)$ is the residue field of $\O_{X,x}$.
The following is a non-standard definition.

\begin{definition}
The germ $(X,x)$ is called {\em log terminal} if there exists a resolution $\pi\colon X'\to X$ of $(X,x)$ such that $\llcorner \Delta \lrcorner\le 0$, where $\Delta=\pi^{*}K_X-K_{X'}$ is the anti-canonical cycle of $\pi$.
\end{definition}

The following is well-known to experts.
For the convenience of readers, we will prove this.

\begin{lemma}[\cite{KoKo} Corollary~2.2.11]\label{ltisrat}
Any log terminal germ $(X,x)$ is rational, that is, $R^{1}\pi_{*}\O_{X'}=0$ and $R^{1}\pi_{*}\O_{X'}(K_{X'})=0$ hold for any resolution $\pi\colon X'\to X$ of $(X,x)$.
\end{lemma}

\begin{proof}
$R^{1}\pi_{*}\O_{X'}(K_{X'})=0$ follows from Theorem~\ref{van}.
If $(X,x)$ is regular and $\pi\colon X'\to X$ is a blow-up at $x$, then $\Delta$ is $\pi$-nef.
If $(X,x)$ is a log terminal singularity and $\pi\colon X'\to X$ is a minimal resolution of $(X,x)$,
 then $\Delta=\{\Delta\}$ is $\pi$-$\Z$-positive by Proposition~\ref{connchain}.
Then we have 
$$
R^{1}\pi_{*}\O_{X'}=R^{1}\pi_{*}\O_{X'}(K_{X'}-\pi^{*}K_X+\Delta)=0
$$
for such $(X,x)$ and $\pi$ by Theorem~\ref{van}.
For any resolution $\pi\colon X'\to X$, we also obtain 
$R^{1}\pi_{*}\O_{X'}=0$ by using the Leray spectral sequence $R^{p}\pi''_{*}R^{q}\pi'_{*}\O_{X'}\Rightarrow R^{m}\pi_{*}\O_{X'}$ for a composition $\pi=\pi'\circ \pi''$ inductively.
\end{proof}

\begin{lemma}[cf.\ \cite{KaMa} Theorem~2] \label{deltax}
We have $\delta_x\le 4[k(x):k]$ if $(X,x)$ is regular, $\delta_x\le 2[k(x):k]$ if $(X,x)$ is a log terminal singularity and $\delta_x=0$ otherwise.
\end{lemma}

\begin{proof}
Let $\pi\colon X'\to X$ be the blow-up at $x$ if $(X,x)$ is regular or the minimal resolution of $(X,x)$ otherwise, and $Z$ the fundamental cycle of $\pi$.
When $(X,x)$ is regular, then $\Delta-Z=-2E$, $E$ is the exceptional curve on $X'$ and so $\delta_x\le -4E^2=4[k(x):k]$ holds.
When $(X,x)$ is a log terminal singularity, then we have
\begin{align*}
\delta_x&\le -(\Delta-Z)^2 \\
&=K_{X'}(\Delta-Z)-Z(X_{X'}+Z)\\
&=-K_{X'}(Z-\Delta)-(2p_a(Z)-2[k(x):k])\\
&=-K_{X'}(Z-\Delta)+2[k(x):k]\\
&\le 2[k(x):k],
\end{align*}
where $p_a(Z)=0$ since $Z$ is the fundamental cycle on a resolution of a rational singularity and the last inequality follows from the fact that $K_{X'}$ is $\pi$-nef and $Z-\Delta$ is effective.
The assertion for the not log terminal case is clear.
\end{proof}

\begin{remark}
For simplicity, we consider a smooth projective surface $X$ over an algebraically closed field and a cluster $\zeta$ on $X$ of length $1$ or $2$.
Then we can compute $\delta_{\zeta}=4$ (resp.\ $\delta_{\zeta}=8$) if $\mathrm{length}(\zeta)=1$ (resp.\ $\mathrm{length}(\zeta)=2$) 
(the upper bound of $\delta_{\zeta}$ is obtained by Lemma~\ref{deltax} and the proof of Lemma~\ref{deltabound} in the next section. 
One can see easily that this is actually equal. More precisely, see Section 4 in \cite{Lan}).
On the other hand, we have $q_{\zeta}=1$ (resp.\ $q_{\zeta}=1$ or $2$) if $\mathrm{length}(\zeta)=1$ (resp.\ $\mathrm{length}(\zeta)=2$).
Thus the bound of $D^2$ in Theorem~\ref{negsemidefnthm} is $q_{\zeta}(\delta_{\zeta}/4q_{\zeta}+1)^2=4$ (resp.\ $8$ or $9$) if $\mathrm{length}(\zeta)=1$ (resp.\ $\mathrm{length}(\zeta)=2$).
Theorem~\ref{negsemidefnthm} in this case was mentioned in Theorem~1 and Remark~1.2 in \cite{Rei}.
\end{remark}

\subsection{Corollaries of Theorem~\ref{sep}}

In this subsection, we give some corollaries of Theorem~\ref{sep} for the freeness and the very ampleness of adjoint linear systems.

\begin{corollary}\label{bpfdim0}
Let $X$ be a normal complete surface over a field of characteristic $0$.
Let $D$ be a nef divisor on $X$ and $x\in X$ a closed point at which $K_X+D$ is Cartier.
Assume that 

\smallskip

\noindent
$(\mathrm{i})$ $D^2>\delta_x$, and

\smallskip

\noindent
$(\mathrm{ii})$ $DB\ge \frac{1}{2}\delta_x$ for any curve $B$ on $X$ passing through $x$.

\smallskip

\noindent
Then $x$ is not a base point of $|K_X+D|$.
\end{corollary}

\begin{proof}
If $x$ is a base point of $|K_X+D|$, then there exists a curve $B$ passing through $x$ on $X$ such that $D-2B$ is big from Theorem~\ref{sep}.
Since $D$ is nef, $(D-2B)D\ge 0$ holds.
Hence $DB\le D^{2}/2< \delta_x/2$, which contradicts the assumption.
\end{proof}

In the rest of the section, we assume that the base field is algebraically closed.

\begin{corollary}\label{bpfdim1}
Let $f\colon X\to Y$ be a fiber space from a smooth surface $X$ to a curve $Y$.
Let $D$ be an $f$-nef divisor on $X$ with $DF>0$ for a fiber $F$ of $f$.
Then for any base point $x$ of $f_{*}\O_X(K_X+D)$, there exists an $f$-exceptional curve $B$ on $X$ passing through $x$ such that one of the following holds.

\smallskip

\noindent
$(\mathrm{i})$ $DB=0$ and $B^2=-1$.

\smallskip

\noindent
$(\mathrm{ii})$ $DB=1$ and $B^2=0$.

\smallskip

\noindent
In particular, if one of the following conditions $(\mathrm{a})$ and $(\mathrm{b})$ holds, then
the natural map $f^{*}f_{*}\O_X(K_X+D)\to \O_X(K_X+D)$ is surjective.

\smallskip

\noindent
$(\mathrm{a})$ $D-mH$ is $f$-nef, where $H$ is $f$-ample and $m\ge 2$.

\smallskip

\noindent
$(\mathrm{b})$ $f$ is a relatively minimal fibration of genus greater than $1$ and $D-K_X$ is $f$-nef.
\end{corollary}

\begin{proof}
Let $x$ be a base point of $f_{*}\O_X(K_X+D)$.
Then Theorem~\ref{sep} implies that there exists an $f$-exceptional curve $B$ passing through $x$ on $X$ such that $(D-B)B\le 1$.
Since $B^2\le 0$ by Zariski's lemma, we have $DB\le 1+B^2\le 1$.
Hence $B^2=0$ or $-1$.
When $B^2=0$, then we can write $B\equiv aF$ for some $a\in \Q_{>0}$.
Then $DB>0$ by assumption.
Hence the first half of the claim follows.
The rest of the claim follows easily by using the fact that $K_XB+B^2$ is even.
\end{proof}

Note that the last assertion of Corollary~\ref{bpfdim1} also follows from Lemmas~1.3.2 and 4.2.1 in \cite{Kon123}.

Similarly to Corollary~\ref{bpfdim1}, the following can be proved.

\begin{corollary}[cf.\ \cite{Sak2} Theorem~7, \cite{Lau} Theorem~3.1]\label{bpfdim2}
Let $f\colon X\to Y$ be a resolution of a normal surface singularity $(Y,y)$.
Let $D$ be an $f$-nef divisor on $X$.
Then for any base point $x$ of $f_{*}\O_X(K_X+D)$, there exists an $f$-exceptional curve $B$ on $X$ passing through $x$ such that $DB=0$ and $B^2=-1$ hold.
In particular, if $D$ is $f$-ample, or $f$ is a minimal resolution and $D-K_X$ is $f$-nef, then the natural map $f^{*}f_{*}\O_X(K_X+D)\to \O_X(K_X+D)$ is surjective.
\end{corollary}

We can show the very ample cases similarly.
Note that for a proper morphism $f\colon X\to Y$ between varieties over an algebraically closed field, a Cartier divisor $L$ on $X$ is $f$-very ample if and only if the natural map $f_{*}\O_{X}(L)\to f_{*}(\O_{X}(L)|_{\zeta})$ is surjective for any cluster $\zeta$ of length $2$ contained in a fiber of $f$.

\begin{corollary}\label{vadim0}
Let $X$ be a normal complete surface over an algebraically closed field of characteristic $0$.
Let $D$ be a nef divisor on $X$ and $\zeta$ a cluster along which $K_X+D$ is Cartier.
Assume that 

\smallskip

\noindent
$(\mathrm{i})$ $D^2>\delta_{\zeta}$, and

\smallskip

\noindent
$(\mathrm{ii})$ $DB\ge \frac{1}{2}\delta_{\zeta}$ for any curve $B$ on $X$ intersecting $\zeta$.

\smallskip

\noindent
Then $|K_X+D|$ separates $\zeta$.
In particular, if $X$ is smooth, $D^2>8$ and $DB\ge 4$ for any curve $B$ on $X$, then $K_X+D$ is very ample.
\end{corollary}

\begin{corollary}\label{vadim1}
Let $f\colon X\to Y$ be a fiber space from a smooth surface $X$ to a curve $Y$.
Let $D$ be an $f$-nef divisor on $X$ with $DF>0$ for a fiber $F$ of $f$.
If $f_{*}\O_X(K_X+D)$ does not separate a cluster $\zeta$ of length $2$, then there exists an $f$-exceptional curve $B$ on $X$ intersecting $\zeta$ such that one of the following holds.

\smallskip

\noindent
$(\mathrm{i})$ $DB=0$ and $B^2=-2$ or $-1$.

\smallskip

\noindent
$(\mathrm{ii})$ $DB=1$ and $B^2=-1$ or $0$.

\smallskip

\noindent
$(\mathrm{iii})$ $DB=2$ and $B^2=0$.

\smallskip

\noindent
In particular, if one of the following conditions $(\mathrm{a})$ and $(\mathrm{b})$ holds, then
$K_X+D$ is $f$-very ample.

\smallskip

\noindent
$(\mathrm{a})$ $D-3H$ is $f$-nef for some $f$-ample divisor $H$ on $X$.

\smallskip

\noindent
$(\mathrm{b})$ $f$ is a relatively minimal fibration of genus greater than $1$, $D-K_X$ is $f$-nef and there are no curves $B$ with $p_a(B)=i$ and $B^{2}=i-2$ $(i=0,1,2)$ contained in fibers.
\end{corollary}

\begin{corollary}\label{vadim2}
Let $f\colon X\to Y$ be a resolution of a normal surface singularity $(Y,y)$.
Let $D$ be an $f$-nef divisor on $X$.
If $f_{*}\O_X(K_X+D)$ does not separate a cluster $\zeta$ of length $2$, then there exists an $f$-exceptional curve $B$ on $X$ intersecting $\zeta$ such that one of the following hold.

\smallskip

\noindent
$(\mathrm{i})$ $DB=0$ and $B^2=-2$ or $-1$.

\smallskip

\noindent
$(\mathrm{ii})$ $DB=1$ and $B^2=-1$.

\smallskip

\noindent
In particular, if one of the following conditions $(\mathrm{a})$ and $(\mathrm{b})$ holds, then
$K_X+D$ is $f$-very ample.

\smallskip

\noindent
$(\mathrm{a})$ $D-2H$ is $f$-nef for some $f$-ample divisor $H$ on $X$.

\smallskip

\noindent
$(\mathrm{b})$ $f$ is a minimal resolution, $D-K_X$ is $f$-nef and there are no $f$-exceptional $(-2)$-curves and $f$-exceptional curves $E$ with $p_a(E)=1$ and $E^{2}=-1$.
\end{corollary}


\section{Extension theorems}
\label{sec:Extension theorems}


\subsection{Main theorem}
Let $X$ be a normal complete surface over a field $k$.
Let us define the invariant $q_X$ as
$$
q_X:=\mathrm{min}\{E^2\ |\ \text{$E$ is an effective divisor on $X$ with $E^2>0$} \}.
$$
For example, $q_X\ge 1$ if $X$ is smooth and $q_X\ge 2$ if $X$ is smooth and $K_X$ is numerically even.
The main theorem is a generalization of Serrano--Paoletti's extension theorem (\cite{Ser}, \cite{Pao}) as follows (for the definitions of the maximal degree and the pseudo-gonality, see Appendix~B).

\begin{theorem}[Extension theorem] \label{extnthm}
Let $X$ be a normal proper surface over a field $k$ of characteristic $0$.
Let $D>0$ be an effective divisor on $X$ all of whose prime components have positive self-intersection numbers.
Let $\varphi\colon D\to C$ be a finite morphism to a proper scheme $C$ over $k$ with scheme-theoretic image $\varphi(D)$.
We put $d:=\maxdeg(\varphi)\mathrm{pgon}_{\infty}(\varphi(D))$ and $q_{X,d}:=\mathrm{min}\{q_X,d\}$.
If $D^2>q_{X,d}(d/q_{X,d}+1)^2$, then there exists a morphism $\psi\colon X\to C$ such that the restriction of $\psi$ to $D$ is equal to $\varphi$.
\end{theorem}

\begin{remark}\label{lowerbound}
Let $\phi(x,d):=q_{x,d}(d/q_{x,d}+1)^2$ for $x,d\in \R_{>0}$, where $q_{x,d}:=\min\{x,d\}$.
Then $\phi(x,d)$ is non-increasing with respect to $x$ and monotonically increasing with respect to $d$.
The lower bound of $D^2$ in Theorem~\ref{extnthm} is nothing but $\phi(q_X,d)$.
Thus in order to check the condition $D^2>\phi(q_X,d)$, it suffices to verify $D^2>\phi(x,d')$ for some $x\le q_X$ and $d'\ge d$.
For example, we can take $x=1$ (resp.\ $x=2$) when $X$ is smooth (resp.\ and $K_X$ is numerically even).
In this case, $\phi(1,d)=(d+1)^2$ and $\phi(2,d)=(d+2)^2/2$ coincide with the lower bounds of $D^2$ in Theorem~3.1 in \cite{Ser}.
Moreover, we may replace $d$ by $d':=\maxdeg(\varphi)\mathrm{ps}_2(\varphi(D))$ (resp.\ $d':=\maxdeg(\varphi)(\mathrm{ps}_2(\varphi(D))-1)$ when $\varphi(D)$ has an irreducible component of geometric multiplicity $1$ with infinitely many rational points) by Proposition~\ref{invariants} (6).
\end{remark}

\begin{lemma} \label{wmapencil}
In order to prove Theorem~\ref{extnthm}, we may assume that $C=\mathbb{P}^1$.
\end{lemma}

\begin{proof}
Assume that Theorem~\ref{extnthm} for $C=\mathbb{P}^1$, that is, Theorem~\ref{introextnthm} is true.
Let $X$, $D$, $\varphi\colon D\to C$ and $d$ be as in Theorem~\ref{extnthm} and suppose that $D^2>q_{X,d}(d/q_{X,d}+1)^2$.
We may assume that $\varphi(D)=C$.
By the definition of $\mathrm{pgon}_{\infty}(C)$, we can take an infinite family $\{\alpha_i\colon C\to \mathbb{P}^1\}_{i\in I}$ of finite morphisms of degree $\le \mathrm{pgon}_{\infty}(C)$ with condition $(A)$ in Definition~\ref{gonmd}.
Let us consider the composition $\alpha_i\circ\varphi\colon D\to \mathbb{P}^1$ of $\varphi$ and $\alpha_i$ for $i\in I$.
Then the degree of $\alpha_i\circ \varphi$ is less than or equal to $d$ by Lemma~\ref{findeg} (2).
Then by assumption, there exist morphisms $\psi_i\colon X\to \mathbb{P}^1$ such that $\psi_i|_{D}=\alpha_i\circ \varphi$ for any $i\in I$.
We take the Stein factorizations of $\psi_i$ as $\psi_i=\beta_i\circ f_i\colon X\to Y_i\to \mathbb{P}^1$, where $f_i\colon X\to Y_i$ is a fibration over a smooth curve $Y_i$ with connected fibers and $\beta_i\colon Y_i\to \mathbb{P}^1$ is finite.
Let $F_i$ denote the numerical equivalence class of a fiber of $f_i$.
Since $0<DF_i\le d$ holds for any $i\in I$,
the set $\{F_i\}_{i\in I}$ belongs to finitely many numerical equivalence classes in $X$.
Indeed, the multiplication of $D$ defines a positive function on the polyhedral cone $\Eff(X)/\{C\ge 0\ |\ DC=0\}$.
Hence there is an infinite subset $J\subset I$ such that $f_i$ is the same fibration for any $i\in J$, which is denoted by $f\colon X\to Y$.
Note that $f|_D\colon D\to Y$ is finite and flat because any irreducible component of $D$ has positive self-intersection number, $D$ is Cohen-Macaulay and $Y$ is smooth.
Then $f|_D$ and the projections $\pr_i\colon D\times_{Y}D\to D$, $i=1,2$ are fppf coverings.
In order to prove Theorem~\ref{extnthm},
it suffices to show the existence of a morphism $\gamma\colon Y\to C$ such that $\gamma\circ f|_{D}=\varphi$.
Indeed, $\psi:=\gamma\circ f$ is the desired one.
For any $i\in J$, we have
$$
\alpha_i\circ \varphi\circ \pr_1=\beta_i\circ f|_D\circ \pr_1=\beta_i\circ f|_D\circ \pr_2=\alpha_i\circ \varphi\circ \pr_2.
$$
Then we obtain $\varphi\circ \pr_1=\varphi\circ \pr_2$ by Lemma~\ref{mono}.
Since $f|_D$ is an fppf covering, the morphism $\varphi\colon D\to C$ descends to a morphism $\gamma\colon Y\to C$ with $\varphi=\gamma\circ f|_D$ (cf.\ VIII, 5.2 in \cite{SGA}), which completes the proof.
\end{proof}

\subsection{Proof of Theorem~\ref{introextnthm}}
Let $X$, $D$, $\varphi\colon D\to \mathbb{P}^1$ and $d$ be as in Theorem~\ref{introextnthm}.
In order to prove Theorem~\ref{introextnthm}, we may assume that $X$ is geometrically connected by replacing $k$ and $\varphi\colon D\to \mathbb{P}^1_{k}$ to $k':=H^{0}(X,\O_X)$ and $\varphi_{k'}\colon D\to \mathbb{P}^{1}_{k}\times_{k}k'\cong \mathbb{P}^1_{k'}$.
For a closed point $\lambda\in \mathbb{P}^1$, let us denote $\mathfrak{a}_{\lambda}:=\varphi^{-1}(\lambda)$ as the fiber of $\varphi$ at $\lambda$.
Let $\Lambda$ be the dense subset of $\mathbb{P}^1$ consisting of closed points $\lambda$ such that $\mathfrak{a}_{\lambda}$ is reduced and contained in the smooth loci of $X$ and $D_{\mathrm{red}}$.

\begin{lemma} \label{notsep}
For each closed point $\lambda\in \Lambda$, the restriction $H^{0}(X,\O_{X}(K_X+D))\to H^{0}(\mathfrak{a}_{\lambda}, \O_{X}(K_X+D)|_{\mathfrak{a}_{\lambda}})$ is not surjective.
\end{lemma}

\begin{proof}
Assume contrary that this is surjective.
By the exact sequence
$$
0\to \O_X(K_X)\to \O_X(K_X+D)\to \omega_{D}\to 0
$$
obtained by applying $\mathscr{E}xt(-,\omega_X)$ to \eqref{restex},
it follows that $H^{0}(D,\omega_D)\to H^{0}(\mathfrak{a}_{\lambda},\omega_D|_{\mathfrak{a}_{\lambda}})$ is surjective.
Then we have $h^{0}(\omega_{D}(-\mathfrak{a}_{\lambda}))=p_a(D)-h^{0}(\omega_D|_{\mathfrak{a}_{\lambda}})$.
On the other hand, by the exact sequence
$$
0\to \O_D\to \O_D(\mathfrak{a}_{\lambda})\to \O_D(\mathfrak{a}_{\lambda})|_{\mathfrak{a}_{\lambda}}\to 0,
$$
we have $\chi(\O_D(\mathfrak{a}_{\lambda}))=\chi(\O_D)+h^{0}(\O_D(\mathfrak{a}_{\lambda})|_{\mathfrak{a}_{\lambda}})$.
Since $D$ is nef and big and $X$ is geometrically connected, we have $h^{0}(\O_D)=h^{0}(\O_X)=1$.
By these equalities and the Serre duality, we have $h^{0}(\O_D(\mathfrak{a}_{\lambda}))=1$, which contradicts the fact that $\mathfrak{a}_{\lambda}$ is a member of a linear pencil on $D$.
\end{proof}

We fix a closed point $\lambda\in \Lambda$ arbitrarily.
Let $D=\sum_{i}m_iD_i$ be the prime decomposition of $D$, where $m_i$ is the geometric multiplicity of $D_i$ in $D$.
Let us write $\mathfrak{a}_{\lambda}|_{D_i}=p_{i,1}+\cdots+p_{i,d_i}$, where the $p_{i,j}$ are closed points on $D_i$.
Note that $[k(\lambda):k]d=\sum_{i,j}m_{i}[k(p_{i,j}):k]$ holds by Lemma~\ref{lbdeg}.

\begin{lemma}\label{deltabound}
For any closed point $\lambda\in \Lambda$, we have 
$\delta_{\mathfrak{a}_{\lambda}}\le 4d[k(\lambda):k]$.
\end{lemma}

\begin{proof}
Since $\lambda$ is contained in $\Lambda$, we can take a system of parameters $z$, $w$ of $\O_{X,p_{i,j}}$ such that
the ideal sheaves $\mathcal{I}_{\mathfrak{a}_{\lambda}/X}$, $\mathcal{I}_{D/X}$, $\mathcal{I}_{D_{i}/X}$ and $\mathcal{I}_{\mathfrak{a}_{\lambda}/D}$ are respectively generated at $p_{i,j}$ by $\{z^{m_i}, w\}$, $z^{m_i}$, $z$ and $w$.
Then for any $p_{i,j}$, we can take a sequence of blow-ups at single points
$$
X_{m_i}\xrightarrow{\pi_{m_i}} X_{m_i-1}\xrightarrow{\pi_{m_i-1}} \cdots \xrightarrow{\pi_2} X_1\xrightarrow{\pi_1} X_0=X
$$
such that $\pi_1$ is the blow-up at $p^{1}_{i,j}=p_{i,j}$, the center $p^{k}_{i,j}$ of $\pi_k$ is infinitely near to $p^{k-1}_{i,j}$ and not contained in the proper transform of $D_i$ for $k\ge 2$ and the push-forward of the sheaf $\O_{X_{m_i}}(-\sum_{k}E^{k}_{i,j})$ to $X$ is isomorphic to $\mathcal{I}_{\mathfrak{a}_{\lambda}/X}$ at $p_{i,j}$, where $E^{k}_{i,j}$ is the total transform of the exceptional divisor of $\pi_{k}$.
By taking these blow-ups for all $p_{i,j}$, we obtain a proper birational morphism $\pi\colon X'\to X$ the center of which is supported at the points $p_{i,j}$ such that $\pi_{*}\O_{X'}(-\sum_{i,j,k}E^{k}_{i,j})=\mathcal{I}_{\mathfrak{a}_{\lambda}/X}$.
Then we can take $Z=\sum_{i,j,k}E^{k}_{i,j}$ as in Definition~\ref{cluster}.
Since the anti-canonical cycle of $\pi$ is $\Delta=-\sum_{i,j,k}E^{k}_{i,j}$, we obtain
$$
\delta_{\mathfrak{a}_{\lambda}}\le -(\Delta-Z)^2=-4\sum_{i,j,k}(E^{k}_{i,j})^2=4\sum_{i,j}m_{i}[k(p_{i,j}):k]=4d[k(\lambda):k].
$$
\end{proof}

From now on, we consider the case that $\lambda\in \Lambda$ is a rational point.
Note that $D$ and $\zeta:=\mathfrak{a}_{\lambda}$ satisfy the assumption of Theorem~\ref{negsemidefnthm} by the condition of $D^2$ in Theorem~\ref{introextnthm}, Lemma~\ref{notsep} and Lemma~\ref{deltabound}.

\begin{lemma} \label{fiberexist}
For any rational point $\lambda\in \Lambda$, there exists a negative semi-definite curve $B_{\lambda}$ on $X$ such that $B_{\lambda}^2=0$ and $D\cap B_{\lambda}$ is a subscheme of $\mathfrak{a}_{\lambda}$.
\end{lemma}

\begin{proof}
Let us put $D':=\pi^{*}D+\Delta-Z=\pi^{*}D-2\sum_{i,j,k}E^{k}_{i,j}$, where $\pi\colon X'\to X$ is as in the proof of Lemma~\ref{deltabound}.
Then $D'$ is big since $D'^{2}= D^2-4d>0$.
Let $D'=P'_{\Z}+N'_{\Z}$ be the integral Zariski decomposition of $D'$.
Then we can write $N'_{\Z}=\pi^{*}B_{\lambda}+\sum_{i,j,k}a^{k}_{i,j}E^{k}_{i,j}$ and $P'_{\Z}=\pi^{*}(D-B_{\lambda})-\sum_{i,j,k}(a^{k}_{i,j}+2)E^{k}_{i,j}$ for some $a^k_{i,j}\in \Z$, where we recall that the curve $B_{\lambda}=\pi_{*}N'_{\Z}$ on $X$ is nothing but the curve $B$ obtained by applying Theorem~\ref{negsemidefnthm} to $D$ and $\zeta=\mathfrak{a}_{\lambda}$ (see the proof of Theorem~\ref{sep}).
Then we have
\begin{align*}
0&\ge P'_{\Z}N'_{\Z} \\
&= (D-B_{\lambda})B_{\lambda}+\sum_{i,j,k}(a^k_{i,j}+2)a^k_{i,j}[k(p_{i,j}):k] \\
&\ge DB_{\lambda}+\sum_{i,j,k}(a^k_{i,j}+2)a^k_{i,j}[k(p_{i,j}):k] \\
&\ge \sum_{i,j}m_i(D_iB_{\lambda})_{p_{i,j}}+\sum_{i,j,k}(a^k_{i,j}+2)a^k_{i,j}[k(p_{i,j}):k] \\
&\ge \sum_{p_{i,j}\in B_{\lambda}}m_i(D_iB_{\lambda})_{p_{i,j}}+\sum_{p_{i,j}\in B_{\lambda}, k}(a^k_{i,j}+2)a^k_{i,j}[k(p_{i,j}):k] \\
&\ge \sum_{p_{i,j}\in B_{\lambda}}m_i[k(p_{i,j}):k]+\sum_{p_{i,j}\in B_{\lambda}, k}(-1)[k(p_{i,j}):k] \\
&=0, 
\end{align*}
where the third and the forth inequalities follow from the fact that $B_{\lambda}$ is negative semi-definite by Theorem~\ref{negsemidefnthm} and thus $B_{\lambda}$ and $D$ have no common components,
and fifth inequality follows from $a^k_{i,j}\ge 0$ for any $i, j, k$ with $p_{i,j}\notin B_{\lambda}$.
Thus all inequalities as above are equalities and then we conclude that
$P'_{\Z}N'_{\Z}=0$, $B_{\lambda}^2=0$, $D\cap B_{\lambda}$ is a subscheme of $\mathfrak{a}_{\lambda}$,
$N'_{\Z}$ is the proper transform of $B_{\lambda}$, $D$ intersects $B_{\lambda}$ only at some points $p_{i,j}$ transversely and $p^k_{i,j}$ is contained in the smooth locus of the proper transform of $B_{\lambda}$ for any $i, j, k$ with $p_{i,j}\in B_{\lambda}$.
\end{proof}


Let $\mathcal{B}$ be the set of all prime divisors $C$ on $X$ which is a component of $B_{\lambda}$ for some rational point $\lambda\in \Lambda$ and satisfies $DC>0$.
Then it is an infinite set since $B_{\lambda}\cap D\subset \mathfrak{a}_{\lambda}$ and $\mathfrak{a}_{\lambda}\cap \mathfrak{a}_{\mu}= \emptyset$ for $\lambda\neq \mu$.
Since any $C\in \mathcal{B}$ satisfies $0<DC\le DB_{\lambda}\le d$, the set $\mathcal{B}$ has  finitely many numerical equivalence classes (see the proof of Lemma~\ref{wmapencil}), say $B_{(1)},\ldots,B_{(m)}$ in $N(X)_{\Z}$.
Let us put $\mathcal{B}_{(i)}:=\{C\in \mathcal{B}\ |\ C\equiv B_{(i)}\}$.
When $B_{(i)}^2<0$, the set $\mathcal{B}_{(i)}$ consists of only one element 
since two other elements $C$ and $C'$ in $\mathcal{B}_{(i)}$ would satisfy $CC'\ge 0$.
Since $\mathcal{B}$ is infinite and any $C\in \mathcal{B}$ satisfies $C^2\le 0$, there exist at least one $B_{(i)}$ such that $B_{(i)}^2=0$ and $\mathcal{B}_{(i)}$ is infinite.
We choose such a $B_{(i)}$ and put $\mathcal{B}:=\mathcal{B}_{(i)}$ again.
Now we use the following result (for $k=\overline{k}$ case, see also \cite{BPS}).

\begin{proposition} \label{fiblem}
Let $X$ be a normal complete surface over a perfect field $k$.
If $X$ has an infinite family $\mathcal{B}=\{C_{\mu}\}_{\mu}$ of integral curves on $X$ which have the same numerical equivalence class $B$ with $B^2=0$, 
then $X$ admits a fibration $f\colon X\to Y$ onto a smooth curve $Y$ such that any member of $\mathcal{B}$ is a fiber of $f$.
\end{proposition}

\begin{proof}
Note that infinite curves in $\mathcal{B}$ have no singular points of $X$ since $X$ has finitely many singular points and any two distinct curves $C_{\mu}, C_{\nu} \in \mathcal{B}$ are disjoint.
Thus we may assume that each member $C_{\mu}$ in $\mathcal{B}$ has no singular points of $X$.
We take a resolution $\pi\colon X'\to X$.
Then $X'$ is a smooth projective surface and
the family $\pi^{*}\mathcal{B}=\{\pi^{*}C_{\mu}\}_{\mu}$ also satisfies the assumption of Proposition~\ref{fiblem}.
If $X'$ admits a fibration $f'\colon X'\to Y$ onto a smooth curve such that any $\pi^{*}C_{\mu}$ is a fiber of $f'$, then the induced rational map $f\colon X\dasharrow Y$ becomes a morphism since $\pi^{*}C_{\mu}$ is disjoint from the $\pi$-exceptional curves.
Thus we may assume that $X$ is smooth projective.

Let us fix a closed embedding $X\subset \mathbb{P}^{n}_{k}$ by a very ample divisor $\O_{X}(1)=\O_{\mathbb{P}^n}(1)|_{X}$.
Note that any member $C_{\mu}\in \mathcal{B}$ has the same Hilbert polynomial with respect to $\O_X(1)$, which is denoted by $P(t)$, since the numerical equivalence class determines the Hilbert polynomial.
Let $\mathcal{H}$ be the Hilbert scheme of closed subschemes of $X$ with the Hilbert polynomial $P(t)$.
Then it is projective and has positive dimension since it has infinite points corresponding to $C_{\mu}\in \mathcal{B}$.
Then we can take an irreducible subcurve $Y'\subset \mathcal{H}$ such that it has a rational point $y_0$ corresponding to some $C_{\mu}\in \mathcal{B}$ as a smooth point.
Indeed, if there exists a smooth point $y_0$ of the maximal reduced subscheme $\mathcal{H}'$ of $\mathcal{H}$ corresponding to some $C_{\mu}\in \mathcal{B}$, then we can take general hyperplanes in $\mathcal{H}'$ passing through $y_0$ as a smooth point and define $Y'$ as these intersection.
If all points corresponding to elements of $\mathcal{B}$ are contained in the singular locus $\mathcal{H}'_{\mathrm{sing}}$ (with the reduced induced structure) of $\mathcal{H}'$, then the dimension of $\mathcal{H}'_{\mathrm{sing}}$ is positive and it has an irreducible component $\mathcal{H}_1$ which contains infinite points corresponding to elements of $\mathcal{B}$.
Then we can repeat this procedure to $\mathcal{H}_1$ and obtain a sequence of reduced subschemes
$\mathcal{H}_m\subset \mathcal{H}_{m-1} \subset \cdots \subset \mathcal{H}_1$
such that the smooth locus of $\mathcal{H}_m$ contains a point $y_0$ corresponding to some $C_{\mu}\in \mathcal{B}$.
Thus we can take $Y'$ as the intersection of general hyperplanes in $\mathcal{H}_m$ or $\mathcal{H}_m$ itself.
Let $u\colon U\to Y'$ be the restriction of the universal family on $\mathcal{H}$.
Let $\tau\colon U\to X$ be the composition of the closed embedding $U\hookrightarrow X\times Y'$ and the first projection $X\times Y'\to X$.
For a closed point $y\in Y'$, let $F_y$ denote the push-forward of $U_y$ by $\tau$ as $1$-cycles, where $U_y$ is a fiber of $u$ at $y$ which is the closed subscheme of $X_{k(y)}:=X\times \mathrm{Spec}(k(y))$ with the Hilbert polynomial $P(t)$ with respect to $\O_{X_{k(y)}}(1)$.
Note that $F_{y_0}$ is equal to $C_{\mu}$ by definition.
For a smooth point $y\in Y'$, we have $F_y\equiv [k(y):k]C_{\mu}$ 
since the degree of $y$ is equal to $[k(y):k]$ and thus $U_y$ is numerically equivalent to $[k(y):k]U_{y_0}$ as Cartier divisors on $U$.
Then we have $F_yC_{\mu}=[k(y):k]C_{\mu}^2=0$.
Since $C_{\mu}$ is irreducible, we can write $F_y=a_yC_{\mu}+F'_y$ for some $a_y\in \Z_{\ge 0}$ with $F'_y\cap C_{\mu}=\emptyset$.

Now we take a sufficiently ample linear system $|\mathfrak{a}|$ on $Y'$ and two general members $\sum_{i}y_i$ and $ \sum_{i}y'_i$ of $|\mathfrak{a}|$ whose supports are contained in the smooth locus of $Y'$.
Then $\sum_{i}F_{y_i}$ is linearly equivalent to $\sum_{i}F_{y'_i}$ and define a linear pencil $\Gamma$ on $X$.
Let $\Gamma=M+Z$ denote the decomposition into the movable part $M$ and the fixed part $Z$.
Then $M$ has no base points lying on $C_{\mu}$ since any irreducible component of $\sum_{i}F_{y_i}$ other than $C_{\mu}$ is disjoint from $C_{\mu}$.
Since $MC_{\mu}=(\sum_{i}F_{y_i}-Z)C_{\mu}=0$ by the nefness of $C_{\mu}$ and $F_{y_i}C_{\mu}=0$, it follows from the Hodge index theorem that $M^2=0$ and then $M$ has no base points.
Thus it defines a fibration $f\colon X\to Y$ which contracts $C_{\mu}$ to a point on $Y$.
Taking the Stein factrization, we may assume that $f$ has connected fibers and $Y$ is smooth.
Since $C_{\mu}$ is integral and $C_{\mu}^2=0$, it is a fiber of $f$ or the multiple component of an irreducible multiple fiber of $f$.
Note that $f$ has only finitely many non-integral fibers because the base field $k$ is perfect (cf.\ Corollary~7.3 in \cite{Bad}).
Since all members of $\mathcal{B}$ contract to points on $Y$ and have the same numerically equivalence class, all of them are fibers of $f$.
\end{proof}

Let us continue the proof of Theorem~\ref{introextnthm}.
By Proposition~\ref{fiblem}, there exists a fibration $f\colon X\to Y$ onto a smooth curve $Y$ such that any member $C\in \mathcal{B}$ is a fiber of $f$.
Let $\overline{D}$ denote the scheme-theoretic image of $(f|_{D},\varphi)\colon D\to Y\times \mathbb{P}^1$ and $h\colon \overline{D}\to Y$ the restriction to $\overline{D}$ of the first projection of $Y\times \mathbb{P}^1$.

\begin{lemma}\label{isom}
In the situation as above, $h\colon \overline{D}\to Y$ is an isomorphism.
\end{lemma}

\begin{proof}
Note that $h$ is finite since any prime component of $D$ has positive self-intersection number and so there are no components of $D$ which are contracted by $f$.
Thus it suffices to show that $h$ is birational.
By Lemma~\ref{fiberexist}, there are infinite rational points $\lambda$ in $\Lambda$ such that $f^{-1}(y_{\lambda})\cap D$ is a non-empty subscheme of $\mathfrak{a}_{\lambda}$ for some closed point $y_{\lambda}$ in $Y$.
For such a $\lambda\in \Lambda$, let $z_{\lambda}$ be the closed point of $Y\times \mathbb{P}^1$ corresponding to $(y_{\lambda},\lambda)$.
The fiber of $(f_{D},\varphi)$ at $z_{\lambda}$ is 
$$
(f|_{D},\varphi)^{-1}(z_{\lambda})=(f|_{D})^{-1}(y_{\lambda})\cap \mathfrak{a}_{\lambda}=(f|_{D})^{-1}(y_{\lambda}).
$$
Thus $z_{\lambda}$ is contained in $\overline{D}$ and $h^{-1}(y_{\lambda})=z_{\lambda}$ holds.
Since $y_{\lambda}$ and $z_{\lambda}$ have the same residue field, the assertion holds by Lemma~\ref{ptbir}.
\end{proof}

Let $\gamma\colon Y\to \mathbb{P}^1$ denote the composition of $h^{-1}\colon Y\to \overline{D}$ and the restriction of the second projection $\overline{D}\to \mathbb{P}^1$.
Then the morphism $\varphi$ decomposes into $f|_{D}$ and $\gamma$.
Hence $\psi:=\gamma\circ f\colon X\to \mathbb{P}^1$ is an extension of $\varphi$, which completes the proof of Theorem~\ref{introextnthm}.

\begin{remark}
When $X$ is a complex analytic normal complete surface, the proof of Theorem~\ref{extnthm} also works with a slight modification.
Indeed, Lemma~\ref{notsep}, Lemma~\ref{deltabound}, Lemma~\ref{fiberexist} and Lemma~\ref{isom} completely work in the analytic settings.
Note that $X$ is Moishezon by the assumption of $D$.
Then Proposition~\ref{fiblem} also works because any resolution of $X$ is projective.
Finally, we can prove Lemma~\ref{wmapencil} similarly to Lemma~\ref{isom} in the analytic settings.
\end{remark}

\subsection{Extension theorem with base points}

The following is a generalization of Theorem~3.8 in \cite{Ser}.
\begin{theorem}[Extension theorem with base points] \label{basepointextnthm}
Let $X$ be a normal proper surface over a field $k$ of characteristic $0$.
Let $D>0$ be an effective divisor on $X$ all of whose prime components have positive self-intersection numbers.
Let $\varphi\colon D\to \mathbb{P}^1$ be a finite morphism of degree $d$ which can not be extended to a morphism on $X$.
We assume that $D^2=q_X(d/q_X+1)^2$ and $q_X<d$.
Then there exists a linear pencil $\{F_{\lambda}\}_{\lambda\in \mathbb{P}^1}$ of Weil divisors with $F_{\lambda}^2=q_X$ and no fixed parts such that the induced rational map $\psi\colon X\dasharrow \mathbb{P}^1$ satisfies $\psi|_{D}=\varphi$.
\end{theorem}

\begin{proof}
We freely use the notation in the proof of Theorem~\ref{introextnthm}.
For a rational point $\lambda\in \Lambda$, we also consider the birational morphism $\pi\colon X'\to X$ with center $\mathfrak{a}_{\lambda}$.
Recall that 
$$
q_{\mathfrak{a}_{\lambda}}:=\mathrm{min}\{E^2, d \ |\ \text{$E$ is an effective divisor on $X$ with $E\cap \mathfrak{a}_{\lambda} \neq \emptyset$ and $E^2>0$}\},
$$
(where we may assume that $\delta_{\mathfrak{a}_{\lambda}}=4d$ by replacing the birational morphism used in the proof of Theorem~\ref{negsemidefnthm} to $\pi$, or we can check directly that $\delta_{\mathfrak{a}_{\lambda}}=4d$).
Then we have $q_X\le q_{\mathfrak{a}_{\lambda}}\le d$.
Since $D^2>4d$, we can apply Theorem~\ref{negsemidefnthm} and then there exists a curve $B_{\lambda}$ intersecting $\mathfrak{a}_{\lambda}$ such that $0<(D-B)B\le d$ holds and $D-2B$ is big.
If moreover $q_X<q_{\mathfrak{a}_{\lambda}}$, then the curve $B_{\lambda}$ becomes negative semi-definite by $D^2>q_{\mathfrak{a}_{\lambda}}(d/q_{\mathfrak{a}_{\lambda}}+1)^2$.
If $q_X=q_{\mathfrak{a}_{\lambda}}$, then $B_{\lambda}$ is negative semi-definite or $B_{\lambda}^2=q_X$ and $D\equiv (d/q_X+1)B$ hold.
If $B_{\lambda}$ is negative semi-definite, then we have $B_{\lambda}^2=0$ and $D\cap B_{\lambda}\subset \mathfrak{a}_{\lambda}$ by the same proof of Lemma~\ref{fiberexist}.
If $B_{\lambda}^2=q_X$ and $D\equiv (d/q_X+1)B$, then we can estimate as in the proof of Lemma~\ref{fiberexist} as follows:
\begin{align*}
0&\ge P'_{\Z}N'_{\Z} \\
&= (D-B_{\lambda})B_{\lambda}+\sum_{i,j,k}(a^k_{i,j}+2)a^k_{i,j}[k(p_{i,j}):k] \\
&= d+\sum_{i,j,k}(a^k_{i,j}+2)a^k_{i,j}[k(p_{i,j}):k] \\
&\ge d+\sum_{i,j,k}(-1)[k(p_{i,j}):k] \\
&=0.
\end{align*}
Thus all inequalities as above are equalities and then we have $P'_{\Z}N'_{\Z}=0$ and $a^k_{i,j}=-1$ for all $i,j,k$.
In particular, $D\cap B_{\lambda}$ contains $\mathfrak{a}_{\lambda}$ as a subscheme.
Moreover, $B_{\lambda}$ and $D$ have no common components.
Indeed, if $B_{\lambda}$ has a prime component $D_i$ of $D$, then it follows from Remark~\ref{subcurve} and $D_i^2>0$ that $B_{\lambda}=D_i$.
Since $\mathfrak{a}_{\lambda}$ is a subscheme of $B_{\lambda}$, then $D=D_i$ holds, which contradicts $D\equiv (d/q_X+1)B$.

Assume that there exist infinitely many rational points $\lambda\in \Lambda$ such that $B_{\lambda}$ is negative semi-definite.
Then the same proof of Theorem~\ref{introextnthm} works in this situation and thus there exists a morphism $\psi\colon X\to \mathbb{P}^1$ such that $\psi|_{D}=\varphi$, which contradicts the assumption of Theorem~\ref{basepointextnthm}.
Hence there exist (infinitely many) rational points $\lambda\in \Lambda$ such that $B_{\lambda}^2=q_X$ and $D\equiv (d/q_X+1)B$ hold.
We fix such a $\lambda\in \Lambda$ arbitrarily.
Let $\mathcal{L}_{D,B_{\lambda}}$ be the cokernel of the natural homomorphism $\O_{X}(B_{\lambda}-D)\to \O_{X}(B_{\lambda})$.
Then we have $H^{0}(X,\O_{X}(B_{\lambda}))\cong H^{0}(D,\mathcal{L}_{D,B_{\lambda}})$ since $H^{0}(X,\O_X(B_{\lambda}-D))=H^1(X,\O_X(B_{\lambda}-D))=0$ by Theorem~\ref{van}.
Moreover, there is a natural inclusion $\O_D(\mathfrak{a}_{\lambda})\subset \mathcal{L}_{D,B_{\lambda}}$.
Indeed, we consider the exact sequence
$$
0\to \O_{X}(B_{\lambda}-D)\to \mathcal{I}_{\mathfrak{a}_{\lambda}/X}\O_X(B_{\lambda})\to \mathcal{L}_{D,B_{\lambda}}(-\mathfrak{a}_{\lambda})\to 0.
$$
Then we have $H^{0}(X,\mathcal{I}_{\mathfrak{a}_{\lambda}/X}\O_{X}(B_{\lambda}))\cong H^{0}(D,\mathcal{L}_{D,B_{\lambda}}(-\mathfrak{a}_{\lambda}))$.
Since $\mathfrak{a}_{\lambda}\subset B_{\lambda}$, there is a narural section $s$ in $H^0(X, \mathcal{I}_{\mathfrak{a}_{\lambda}/X}\O_{X}(B_{\lambda}))$.
This section $s$ defines a desired inclusion $\O_D(\mathfrak{a}_{\lambda})\subset \mathcal{L}_{D,B_{\lambda}}$.
Let $\{F_{\mu}\}_{\mu \in \mathbb{P}^1}$ be the linear subpencil of $|B_{\lambda}|$ obtained by the pull-back of the pencil $\{\mathfrak{a}_{\mu}\}_{\mu\in \mathbb{P}^1}$.
Then $\mathfrak{a}_{\mu} \subset F_{\mu}\cap D$ holds for any $\mu$ and $F_{\lambda}=B_{\lambda}$.
If the pencil $\{F_{\mu}\}_{\mu \in \mathbb{P}^1}$ has the fixed part $Z>0$, then the movable part $M>0$ with $B_{\lambda}=M+Z$ has $M^2>0$ and intersects $\mathfrak{a}_{\lambda}$, which contradicts Remark~\ref{subcurve}.
Hence $\{F_{\mu}\}_{\mu \in \mathbb{P}^1}$ has no fixed parts and we complete the proof.
Note that the pencil $\{F_{\mu}\}_{\mu \in \mathbb{P}^1}$ defines a rational map $\psi\colon X\dasharrow \mathbb{P}^1$ as the composition of the birational map $X\leftarrow \mathbb{P}_X(\O_X(B_{\lambda}))$ and $\mathbb{P}_X(\O_X(B_{\lambda}))\hookrightarrow \mathbb{P}_X(V\otimes \O_X)=\mathbb{P}_{k}(V)\times_{k} X\to \mathbb{P}_k(V)=\mathbb{P}^1$, where $V\subset H^0(X,\O_X(B_{\lambda}))$ is the subspace corresponding to $\{F_{\mu}\}_{\mu \in \mathbb{P}^1}$.
\end{proof}

\subsection{Extension theorem on movable divisors}

For a divisor $L$ on a normal complete surface $X$, we say {\em $L$ has a non-trivial numerical linear system} if the numerical linear system $|L|_{\mathrm{num}}:=\{D>0\ |\ D\equiv L\}$ consists of infinite members.
Let us define a similar invariant $q_{X,\infty}$ to $q_X$ as
$$
q_{X,\infty}:=\mathrm{min}\{E^2\ |\ \text{$E$ has a non-trivial numerical linear system with $E^2>0$} \}.
$$
Clearly, we have $q_{X,\infty}\ge q_X$.
The following is a variant of extension theorems:

\begin{theorem}[Extension theorem on movable divisors]\label{movableextnthm}
Let $X$ be a normal proper surface over a field $k$ of characteristic $0$.
Let $D>0$ be an effective divisor on $X$ all of whose prime components have non-trivial numerical linear systems and positive self-intersection numbers.
Let $\varphi\colon D\to \mathbb{P}^1$ be a finite morphism of degree $d$.
We put $q_{X,\infty,d}:=\mathrm{min}\{q_{X,\infty},d\}$.
If $D^2>q_{X, \infty,d}(d/q_{X,\infty,d}+1)^2$ $($resp.\ $D^2=q_{X, \infty}(d/q_{X,\infty}+1)^2$ and $q_{X,\infty}<d$$)$, then there exists a morphism $\psi\colon X\to \mathbb{P}^1$ $($resp.\ or a rational map $\psi\colon X\dasharrow \mathbb{P}^1$ induced by a linear pencil $\{F_{\lambda}\}_{\lambda\in \mathbb{P}^1}$ with no fixed parts and $F_{\lambda}^2=q_{X,\infty}$$)$ such that the restriction of $\psi$ to $D$ is equal to $\varphi$.
\end{theorem}

\begin{proof}
We freely use the notations in the proof of Theorem~\ref{introextnthm}.
First, we note that the set $\mathcal{E}_N$ of prime divisors $C$ on $X$ with $|C|_{\mathrm{num}}$ finite and $DC\le N$ is a finite set for any $N>0$.
If $D^2>q_{X, \infty,d}(d/q_{X,\infty,d}+1)^2$, then the set $\Lambda'$ of rational points $\lambda$ in $\Lambda$ such that the curve $B_{\lambda}$ obtained as in Lemma~\ref{fiberexist} is not negative semi-definite is a finite set.
Indeed, if $B_{\lambda}$ is not negative semi-definite, then the proof of Theorem~\ref{negsemidefnthm} says that $q_{\mathfrak{a}_{\lambda}}<q_{X,\infty,d}$ and there exists a subcurve $0<C_{\lambda}\le B_{\lambda}$ such that  $C_{\lambda}\cap \mathfrak{a}_{\lambda}\neq \emptyset$ and $C_{\lambda}\in \mathcal{E}_N$, where $N:=d+q_{X,\infty, d}$.
Since $D$ has no elements of $\mathcal{E}_N$, the set $\varphi(D\cap (\cup \mathcal{E}_N))$ containing $\Lambda'$ is finite.
Hence there exist infinite rational points $\lambda$ in $\Lambda$ with $B_{\lambda}$ is negative semi-definite and then the proof of Theorem~\ref{introextnthm} works.
The assertion for $D^2=q_{X, \infty}(d/q_{X,\infty}+1)^2$ can also be proved in the same way.
\end{proof}


\section{Applications}
\label{sec:Applications}


In this section, we give examples and applications of extension theorems.
Throughout this section, we work over an arbitrary field $k$ of characteristic $0$ unless otherwise stated.

\begin{example} \label{planecurve}
Let $D\subset \mathbb{P}^2$ be a plane curve of degree $m\ge 2$.
As a corollary of Theorem~\ref{introextnthm} and Theorem~\ref{basepointextnthm}, 
we can show easily that for any (possibly singular, non-reduced or reducible) plane curve $D$ of degree $m$, we obtain $\mathrm{pgon}(D)=m-1$ if $D$ has a smooth rational point or $\mathrm{pgon}(D)=m$ otherwise (the equality follows by the existence of the inner or outer projection from a rational point in $\mathbb{P}^2$).
Moreover, any (pseudo-)gonal pencil of $D$ is obtained by the inner projection from a smooth rational point of $D$ when $\mathrm{pgon}(D)=m-1$.
This generalizes the classical result (Theorem~2.3.1 in \cite{Nam}) for the case of smooth plane curves over $\C$.
For example, if $F_m=\{x^m+y^m-z^m=0\} \subset \mathbb{P}^2_{\Q}$ is the Fermat curve of degree $m\ge 3$, then Fermat's last theorem (\cite{Wil}) says that all base point free pencils of $F_m$ of degree $m-1$ are obtained by the inner projections from trivial rational points with $xyz=0$.
In particular, we have $\mathrm{gon}(F_m)=m-1$ and $\mathrm{gon}_{\infty}(F_m)=\mathrm{s}_2(F_m)=m$.
Moreover, Faltings's theorem (\cite{Fal}) says that any integral plane curve $D$ of degree $m$ and of geometric genus greater than $1$ over a number field has finitely many base point free pencils of degree $m-1$, that is, $\mathrm{pgon}_{\infty}(D)=m$.
\end{example}

More generally, the following can be shown:
\begin{proposition}\label{pgonpic1}
Let $X$ be a normal proper surface of Picard number $1$ over a field of characteristic $0$.
Let $H>0$ be an effective divisor on $X$ $($resp.\ with a non-trivial numerical linear system$)$ such that $H^2$ is minimal among all non-zero effective divisors, that is, $q_X=H^2$ $($resp.\ $q_{X,\infty}=H^2$$)$.
Then for any effective divisor $D>0$ on $X$ $($resp.\ all of prime components have non-trivial numerical linear systems$)$, one of the following holds:

\smallskip

\noindent
$(\mathrm{i})$ $D^2/4\le \mathrm{pgon}(D)\le H^2$, or

\smallskip

\noindent
$(\mathrm{ii})$ $\mathrm{pgon}(D)\ge H(D-H)$ and $\mathrm{pgon}(D)> H^2$.

\smallskip

\noindent
Moreover, if $\mathrm{pgon}(D)=H(D-H)$ holds in the case of $(\mathrm{ii})$, then any pseudo-gonal morphism $\varphi\colon D\to \mathbb{P}^1$ extends to a rational map $\psi\colon X\dasharrow \mathbb{P}^1$ defined by a linear pencil $\{F_{\lambda}\}_{\lambda\in \mathbb{P}^1}$  on $X$ with $F_{\lambda}^2=q_X$ $($resp.\ $F_{\lambda}^2=q_{X,\infty}$$)$.
\end{proposition}

\begin{proof}
We only prove the assertion for $H^2=q_X$.
By using Theorem~\ref{movableextnthm} instead of Theorem~\ref{introextnthm} and Theorem~\ref{basepointextnthm}, we can prove the assertion for $H^2=q_{X,\infty}$ similarly. 
Let $D$ be an effective divisor on $X$ and take a morphism $\varphi\colon D\to \mathbb{P}^1$ of degree $d:=\mathrm{pgon}(D)$.
Note that every prime component of $D$ has positive self-intersection number and $X$ has no base point free pencils since the Picard number of $X$ is $1$.
By the same reason, we can write $D\equiv rH$ for some rational number $r\ge 1$.
Hence by Theorem~\ref{introextnthm}, we have $D^2\le q_{X,d}(d/q_{X,d}+1)^2$.
If $d\le H^2=q_X$, then $q_{X,d}=d$ and thus the above inequality implies $D^2\le 4d$, which is the case (i).
Otherwise, we have $q_{X,d}=q_X$ and so $D^2\le q_{X}(d/q_X+1)^2$ holds.
Hence the case (ii) holds.
The last assertion is due to Theorem~\ref{basepointextnthm}.
\end{proof}

\begin{example}
Let $X$ be a projective cone over a rational normal curve of degree $e$.
It is a proper normal surface of Picard number $1$ and obtained by the contraction $\pi\colon \Sigma_e\to X$ of the $(-e)$-section $C_0$ of the ruling of the $e$-th Hirzebruch surface $p\colon \Sigma_e\to \mathbb{P}^1$.
Let $\Gamma$ be the numerical equivalence class of a fiber of $p$.
Then the push-forward $\overline{\Gamma}:=\pi_{*}\Gamma$ generates $\N(X)_{\Z}$ and computes $q_X=\overline{\Gamma}^2=1/e$.
Let us take a member $D'\in |dC_0+(de+1)\Gamma|$ arbitrarily for $d\ge 1$.
Then $D'$ intersects $C_0$ transversely at one point.
Thus the push-forward $D:=\pi_{*}D'$ is isomorphic to $D'$ as schemes and then $D$ has a base point free pencil of degree $d$ induced by $p|_{D'}\colon D'\to \mathbb{P}^1$.
This pencil can not extend to $X$ because the Picard number of $X$ is $1$.
Since $D\equiv (de+1)\overline{\Gamma}$, we have $D^2=q_X(d/q_X+1)^2$.
This shows that the lower bound of $D^2$ in Theorem~\ref{extnthm} is sharp even when $X$ is singular.
\end{example}

\begin{example}
More generally, we consider a weighted projective space $X=\mathbb{P}(a_0,a_1,a_2)$.
It is a normal complete surface of Picard number $1$ with cyclic quotient singularities.
We may assume that $\mathrm{gcd}(a_i,a_j)=1$ for any $i\neq j$ and $a_0\le a_1\le a_2$.
Then we have $\mathrm{Cl}(X)=\Z\O_X(1)$, $\mathrm{Pic}(X)=\Z\O_X(a_0a_1a_2)$ and $\O_X(1)^2=(a_0a_1a_2)^{-1}$ (for a general reference for weighted projective spaces, see \cite{Dol}).
The Poincar\'{e} series $P_{(a_0,a_1,a_2)}(t)$ of the weighted polynomial ring $k[x_0,x_1,x_2]\cong \bigoplus_{m\in \Z}H^{0}(X,\O_X(m))$ with $\mathrm{deg}(x_i)=a_i$ is $(1-t^{a_0})^{-1}(1-t^{a_1})^{-1}(1-t^{a_2})^{-1}$.
Let $L_i$ $(i=0,1,2)$ denote the weighted projective lines defined by $x_i=0$ in $X$.
Let $m_{(a_0,a_1,a_2)}$ be the minimum of the number $m$ such that the coefficient of $t^{m}$ in $P_{(a_0,a_1,a_2)}(t)$ is greater than $1$, that is, $\O_X(m)$ has a non-trivial numerical linear system.
Then we have $q_X=a_0/a_1a_2$ and $q_{X,\infty}=m_{(a_0,a_1,a_2)}^2/a_0a_1a_2$.
It is easy to see that $m_{(a_0,a_1,a_2)}=a_1$ if $a_0=1$ and $m_{(a_0,a_1,a_2)}>a_1$ if $a_0>1$.
Let $D$ be an effective divisor $D$ of degree $m$ on $X$, that is, $D\in |\O_X(m)|$.
Then its arithmetic genus $p_a(D)$ is equal to the coefficient of $t^{m-a_0-a_1-a_2}$ in $P_{(a_0,a_1,a_2)}(t)$
since $\O_X(K_X)=\O_X(-a_0-a_1-a_2)$, $H^{1}(X,\O_X(n))=0$ for any $n\in \Z$ and the exact sequence \eqref{restex}.
From Proposition~\ref{pgonpic1}, for any effective divisor $D$ of degree $m$ which does not contain $L_i$ with $a_i<m_{(a_0,a_1,a_2)}$, one of the following holds (which is a generalization of the results in \cite{Bal}):

\smallskip

\noindent
$(\mathrm{i})$ $m^2/4a_0a_1a_2\le \mathrm{pgon}(D)\le m_{(a_0,a_1,a_2)}^2/a_0a_1a_2$, or

\smallskip

\noindent
$(\mathrm{ii})$ $\mathrm{pgon}(D)\ge m_{(a_0,a_1,a_2)}(m-m_{(a_0,a_1,a_2)})/a_0a_1a_2$ and $\mathrm{pgon}(D)> m_{(a_0,a_1,a_2)}^2/a_0a_1a_2$.

\smallskip

\noindent
In particular, $D$ is non-hyperelliptic if $m>2a_0a_1a_2/m_{(a_0,a_1,a_2)}+m_{(a_0,a_1,a_2)}$.
Moreover, the above bound is sharp.
Indeed, if $a_0=1$, then the divisor $D$ of degree $m=a_2d+a_1$ defined by the equation
$$
\sum_{i+a_1j=m}c_{ij}x_0^ix_1^j+x_1x_2^d=0
$$
(with $c_{ij}$ general) has pseudo-gonality $d$ which is computed by the morphism defined by the projection $X\dasharrow \mathbb{P}(1,a_1)\cong \mathbb{P}^1; (x_0:x_1:x_2)\mapsto (x_0^{a_1}:x_2)$.
The simplest examples of the case $a_0\ge 2$ are as follows.
Any general effective divisor $D$ of degree $17$ in $X=\mathbb{P}(2,3,5)$ is hyperelliptic  since $p_a(D)=2$.
On the other hand, $m_{(2,3,5)}=5$ and then the lower bound of $\mathrm{pgon}(D)$ in (ii) is $2$.
\end{example}

By using Theorem~\ref{introextnthm} instead of Serrano--Paoletti's extension theorem (\cite{Ser}, \cite{Pao}), 
Marten's result \cite{Mar} can be generalized as follows.

\begin{proposition}
Let $D\in |\alpha C_0+\beta \Gamma|$ be an effective member on the Hirzebruch surface $\Sigma_e$ which has no fibers of the rulings and no negative section $($if $e=0$, we may assume that $\alpha\le \beta$$)$.
Then the following holds.

\smallskip

\noindent
$(1)$ We have $\mathrm{pgon}(D)=\alpha$ unless $e=1$, $\beta=\alpha\ge 2$ and $D$ has at least one smooth rational point, where in this case $D$ is isomorphic to a plane curve of degree $\alpha$ and $\mathrm{pgon}(D)=\alpha-1$.

\smallskip

\noindent
$(2)$ If there exists a morphism $\varphi\colon D\to \mathbb{P}^1$ of degree $\alpha$ which is not obtained as a restriction of the ruling $p\colon \Sigma_e\to \mathbb{P}^1$, 
then $D$ is one of the following:

\begin{enumerate}
\item[(i)]
$e=1$ and $\beta=\alpha\ge 2$. In this case, $D$ is isomorphic to a plane curve of degree $\alpha$ and has infinite base point free pencils of degree $\alpha$.
\item[(ii)]
$e=1$, $\beta=\alpha+1\ge 3$ and $D$ has at least one smooth rational point.
In this case, $D$ is isomorphic to a plane curve of degree $\alpha+1$.
\item[(iii)]
$e=0$ and $(\alpha,\beta)=(2,2)$, that is, $D$ is a member of $|-K_{\Sigma_0}|$.
\item[(iv)]
$e=2$ and $(\alpha,\beta)=(2,4)$, that is, $D$ is a member of $|-K_{\Sigma_2}|$.
\end{enumerate}

\end{proposition}

\begin{proof}
Note that $\alpha>0$ and $\beta\ge e\alpha$ hold by the nefness of $D$.
Clearly we have $\mathrm{pgon}(D)\le \alpha$ by the existence of $p|_D\colon D\to \mathbb{P}^1$.
Let us put $d:=\mathrm{pgon}(D)$.
Since $q_{\Sigma_e}=1$ if $e$ is odd and $q_{\Sigma_e}=2$ if $e$ is even, we have
$$
q_{\Sigma_e,d}\left(\frac{d}{q_{\Sigma_e,d}}+1\right)^2=\left\{\begin{array}{l}
(d+2)^2/2\quad \text{(if $e$ is even and $d\ge 2$)}, \\
(d+1)^2\quad\quad \text{(otherwise)}.
\end{array}
\right.
$$
In particular, $q_{\Sigma_e,d}(d/q_{\Sigma_e,d}+1)^2\le (d+1)^2$ holds.

Suppose that $d<\alpha$. 
Then any morphism $D\to \mathbb{P}^1$ of degree $d$ can not be extended to a morphism on $\Sigma_e$ because any morphism from $\Sigma_e$ to $\mathbb{P}^1$ factors through the ruling $p$ (one of two rulings if $e=0$).
Then by Theorem~\ref{introextnthm}, we have 
$$
\alpha^2(e+2\delta_{e,0})\le D^2\le q_{\Sigma_e,d}(d/q_{\Sigma_e,d}+1)^2\le (d+1)^2\le \alpha^2,
$$
where $\delta_{e,0}=1$ if $e=0$ and $\delta_{e,0}=0$ otherwise.
Thus we have $e=1$ and $\beta=\alpha\ge 2$.
Since $D$ does not contain $C_0$ and $DC_0=0$, it sends isomorphically to a plane curve of degree $\alpha$ by the contraction $\pi\colon \Sigma_1\to \mathbb{P}^2$ of $C_0$.
By the assumption of $d<\alpha$, it follows from Example~\ref{planecurve} that $d=\alpha-1$ and $D$ has a smooth rational point.

Next, we assume that $d=\alpha$ and there exists a morphism $\varphi\colon D\to \mathbb{P}^1$ of degree $\alpha$ which can not be extended to $\Sigma_e$.
By using Theorem~\ref{introextnthm} similarly, we have 
$$
\alpha^2(e+2\delta_{e,0})\le D^2\le q_{\Sigma_e,d}(d/q_{\Sigma_e,d}+1)^2=\left\{\begin{array}{l}
(\alpha+2)^2/2\quad \text{(if $e$ is even and $\alpha\ge 2$)}, \\
(\alpha+1)^2\quad\quad \text{(otherwise)}.
\end{array}
\right.
$$
Hence we conclude that $D$ is one of (i), (ii), (iii) and (iv) by a simple computation.
\end{proof}

The following is a partial generalization of the results in \cite{Haru} and \cite{ChKi}.

\begin{proposition}\label{geomruled}
Let $p\colon X\to C$ be a geometrically ruled irrational surface over a smooth connected projective curve $C$.
Let $C_0$ be a minimal section of $p$ and put $e:=-C_0^2$.
Assume that 
$$
e>\left\{\begin{array}{l}

\mathrm{gon}(C)^2/2+\mathrm{gon}(C)/2+1/8 \quad \text{$($if $e$ is even$)$}, \\
\mathrm{gon}(C)^2\quad\quad\quad\quad\quad\quad\quad\quad\quad\quad \text{$($if $e$ is odd$)$}.
\end{array}
\right.
$$
Then for any effective divisor $D>0$ on $X$ all of whose prime components have positive self-intersection numbers, we have $\mathrm{pgon}(D)=\deg(p|_D)\mathrm{gon}(C)$.
\end{proposition}

\begin{proof}
Note that $\mathrm{pgon}(D)\le \deg(p|_D)\mathrm{gon}(C)$ is clear.
Since the Picard group of $X$ is generated by $C_0$ and $p^{*}\mathrm{Pic}(C)$, 
we can write $D\equiv \alpha C_0+\beta \Gamma$ for some integers $\alpha, \beta \in \Z$,
where $\Gamma$ is the numerically equivalence class of a fiber of $p$ at a ``rational point'', that is, $\Gamma=[k(x):k]^{-1}p^{-1}(x) \in \N(X)_{\Q}$ for $x \in C$.
Note that when $e>0$, any morphism $X\to \mathbb{P}^1$ factors through $p\colon X\to C$.
Indeed, if there exists a morphism $q\colon X\to \mathbb{P}^1$ which does not factor through $p$, 
then a general fiber $F$ of $q$ satisfies $F\Gamma>0$, $FC_0\ge 0$ and $F^2=0$.
When writing $F\equiv aC_0+b\Gamma$, we therefore have $a>0$, $b\ge ae$ and $2b=ae$, which imply $e\le 0$.

Now we assume that $d:=\mathrm{pgon}(D)<\alpha \mathrm{gon}(C)$.
Then any pseudo-gonal morphism $\varphi\colon D\to \mathbb{P}^1$ can not be extended to a morphism $X\to \mathbb{P}^1$ because this morphism would factor though the ruling $p$.
Thus by Theorem~\ref{introextnthm}, we have $D^2\le q_{X,d}(d/q_{X,d}+1)^2$.
Note that when $e$ is even, $q_X\ge 2$ holds since $K_X$ is numerically even.
Hence we have
$$
q_{X,d}\left(\frac{d}{q_{X,d}}+1\right)^2\le \left\{\begin{array}{l}
(d+2)^2/2\quad \text{(if $e$ is even)}, \\
(d+1)^2\quad\quad \text{(if $e$ is odd)}.
\end{array}
\right.
$$
From these, it follows that
$$
-\alpha^2e+2\alpha\beta=D^2\le \left\{\begin{array}{l}
(\alpha \mathrm{gon}(C)+1)^2/2\quad \text{(if $e$ is even)}, \\
\alpha^2\mathrm{gon}(C)^2 \quad\quad\quad\quad \text{(if $e$ is odd)}.
\end{array}
\right.
$$
Thus we have
$$
\beta \le \left\{\begin{array}{l}
(\mathrm{gon}(C)^2+2e)\alpha/4+\mathrm{gon}(C)/2+1/4\alpha \quad \text{(if $e$ is even)}, \\
(\mathrm{gon}(C)^2+e)\alpha/2 \quad\quad\quad\quad\quad\quad\quad\quad\quad\quad\  \text{(if $e$ is odd)}.
\end{array}
\right.
$$
On the other hand, $\beta\ge \alpha e$ holds since $D$ does not contain $C_0$ and hence $DC_0\ge 0$.
Thus we obtain
$$
e\le \left\{\begin{array}{l}

\mathrm{gon}(C)^2/2+\mathrm{gon}(C)/2+1/8 \quad \text{$($if $e$ is even$)$}, \\
\mathrm{gon}(C)^2\quad\quad\quad\quad\quad\quad\quad\quad\quad\quad \text{(if $e$ is odd)},
\end{array}
\right.
$$
which contradicts the assumption in Proposition~\ref{geomruled}.
\end{proof}

\begin{proposition}[Invariance of the pseudo-gonality] \label{invpgon}
Let $X$ be a normal proper surface over a field of characteristic $0$.
Let $L$ be a Weil divisor on $X$ and $|L|_{\mathrm{num}}:=\{D>0\ |\ D\equiv L\}$ the numerical linear system associated to $L$.
If there is a member $D\in |L|_{\mathrm{num}}$ all of whose prime components have positive self-intersection numbers $($resp.\ and non-trivial numerical linear systems$)$ such that the following condition $(C)$ $($resp.\ $(C_{\infty})$$)$ holds:

\smallskip

\noindent
$(C)$ $q_X\le \mathrm{pgon}(D) <\sqrt{q_XL^2}-q_X$ or $\mathrm{pgon}(D)<\min\{q_X,L^2/4\}$.

\smallskip

\noindent
$($resp.\ $(C_{\infty})$ $q_{X,\infty}\le \mathrm{pgon}(D) <\sqrt{q_{X,\infty}L^2}-q_{X,\infty}$ or $\mathrm{pgon}(D)<\min\{q_{X,\infty},L^2/4\}$.$)$

\smallskip

\noindent
Then any member $D\in |L|_{\mathrm{num}}$ all of whose prime components have positive self-intersection numbers has the same pseudo-gonality and it is realized as a restriction $\psi|_D$ of a morphism $\psi\colon X\to \mathbb{P}^1$.
\end{proposition}

\begin{remark}
The condition $(C)$ (resp.\ $(C_{\infty})$) for $d:=\mathrm{pgon}(D)$ is equivalent to $D^2>q_{X,d}(d/q_{X,d}+1)^2=\phi(q_X,d)$ (resp.\ $D^2>q_{X,\infty,d}(d/q_{X,\infty,d}+1)^2=\phi(q_{X,\infty},d)$), which is nothing but the lower bound of $D^2$ in Theorem~\ref{introextnthm}.
Thus in order to confirm $(C)$ (resp.\ $(C_{\infty})$), it suffices to check the following condition $(C)_x$ for some $0<x\le q_X$ (resp.\ $0<x\le q_{X,\infty}$) (Remark~\ref{lowerbound}):

\smallskip

\noindent
$(C)_{x}$ $x\le \mathrm{pgon}(D) <\sqrt{xL^2}-x$ or $\mathrm{pgon}(D)<\min\{x,L^2/4\}$.

\smallskip

\noindent
For example, when $X$ is smooth (resp.\ and $K_X$ is numerically even),  
$\mathrm{pgon}(D) <\sqrt{L^2}-1$ (resp.\ $\mathrm{pgon}(D) <\sqrt{2L^2}-2$) implies the condition $(C)$.
\end{remark}

\begin{proof}
Let $D$ be as in Proposition~\ref{invpgon}.
We may assume that $d:=\mathrm{pgon}(D)$ is minimal among such divisors in $|L|_{\mathrm{num}}$ satisfying the condition $(C)$ (resp.\ $(C_{\infty})$) (see Remark~\ref{lowerbound}).
We take a pseudo-gonal morphism $\varphi\colon D\to \mathbb{P}^1$ of degree $d$.
Then by Theorem~\ref{introextnthm} (resp.\ Theorem~\ref{movableextnthm}), there exists a morphism $\psi\colon X\to \mathbb{P}^1$ such that $\psi|_D=\varphi$.
Thus for any member $D'\in |L|_{\mathrm{num}}$ all of whose prime components have positive self-intersection numbers, the restriction $\psi|_{D'}$ is finite and hence $\mathrm{pgon}(D')\le \deg(\psi|_{D'})=d$, where the last equality follows from $D'\equiv D$.
By the minimality of $\mathrm{pgon}(D)$, we have $\mathrm{pgon}(D')=\mathrm{pgon}(D)$, which completes the proof.
\end{proof}

\begin{example}
Let $\pi\colon X\to \mathbb{P}^2$ be a $K3$ surface of genus $2$ over $\C$, that is, a double cover of $\mathbb{P}^2$ branched along a smooth plane sextic curve.
Then the linear system $|\pi^{*}\O_{\mathbb{P}^2}(3)|$ of genus $10$ curves on $X$ consists of plane sextics (gonality $5$) as general members and bielliptic curves (gonality $4$) as special members (cf.\ p.53 in \cite{DoMo}).
Thus the invariance of the gonality does not hold in this example.
Indeed, one can see that $q_X=2$, $L^2=18$ and then the condition $(C)$ in Proposition~\ref{invpgon} is equivalent to $\mathrm{pgon}(D)=1$, $2$ or $3$, where $L:=\pi^{*}\O_{\mathbb{P}^2}(3)$.
On the other hand, if a $K3$ surface $X$ has such a linear system $|L|$, that is, its general members have gonality $5$ and the invariance of the pseudo-gonality does not hold, then $X$ must contain a genus $2$ curve since $q_X$ has to be $2$ from Proposition~\ref{invpgon}.
\end{example}

\appendix

\section{Intersection theory on a normal surface}
\label{sec:Intersection theory on a normal surface}

In this appendix, we recall some fundamental properties of Weil divisors and Mumford's intersection theory \cite{Mum} on normal surfaces.

Let $X$ be a normal surface (algebraic over a field $k$ or complex analytic) and $f\colon X\to Y$ a proper surjective morphism to a variety $Y$.
Let $\WDiv(X)$ denote the group of Weil divisors on $X$. 
Let $\WDiv(X/Y)$ denote the subgroup of $\WDiv(X)$ consisting of $f$-exceptional Weil divisors on $X$.
Let $\CDiv(X)$ denote the group of Cartier divisors on $X$, that is, $\CDiv(X)=H^{0}(X,\mathcal{K}_{X}^{*}/\O_{X}^{*})$, where $\mathcal{K}_{X}^{*}$ is the sheaf of invertible rational (or meromorphic) functions on $X$.
For a Cartier divisor $D$ on $X$ and an $f$-exceptional integral curve $C$ on $X$, the intersection number of $D$ and $C$ is defined as $DC:=\mathrm{deg}(\nu^{*}\O_X(D))$, where $\nu \colon \widetilde{C}\to C$ is the normalization of $C$.
Note that the intersection number $DC$ depends on the base field $k$ (see also Definition~\ref{degreedef}).
By this intersection pairing, we can define the intersection form $\CDiv(X)\times \WDiv(X/Y)\to \Z$.
On the other hand, since $X$ is normal, the natural cycle map
$$
\CDiv(X)\to \WDiv(X);\ D \mapsto \sum_{C:\text{prime divisor}}\mathrm{ord}_{C}(D)C
$$
is injective.

\begin{definition}[Mumford's intersection form \cite{Mum}] 
We can define the extended intersection form 
$$
\WDiv(X)\times \WDiv(X/Y)\to \Q
$$
as follows.
Let $\pi\colon X'\to X$ be a resolution of singularities (for existence, see \cite{Lip}).
For a Weil divisor $D$ on $X$, we define the {\em Mumford pull-back} $\pi^{*}D$ as a $\Q$-divisor $\widehat{D}+\sum_{i}a_{i}E_{i}$, where $\widehat{D}$ is the proper transform of $D$ and the rational coefficients $a_{i}$ of the $\pi$-exceptional prime divisors $E_i$ are uniquely determined by the equations $\widehat{D}E_{j}+\sum_{i}a_{i}E_{i}E_{j}=0$ for each $j$ since $\pi$-exceptional divisors $E_i$ form a negative definite matrix (cf.\ \cite{Mum}, \cite{Ish}).
For a Weil divisor $D$ and an $f$-exceptional Weil divisor $E$, we define the intersection number $DE$ as $DE:=\pi^{*}D\pi^{*}E$.
In general, the Mumford pull-back is defined for any proper birational morphism $\pi\colon X'\to X$ from a normal surface $X'$ as a group homomorphism $\pi^{*}\colon \WDiv(X)\to \WDiv(X')\otimes\Q$ by the same way.
\end{definition}

\begin{definition}[Numerical class groups]
Let
$$
N^{1}(X/Y)_{\Z}:=\WDiv(X)/\{D\in \WDiv(X)\ |\ \text{$DE=0$ for $\forall E\in \WDiv(X/Y)$}\}
$$
and
$$
N_{1}(X/Y)_{\Z}:=\WDiv(X/Y)/\{E\in \WDiv(X/Y)\ |\ \text{$DE=0$ for $\forall D\in \WDiv(X)$}\}
$$
denote the groups consisting of the numerical equivalence classes of Weil divisors and of $f$-exceptional Weil divisors, respectively.
Let us denote 
$$
N^{1}(X/Y)_{\Q}:=N^{1}(X/Y)_{\Z}\otimes \Q,\quad N^{1}(X/Y):=N^{1}(X/Y)_{\Z}\otimes \R
$$
and
$$
N_{1}(X/Y)_{\Q}:=N_{1}(X/Y)_{\Z}\otimes \Q,\quad N_{1}(X/Y):=N_{1}(X/Y)_{\Z}\otimes \R.
$$
Note that these notations are not standard because we use Mumford's intersection form instead of the usual intersection form.
When $\dim(Y)=0$, we denote by $N(X)=N^{1}(X)=N_{1}(X)$ instead of $N^{1}(X/Y)$ and $N_{1}(X/Y)$.
The real vector spaces $\N^{1}(X/Y)$ and $\N_{1}(X/Y)$ are finite dimensional (\cite{Kle}, Ch. IV, Section 1, Proposition 4) and the intersection form induces a non-degenerate bilinear form $\N^{1}(X/Y)\times \N_{1}(X/Y)\to \R$.
For a proper birational morphism $\pi\colon X'\to X$, the Mumford pull-back induces injective homomorphisms $\pi^{*}\colon \N^{1}(X/Y)\into \N^{1}(X'/Y)$ and $\pi^{*}\colon \N_{1}(X/Y)\into \N_{1}(X'/Y)$ and we can write
\begin{equation}\label{numrelation}
\N^{1}(X'/Y)_{\Q}=\pi^{*}\N^{1}(X/Y)_{\Q}\oplus (\oplus_{i}\Q E_{i}),\quad \N_{1}(X'/Y)_{\Q}=\pi^{*}\N_{1}(X/Y)_{\Q}\oplus (\oplus_{i}\Q E_{i}),
\end{equation}
where $E_i$ are $\pi$-exceptional prime divisors.
\end{definition}

\begin{definition}[Numerical properties of divisors]
Let $\Eff^{1}(X/Y)$ (resp.\ $\Eff_{1}(X/Y)$) denote the subgroup of $N^{1}(X/Y)$ (resp.\ $N_{1}(X/Y)$) generated by numerical equivalence classes of effective (resp.\ $f$-exceptional effective) divisors on $X$.
We define $\PE^{1}(X/Y)$ (resp.\ $\PE_{1}(X/Y)$) as the closure of $\Eff^{1}(X/Y)$ in $\N^{1}(X/Y)$ (resp.\ $\Eff_{1}(X/Y)$ in $\N_{1}(X/Y)$).
An $\R$-divisor (resp.\ $f$-exceptional $\R$-divisor) $D$ on $X$ is said to be {\em $f$-pseudo-effective} (resp.\ {\em $f$-exceptional pseudo-effective}) if the numerical equivalence class of $D$ belongs to $\PE^{1}(X/Y)$ (resp.\ $\PE_{1}(X/Y)$).

An $\R$-divisor (resp.\ $f$-exceptional $\R$-divisor) $D$ on $X$ is called {\em $f$-nef} (resp.\ {\em nef}) if $DC\ge 0$ (resp.\ $CD\ge 0$) holds for any $f$-exceptional curve (resp.\ any curve) $C$ on $X$.
Let $\Nef^{1}(X/Y)$ (resp.\ $\Nef_{1}(X/Y)$) be the subgroup of $N^{1}(X/Y)$ (resp.\ $N_{1}(X/Y)$) generated by numerical equivalence classes of $f$-nef divisors on $X$ (resp.\ nef divisors on $X$).
Then $\PE^{1}(X/Y)$ and $\Nef^{1}(X/Y)$ (resp.\ $\PE_{1}(X/Y)$ and $\Nef_{1}(X/Y)$) are closed cones in $\N^{1}(X/Y)$ (resp.\ $\N_{1}(X/Y)$) and we have the duality of cones (cf.\ \cite{Laz} Proposition~1.4.28)
$$
\PE^{1}(X/Y)=\Nef_{1}(X/Y)^{*},\quad \PE_{1}(X/Y)=\Nef^{1}(X/Y)^{*}.
$$

Recall that for an $n$-dimensional normal complete variety $X$, a $\Q$-Cartier $\Q$-divisor $D$ on $X$ is {\em big} if there is a positive number $\alpha$ such that $h^{0}(X,\O_{X}(mD))\ge \alpha m^n$ for all sufficiently large and divisible integer $m$ (cf.\ \cite{Laz} Definition~2.2.1).
For a proper morphism $f\colon X\to Y$ between varieties, a $\Q$-Cartier $\Q$-divisor $D$ on $X$ is called $f$-big if the restriction of $D$ to the normalization of the generic fiber of the Stein factorization of $f\colon X\to f(X)$ is big.
For a proper surjective morphism $f\colon X\to Y$ from a normal surface $X$ to a variety $Y$, a Weil divisor $D$ on $X$ is called $f$-big if the pull-back $\pi^{*}D$ is $(f\circ \pi)$-big for some (or equivalently, any) resolution $\pi \colon X'\to X$.
Let $\Big(X/Y)$ be the cone in $\N^{1}(X/Y)$ generated by $f$-big divisors on $X$.
An $\R$-divisor $D$ on $X$ is called $f$-big if the numerical class of $D$ is contained in $\Big(X/Y)$. 

An $\R$-divisor $D$ is called {\em numerically $f$-ample} if $DC>0$ for any $C \in \PE_{1}(X/Y)\setminus \{0\}$.
Let $\Amp(X/Y)$ be the cone in $\N^{1}(X/Y)$ consisting of numerical equivalence classes of numerically $f$-ample divisors.
\end{definition}

\begin{lemma} \label{pushpull}
Let $\pi\colon X'\to X$ be a proper birational morphism between normal surfaces and $f':=f\circ \pi\colon X'\to Y$.

\smallskip

\noindent
$(1)$ The push-forward of any $f'$-pseudo-effective $($resp.\ $f'$-nef, $f'$-big, numerically $f'$-ample, $f'$-exceptional pseudo-effective, $f'$-exceptional nef$)$ $\R$-divisor on $X'$ by $\pi$ is $f$-pseudo-effective $($resp.\ $f$-nef, $f$-big, numerically $f$-ample, $f$-exceptional pseudo-effective, $f$-exceptional nef$)$. 

\smallskip

\noindent
$(2)$ The pull-back of any $f$-pseudo-effective $($resp.\ $f$-nef, $f$-big, $f$-exceptional pseudo-effective, $f$-exceptional nef$)$ $\R$-divisor on $X$ by $\pi$ is $f'$-pseudo-effective $($resp.\ $f'$-nef, $f'$-big, $f'$-exceptional pseudo-effective, $f'$-exceptional nef$)$. 
\end{lemma}

\begin{proof}
The claims about nefness in (1) and (2) and numerically ampleness in (1) are easy to check.
The claims about pseudo-effectiveness in (1) and (2) follows from the fact that the effectivity preserves under the push-forward and the pull-back, which are continuous maps.
The claim about bigness in (2) follows by definition.
The claim about bigness in (1) with $\dim(Y)\ge 1$ is trivial since the generic fibers of $f$ and $f'$ are isomorphic.
We assume that $\dim(Y)=0$. Let $D$ be a big divisor on $X'$.
Then, by Lemma~\ref{proj}, we have
$$
\pi_{*}\O_{X'}(\llcorner \pi^{*}\pi_{*}mD \lrcorner+E)\cong \O_{X}(m\pi_{*}D)
$$
for any $\pi$-exceptional effective divisor $E$ on $X'$ and any positive interger $m$.
Since $D$ is big and $mD\le \llcorner \pi^{*}\pi_{*}mD \lrcorner+E$ for sufficiently effective $E$, it follows that $\pi_{*}D$ is big.
\end{proof}

\begin{lemma}\label{inclusion}
The following hold.

\smallskip

\noindent
$(1)$ $\NA(X/Y)\subset \Big(X/Y)\subset \Eff^{1}(X/Y) \subset \PE^{1}(X/Y)$.

\smallskip

\noindent
$(2)$ $\NA(X/Y) \subset \Nef^{1}(X/Y)\subset \PE^{1}(X/Y)$.
\end{lemma}

\begin{proof}
Note that $\Big(X/Y)=\Eff^{1}(X/Y)=\PE^{1}(X/Y)=\N^{1}(X/Y)$ for $\dim(Y)=2$.
For $\dim(Y)=0$, the claim (1) can be shown straightforward.
For $\dim(Y)=1$, the claim (1) follows from the fact that $f_{*}\O_{X}(mD)\otimes \O_{Y}(\mathfrak{a})$ is globally generated for any $f$-big divisor $D$ on $X$, $m\gg 0$ and a divisor $\mathfrak{a}$ on $Y$ of sufficiently high degree.
The claim (2) in $\dim(Y)=0$ follows from the fact that $DD'\ge 0$ holds for any nef divisors $D$, $D'$ on $X$ (cf.\ \cite{Kle}).
We assume $\dim(Y)=1$. Let $E$ be an $f$-exceptional nef $\R$-divisor on $X$.
Since $FE=0$ for a fiber $F$ of $f$, we can write $E=aF$ in $N_{1}(X/Y)$ for some $a\in \R$.
If there exists an $f$-horizontal curve on $X$, we have $a\ge 0$ by the nefness of $E$.
Then any $f$-nef divisor $D$ is $f$-pseudo-effective since $DF\ge 0$.
If there exist no $f$-horizontal curves on $X$, then $N^{1}(X/Y)$ is generated by $f$-exceptional curves.
Thus we have $F=0$  in $N_{1}(X/Y)$ and then $\Nef_{1}(X/Y)=0$.
This implies $\PE^{1}(X/Y)=N^{1}(X/Y)$, whence (2) follows.
\end{proof}

\begin{lemma}\label{numample}
The following hold.

\smallskip

\noindent
$(1)$ If $\dim(Y)=0$, an $\R$-divisor $D$ on $X$ is numerically ample if and only if $D^2>0$ and $DC>0$ for any curve $C$ on $X$, that is, $D$ satisfies Nakai-Moishezon's condition.


\smallskip

\noindent
$(2)$ A $\Q$-divisor $D$ on $X$ is $f$-ample if and only if $D$ is numerically $f$-ample and $\Q$-Cartier.

\smallskip

\noindent
$(3)$ Assume $\dim(Y)=0$. 
Then $\NA(X)\neq \emptyset$ if and only if $X$ is algebraic or Moishezon.
In this case, the signature of the intersection form on $N(X)$ is $(1,\rho-1)$, where $\rho:=\dim N(X)$.
Otherwise, the intersection form on $N(X)$ is negative definite.

\smallskip

\noindent
$(4)$ $\NA(X/Y)=\emptyset$ if and only if $\Big(X/Y)=\emptyset$.
\end{lemma}

\begin{proof}
If $\dim(Y)=0$ and $D$ is numerically ample, then $D^2>0$ holds since $D$ is pseudo-effective and not numerically trivial.
Thus $D$ satisfies Nakai-Moishezon's condition.
Conversely, assume that $D$ satisfies Nakai-Moishezon's condition and $DC=0$ for some pseudo-effective divisor $C$ on $X$.
Then the Hodge index theorem implies that $C$ is numerically trivial.
Hence $D$ is numerically ample, which completes the proof of (1).

The claim (2) with $\dim(Y)>0$ is easy to prove. 
The claim (2) with $\dim(Y)=0$ follows from (1) and the Nakai-Moishezon criterion for ampleness (cf.\ \cite{Kle} Chapter 3).

We assume $\dim(Y)=0$.
Note that a resolution of $X$ is projective if and only if $X$ is algebraic or Moishezon.
In this case, Lemma~\ref{pushpull} (1) implies $\NA(X)\neq \emptyset$.
Conversely, let $D$ be a numerically ample divisor on $X$ and take a resolution $\pi\colon X'\to X$.
Then $m\pi^{*}D-Z$ is numerically ample for sufficiently large $m\gg0$, where $Z$ is the fundamental cycle of $\pi$ on $X'$.
Since $X'$ is regular, it is $\Q$-Cartier.
Then it is ample by (2), whence $X'$ is projective.
The rest of claim (3) follows from \eqref{numrelation}.

The claim (4) with $\dim(Y)=0$ follows from (3).
Indeed, if there exists a big divisor $D$ on $X$, then the movable part $M$ of $mD$, $m\gg 0$ has $M^2>0$.
Assume that $\dim(Y)=1$.
Let $D$ be an $f$-big divisor on $X$.
Then $DF>0$ holds, where $F$ is a fiber of $f$.
For any reducible fiber $F_i$ of $f$, let $Z_i$ be the fundamental cycle the support on which is the union of irreducible components $C$ of $F_i$ with $DC\le 0$.
Then $mD-\sum_{i}Z_i$ is numerically $f$-ample for $m\gg 0$, which implies (4).
The claim (4) with $\dim(Y)=2$ follows from the fact that $\NA(X/Y)$ contains the fundamental cycle of $f$ and $\Big(X/Y)=N^{1}(X/Y)$.
\end{proof}

\begin{lemma}[Kodaira's lemma. cf.\ \cite{Laz} Corollary~2.2.7]\label{kodairalem}
Let $D$ be an $\R$-divisor on $X$ and assume that $\Amp(X/Y)\neq \emptyset$.
Then the following are equivalent.

\smallskip

\noindent
$(1)$ $D$ is $f$-big.

\smallskip

\noindent
$(2)$ $D=A+E$ for some numerically $f$-ample $\Q$-divisor $A$ and effective $\R$-divisor $E$ on $X$.

\smallskip

\noindent
$(3)$ There exists a positive number $\alpha>0$ such that $\rank f_{*}\O_{X}(\llcorner mD \lrcorner)\ge \alpha m^{\dim X-\dim Y}$ for sufficiently large and divisible $m\gg 0$.
\end{lemma}

\begin{proof}
If $\dim(Y)=2$, all the conditions (1), (2) and (3) are satisfied by any $\R$-divisor $D$ on $X$.
Then we may assume $\dim(Y)\le 1$.
By Lemma~\ref{inclusion} (1), numerically $f$-ample divisors are $f$-big. 
Thus (2) implies (1) and (3).
Conversely, we will show (2) under the assumption of (1).
We may assume $D$ is a $\Z$-divisor since a numerically $f$-ample $\R$-divisor can be decomposed into the sum of a numerically $f$-ample $\Q$-divisor and an effective $\R$-divisor.
By taking a resolution of $X$, we may assume that $X$ is regular.
Take an $f$-ample effective divisor $A$ on $X$ and consider the exact sequence
\begin{equation} \label{bigex}
0\to \O_{X}(mD-A)\to \O_{X}(mD)\to \O_{A}(mD)\to 0.
\end{equation}
Note that $A$ maps onto $Y$.
Since $D$ is $f$-big and $A$ is $1$-dimensional, we have $f_{*}\O_{X}(mD-A)\neq 0$ for $m\gg 0$.
Thus we can write $mD=A+E$ by taking $E\in |mD-A|$ if $\dim(Y)=0$, or taking $E\in |mD-A+f^{*}\mathfrak{a}|$ for a divisor $\mathfrak{a}$ on $Y$ of sufficiently high degree and replacing $A-f^{*}\mathfrak{a}$ by $A$ if $\dim(Y)=1$, whence (2) follows.
Replacing $mD$ with $\llcorner mD \lrcorner$ in \eqref{bigex} and tracing the same proof as above, we can prove that (3) implies (2).
\end{proof}

\begin{lemma}\label{open}
The following hold.

\smallskip

\noindent
$(1)$ For any divisor $B$ and any numerically $f$-ample $($resp.\ $f$-big$)$ divisor $D$ on $X$, the divisor $B+nD$ is numerically $f$-ample $($resp.\ $f$-big$)$ for sufficiently large $n\gg0$.
In particular, $\Amp(X/Y)$ $($resp.\ $\Big(X/Y)$$)$ is an open subset of $N^{1}(X/Y)$.

\smallskip

\noindent
$(2)$ If $\Amp(X/Y)\neq \emptyset$, 
then an $\R$-divisor $D$ is $f$-nef $($resp.\ $f$-pseudo-effective$)$ if and only if 
$D+A$ is numerically $f$-ample $($resp.\ $f$-big$)$ for any numerically $f$-ample $\R$-divisor $A$ on $X$.

\smallskip

\noindent
$(3)$ If $\Amp(X/Y)\neq \emptyset$, it follows that
$$
\overline{\Amp(X/Y)}=\Nef^{1}(X/Y),\quad \Amp(X/Y)=\Int(\Nef^{1}(X/Y)),
$$
and
$$
\overline{\Big(X/Y)}=\PE^{1}(X/Y),\quad \Big(X/Y)=\Int(\PE^{1}(X/Y)).
$$
\end{lemma}

\begin{proof}
The claim (1) for numerically ampleness is easy to prove. 
The claim (1) for bigness follows from Lemma~\ref{kodairalem}. 

The claim (2) for the nef case is easy to show.
We will prove (2) for the pseudo-effective case.
Let $D$ be an $\R$-divisor on $X$ such that $D+A$ is $f$-big for any numerically $f$-ample $\R$-divisor $A$ on $X$.
Take numerically $f$-ample $\Q$-divisors $A_n$ on $X$ such that $A_n\rightarrow 0$ in $\N^{1}(X/Y)$ $(n\rightarrow \infty)$.
Then $D+A_n$ is $f$-big and in particular effective.
Hence $D$ is $f$-pseudo-effective.

Let $D$ be an effective divisor on $X$. Then $D+A$ is $f$-big for any numerically $f$-ample divisor $A$ on $X$ from Lemma~\ref{kodairalem}.
Hence, it suffices to show that the property of $D$ that $D+A$ is $f$-big for any numerically $f$-ample $\R$-divisor $A$ is a closed condition.
Let $D$ be an $\R$-divisor on $X$ and assume that there exists a numerically $f$-ample $\R$-divisor $A$ such that $D+A$ is not $f$-big.
In particular, $D+(1/2)A$ is not $f$-big.
We consider $D+\varepsilon B+(1/2)A$ for any divisor $B$ and $\varepsilon >0$.
Since $A$ is numerically $f$-ample, $(1/2)A-\varepsilon B$ is numerically $f$-ample for sufficiently small $\varepsilon$ from (1), which implies that $D+\varepsilon B+(1/2)A$ is not $f$-big.
Hence (2) follows.

The claim (3) follows from (1), (2) and the standard argument of the topology.
Note that $\Big(X/Y)\neq \emptyset$ from Lemma~\ref{numample} (4).
\end{proof}

\begin{corollary} \label{fib}
If $\dim(Y)=1$, then an $\R$-divisor $D$ on $X$ is $f$-pseudo-effective $($resp.\ $f$-big$)$
 if and only if $DF\ge 0$ $($resp.\ $DF>0$$)$ for a fiber $F$ of $f$.
\end{corollary}

\begin{proof}
This follows from Lemmas~\ref{kodairalem} and \ref{open}.
\end{proof}

Let us assume that $\dim(Y)=0$ and $X$ is algebraic or Moishezon.
Let $\Cpp(X)$ be the cone in $\N(X)$ consisting of numerical equivalence classes of $\R$-divisors $D$ on $X$ with $D^2>0$ and $DF>0$ for some nef divisor $F$ on $X$.


\begin{lemma}\label{positive}
Let $D$ be an $\R$-divisor on $X$ with $D^2> 0$ $($resp.\ $D^{2}\ge 0$$)$. 
The following are equivalent.

\smallskip

\noindent
$(1)$ $D$ is big $($resp.\ pseudo-effective$)$.

\smallskip

\noindent
$(2)$ There exists a nef and big divisor $F$ such that $DF>0$ $($resp.\ $DF\ge 0$$)$.

\smallskip

\noindent
$(3)$ There exists a nef divisor $F$ on $X$ such that $DF>0$ $($resp.\ or $D\equiv 0$$)$.

\smallskip

\noindent
In particular, $\Big(X)\cap \{D\in N(X)\ |\ D^{2}>0\}=\Cpp(X)$ holds.
\end{lemma}

\begin{proof}
We may assume that $D$ is a $\Z$-divisor.
First, we show the equivalence for bigness.
Assume that $D$ is big. 
Then $DF>0$ holds for a numerically ample divisor $F$ on $X$ (note that $\NA(X)\neq \emptyset$ from Lemma~\ref{numample}~(4)).
Then $D$ satisfies (2).
The implication from (2) to (3) is clear.
We assume that $D$ satisfies the condition (3).
By taking a resolution of $X$, we may assume that $X$ is regular by Lemma~\ref{pushpull} (1).
By the Riemann-Roch theorem on $X$ and the condition (3), we have $\dim |mD|\to \infty$ ($m\to \infty$).
In particular, $D$ is pseudo-effective and then $\Cpp(X)$ is contained in $\PE^{1}(X)$.
Since $\Cpp(X)$ is open in $N(X)$, we have $\Cpp(X)\subset \Big(X)$ by Lemma~\ref{open} (3).
Hence $D$ satisfies (1).

Next, we show the equivalence for pseudo-effectiveness.
If $D\equiv 0$, then the claim is trivial.
Thus we may assume that $D$ is numerically non-trivial.
The proof that (1) implies (2) (or (3)) is similar to the big case.
The implication from (2) (or (3)) to (1) follows from the equivalence for bigness by replacing $D$ with $D+A_n$ for numerically ample $\Q$-divisors $A_n$ with $A_n\to 0$ in $N(X)$ ($n\to \infty$).
\end{proof}

\begin{corollary} \label{reflection}
Let $D$ be a pseudo-effective $\R$-divisor on $X$ with $D=P+N$ the Zariski decomposition in Theorem~\ref{Zardecomp}.
Then the following hold.

\smallskip

\noindent
$(1)$ $D$ is nef and big if and only if $D$ is nef and $D^2>0$.
In particular, $D$ is big if and only if $P^2>0$.

\smallskip

\noindent
$(2)$ If further assume that $D^2>0$.
Then $D':=P-N$ is big and $D'^{2}>0$.
In particular, this operation gives a map $\Cpp(X)\to \Cpp(X);\ D\mapsto D'$.
\end{corollary}

\begin{proof}
The claim (1) follows from Lemmas~\ref{kodairalem} and \ref{positive}.
Suppose $D^2>0$. 
Note that $P$ is nef and big since $D$ is big.
Hence the first half of the claim (2) follows from Lemma~\ref{positive}, $D'P=(P-N)P=P^2>0$ and $D'^{2}=P^2+N^2=D^2>0$.
Since the negative part $N$ of $D$ is determined by its numerical equivalence class, this operation defines a self-map on $\Cpp(X)$.
\end{proof}


The following lemma is purely linear-algebraic and easy to show.
\begin{lemma} \label{linalg}
Let $A=\cup_{i=1}^{n}A_{i}$ be the irreducible decomposition of a connected complete curve $A$ on $X$.
Then the following are equivalent.

\smallskip

\noindent
$(1)_{<0}$ $($resp.\ $(1)_{=0}$, $(1)_{>0})$ There is an effective divisor $Z=\sum_{i=1}^{n}a_iA_i$ with $a_i\in \Q_{>0}$ such that $A_{j}Z<0$ $($resp.\ $A_{j}Z=0,\ A_{j}Z>0)$ for any $j$.

\smallskip

\noindent
$(2)_{<0}$ $($resp.\ $(2)_{=0}$, $(2)_{>0})$ The matrix $(A_{i}A_{j})_{ij}$ is negative definite $($resp.\ negative semi-definite and not negative definite, not negative semi-definite$)$.
\end{lemma}

\begin{proof}
It is clear that $(1)_{>0}$ implies $(2)_{>0}$ because $Z^{2}>0$.
We assume $(1)_{<0}$ (resp.\ $(1)_{=0}$).
In order to show $(2)_{<0}$ (resp.\ $(2)_{=0}$), it suffices to prove that for any non-zero $\Q$-divisor $B=\sum_{i=1}b_iA_i$, it holds $B^{2}<0$ (resp.\ $B^{2}\le 0$). 
One can check that
$$
B^2\le  \sum_{i}\frac{b_{i}^{2}}{a_i}A_iZ<0\ (\text{resp.\ $\le 0$})
$$
by a direct computation and $A_{i}A_{j}\ge 0$ for $i\neq j$. 

Next, we assume $(2)_{<0}$ (resp.\ $(2)_{=0}$, $(2)_{>0}$).
For an $\R$-divisor $B=\sum_{i=1}^{n}b_{i}A_{i}$, we write $B_{+}:=\sum_{i=1}^{n}|b_{i}|A_{i}$.
Then one can check $B^2\le B_{+}^{2}$ since $A_{i}A_{j}\ge 0$ for $i\neq j$.
Let $\alpha$ be the maximal eigenvalue of $(A_{i}A_{j})_{ij}$ and $V_{\alpha}\subset \R^{n}$ the corresponding eigenspace.
Note that $\alpha$ is not necessarily rational.
For a non-zero vector $b=\t (b_i\ldots,b_n)\in V_{\alpha}$, we put $B:=\sum_{i=1}^{n}b_{i}A_{i}$.
Since $A_{j}B=\alpha b_{j}$ for any $j$, we have $B^2=\alpha\sum_{j=1}|b_j|^2$.
On the other hand, it follows that $B_{+}^2=\sum_{i,j}(A_{i}A_{j})|b_{i}||b_{j}|\le \alpha\sum_{i=1}^{n}|b_{i}|^{2}=B^{2}$, whence $B^{2}=B_{+}^{2}$ and $|b|=\t (|b_1|,\ldots,|b_n|)\in V_{\alpha}$ hold.
Since $A$ is connected, we can check that $b_{i}\neq 0$ holds for any $i$.
Then we may assume that each element of $b$ is positive.
Thus we have $A_{j}B=\alpha b_{j}<0\ \text{(resp.\ $=0$, $>0$)}$ by the condition $(2)_{<0}$ (resp.\ $(2)_{=0}$, $(2)_{>0}$).
In the case of $(2)_{=0}$ (that is, $\alpha=0$), we can take all the coefficients $b_i$ are rational because $\alpha$ is rational. Then $Z:=B$ satisfies $(1)_{=0}$.
In the case of $(2)_{<0}$ (resp.\  $(2)_{>0}$), then we can shift the coefficients $b_{i}$ to rational numbers preserving the condition $A_{j}B=A_{j}(\sum_{i=1}^{n}b_{i}A_{i})<0$ (resp.\ $>0$) for each $j$.
Hence we may assume all the $b_i$ are positive rational.
Then $Z:=B$ satisfies $(1)_{<0}$ (resp.\ $(1)_{>0}$).
\end{proof}


\section{Morphisms on proper curves}
\label{sec:Morphisms on proper curves}


In this appendix, we collect definitions and basic properties of some invariants of proper curves over a field $k$.

\begin{definition}[Degree for finite morphisms]
Let $f\colon X\to Y$ be a finite morphism between schemes with the scheme-theoretic image $f(X)$ irreducible.
Then let us define the {\em degree} of $f$ as follows.
First, we may assume that $Y=f(X)$.
Let us recall that if $X$ is integral, the degree of $f$ means the degree of the field extension $[K(X):K(Y)]$ of the function fields $K(Y)\subset K(X)$.
If $X$ is irreducible, the degree of $f$ is defined as $\deg(f):=(m_X/m_Y)\deg(f_{\red})$,
where $f_{\red}\colon X_{\red}\to Y_{\red}$ is the induced morphism between the maximal reduced subschemes and for an irreducible scheme $X$ with the generic point $\eta_X$, $m_X:=\length(\O_{X,\eta_X})$ denotes the geometric multiplicity of $X$.
Note that if $f$ is flat at the generic point, then we have $m_{f^{-1}(Y_{\red})}=m_X/m_Y$ and $\deg(f)$ coincides with the degree of $f|_{f^{-1}(Y_{\red})}$. 
If $Y$ is irreducible, then we define $\deg(f):=\sum_{i}\deg(f|_{X_i})$, where $X=\cup_{i}X_i$ is the irreducible decomposition of $X$ and $f|_{X_i}\colon X_i\to Y$ is the restriction of $f$ to $X_i$ (note that the scheme-theoretic image of $f|_{X_i}$ is not necessarily coincide with $Y$).
Note that if $f$ is flat and of finite presentation, $\deg(f)$ coincides with the rank of the locally free $\O_Y$-module $f_{*}\O_X$.

Let $f\colon X\to Y$ be a finite morphism of schemes such that $X$ has finitely many irreducible components.
Then the {\em maximum degree} of $f$ is defined to be
$\maxdeg(f):=\max_{i}\{\deg(f|_{f^{-1}(Y_i)})\}$, where $f(X)=\cup_{i}Y_i$ is the irreducible decomposition of the scheme-theoretic image of $f$.
\end{definition}

The following lemma easily follows by definition.
\begin{lemma} \label{findeg}
Let $f\colon X\to Y$ and $g\colon Y\to Z$ be two scheme-theoretically surjective finite morphisms between schemes.

\smallskip

\noindent
$(1)$ If $X$ is irreducible, then $\deg(g\circ f)=\deg(f)\deg(g)$ holds.

\smallskip

\noindent
$(2)$ If $Z$ is integral, then $\deg(g\circ f)\le \maxdeg(f)\deg(g)$ holds.

\end{lemma}

\begin{definition}[Degree for line bundles on curves] \label{degreedef}
Let us recall that for a regular proper curve $C$ over a field $k$ and a line bundle $\mathcal{L}\cong \O_C(\sum_{i}a_ix_i)$ on $C$, the degree of $\mathcal{L}$ is defined to be $\deg(\mathcal{L}):=\sum_{i}a_i[k(x_i):k]$.
Note that the degree depends on the base field $k$.
For a line bundle $\mathcal{L}$ on an integral proper curve $C$, the degree of $\mathcal{L}$ is defined as $\deg(\mathcal{L}):=\deg(\nu^{*}\mathcal{L})$, where $\nu\colon C'\to C$ is the normalization.
Let $C$ be a proper curve over a field $k$ and $\mathcal{L}$ a line bundle on $C$.
Then the degree of $\mathcal{L}$ is defined to be $\deg(\mathcal{L}):=\sum_{i}m_i\deg(\mathcal{L}|_{C_{i,\red}})$,
where $C=\cup_{i}C_i$ is the irreducible decomposition of $C$ and $m_i$ is the geometric multiplicity of $C_i$.
The definition of degree is equivalent to Definition~1.4 in \cite{Ful}
and Definition~33.43.1 in \cite{Sta}, that is, $\deg(\mathcal{L})=\chi(\mathcal{L})-\chi(\O_C)$ (cf.\ Lemma~33.43.6 in \cite{Sta}).
For a Cartier divisor $\mathfrak{a}$ on $C$, the degree of $\mathfrak{a}$ is defined as $\deg(\mathfrak{a}):=\deg(\O_{C}(\mathfrak{a}))$.
Note that if $\mathfrak{a}$ is effective, then $\deg(\mathfrak{a})=h^{0}(\O_{\mathfrak{a}})$ holds from the exact sequence
$$
0\to \O_C\to \O_C(\mathfrak{a})\to \O_C(\mathfrak{a})|_{\mathfrak{a}}\to 0.
$$
\end{definition}

The following lemma is easy to show by definition.

\begin{lemma} \label{lbdeg}
Let $\varphi \colon D\to C$ be a scheme-theoretically surjective finite morphism of proper curves over a field $k$ such that $D$  is irreducible or $C$ is integral.
Let $\mathcal{L}$ be a line bundle on $C$.
Then $\deg(\varphi^{*}\mathcal{L})=\deg(\varphi)\deg(\mathcal{L})$ holds.
\end{lemma}

\begin{definition}[Rational maps]
Let us recall that a {\em rational map} $f\colon X\dasharrow Y$ between algebraic schemes means the equivalence class of a pair $(U,f_U\colon U\to Y)$ such that $U$ is a dense open subset of $X$ and $f_U\colon U\to Y$ is a morphism, where two pair $(U,f_U)$ and $(V,f_V)$ are equivalent if $f_U$ coincides with $f_V$ on a dence open subset $W\subset U\cap V$.

A rational map $f\colon X \dasharrow Y$ is said to be {\em birational} if $f$ is dominant and there is a dominant rational map $g\colon Y\dasharrow X$ which is inverse as dominant rational maps.
It is equivalent to that $f$ induces a bijection between the set of generic points of $X$  and that of $Y$ and the induced local homomorphism $\O_{Y,f(\eta)}\to \O_{X,\eta}$ is an isomorphism for any generic point $\eta$ of $X$ (cf.\ Section 29.49 in \cite{Sta}).
Note that if $\varphi \colon D\to C$ is a birational morphism between proper curves, then $\varphi$ is finite.
\end{definition}

\begin{lemma} \label{birchar}
Let $f\colon X\to Y$ be a scheme-theoretically surjective morphism between algebraic schemes.
Then $f$ is birational if and only if $f_{\red}\colon X_{\red}\to Y_{\red}$ is birational and $X\times_{Y}Y_{\red}\cong X_{\red}$.
\end{lemma}

\begin{proof}
If $f$ is birational, then clearly $f_{\red}$ is birational and $X\times_{Y}Y_{\red}$ is reduced.
Conversely, assume that $f_{\red}$ is birational and $X\times_{Y}Y_{\red}$ is reduced.
Restricting any irreducible component of $Y$, we may assume that $Y$ is irreducible.
Let $\eta_X$ and $\eta_Y$ respectively denote the generic points of $X$ and $Y$.
Since $f$ is scheme-theoretically surjective, the local homomorphism $\O_{Y,\eta_Y}\to \O_{X,\eta_X}$ is injective.
By the assumption of $f$, we have $\O_{Y,\eta_Y}/\mathfrak{m}_Y\cong \O_{X,\eta_X}/\mathfrak{m}_X$ and $\mathfrak{m}_Y\O_{X,\eta_X}=\mathfrak{m}_X$,
where $\mathfrak{m}_X$ and $\mathfrak{m}_Y$ are the maximal ideals of $\O_{X,\eta_X}$ and $\O_{Y,\eta_Y}$, respectively.
Thus Nakayama's lemma says that $\O_{Y,\eta_Y}$ is isomorphic to $\O_{X,\eta_X}$ and so $f$ is birational.
\end{proof}

\begin{definition}[Finite rational maps of curves]
Let $\varphi \colon D\dasharrow C$ be a rational map between proper curves over a field $k$.
Then $\varphi$ is called {\em finite} if there is a finite morphism $\varphi'\colon D'\to C$ and a birational map $\pi\colon D'\dasharrow D$ such that $\varphi \circ \pi=\varphi'$, 
or equivalently, $\varphi$ sends any generic point of $D$ to a generic point of $C$.
The degree of $\varphi$ is defined as that of $\varphi'$ when $C$ is irreducible.
Note that any birational map between proper curves is finite of degree $1$.
\end{definition}

The following lemma includes the point-wise characterization of the birationality of proper curves  as a special case ($\varphi=\mathrm{id}$).

\begin{lemma} \label{ptbir}
Let $\varphi \colon D\to C$ and $\psi\colon C\to B$ be scheme-theoretically surjective finite morphisms between proper curves over a field $k$.
Assume that $B$ is integral or $\varphi$ is flat, and 
for any irreducible component $B_i$ of $B$, there exist closed regular points $\lambda_i\in B_{i,\red}$ and $z_i\in \psi^{-1}(B_i)_{\red}$  such that $(\psi\circ \varphi)^{-1}(\lambda_i)=\varphi^{-1}(z_i)$ holds scheme-theoretically and $k(z_i)\cong k(\lambda_i)$ via $\psi$. 
Then $\psi$ is birational.
\end{lemma}

\begin{proof}
Restricting any irreducible component of $B$, we may assume that $B$ is irreducible.
By assumption, we can take closed regular points $\lambda\in B_{\red}$ and $z\in C_{\red}$ such that $(\psi\circ \varphi)^{-1}(\lambda)=\varphi^{-1}(z)$ and $k(z)\cong k(\lambda)$ hold.
In particular, $\psi^{-1}(\lambda)=z$ holds as a set and thus $C$ is irreducible.
Let us put $C':=C\times_{B} B_{\red}$, $D':=D\times_{B} B_{\red}$ and $D'':=D\times_{C} C_{\red}$.
Then $(\psi\circ \varphi)^{-1}(\lambda)$ and $\varphi^{-1}(z)$ can be viewed as Cartier divisors on $D'$ and $D''$, respectively.
By Lemma~\ref{lbdeg}, we have $\deg((\psi\circ \varphi)^{-1}(\lambda))=\deg(\psi \circ \varphi|_{D'})[k(\lambda):k]$ and $\deg(\varphi^{-1}(z))=\deg(\varphi|_{D''})[k(z):k]=\deg(\psi_{\red})^{-1}\deg(\psi_{\red}\circ \varphi|_{D''})[k(z):k]$.
Thus $(\psi\circ \varphi)^{-1}(\lambda)=\varphi^{-1}(z)$ and $k(z)\cong k(\lambda)$ imply that 
$\psi_{\red}$ is birational and $D'=D''$.
Assume that $B$ is integral.
Then we have $C'=C$ and $D'=D$. 
Thus $C'$ is reduced since $C$ is the scheme-theoretic image of $\varphi$.
If $\varphi$ is flat, then $D'\to C'$ is also scheme-theoretically surjective and then $C'=C_{\red}$ holds.
Hence $\psi$ is birational by Lemma~\ref{birchar}.
\end{proof}

\begin{definition}[Invariants of proper curves] \label{gonmd}
Let $C$ be a proper curve over a field $k$.
Recall that the number
$$
\mathrm{gon}(C):=\min\{s \ |\ \text{$\exists$ a finite rational map $C\dasharrow \mathbb{P}^1$ of degree $s$}\}
$$
is called the {\em gonality} of $C$.
Let us define some similar invariants as follows.
The {\em pseudo-gonality} of $C$ is defined as
$$
\mathrm{pgon}(C):=\min\{s \ |\ \text{$\exists$ a finite morphism $C\to \mathbb{P}^1$ of degree $s$}\}.
$$

For a proper curve $C$ which is not birational (resp.\ isomorphic) to $\mathbb{P}^1$, we say that $C$ {\em has sufficiently many $($resp.\ base point free$)$ pencils of degree} $\le s$
 if there exists an infinite family $\{\alpha_i\colon C\dasharrow \mathbb{P}^1\}_{i\in I}$ of finite rational maps (resp.\ morphisms) of degree $\le s$ satisfying the following condition $(A)$:

\smallskip

\noindent
$(A)$ For any infinite distinct sequence $\{\alpha_{i_k}\}_{k=1}^{\infty}$, there exists an integer $l>0$ such that $(\alpha_{i_k})_{k=1}^{l}\colon C\dasharrow \prod_{k=1}^{l}\mathbb{P}^1$ is birational onto image.

\smallskip

The numbers
$$
\mathrm{gon}_{\infty}(C):=\min\{s \ |\ \text{$C$ has sufficiently many pencils of degree $\le s$}\}
$$
and
$$
\mathrm{pgon}_{\infty}(C):=\min\{s \ |\ \text{$C$ has sufficiently many base point free pencils of degree $\le s$}\}
$$
are respectively called the {\em $\infty$-gonality} of $C$ and the {\em pseudo-$\infty$-gonality} of $C$.
If $C$ is birational (resp.\ isomorphic) to $\mathbb{P}^1$, we define $\mathrm{gon}_{\infty}(C):=1$ (resp.\ $\mathrm{pgon}_{\infty}(C):=1$).

For a proper curve $C$, we define the {\em minimal degree of $C$ as a $($resp.\ pseudo-$)$plane curve model}, denoted by $\mathrm{s}_2(C)$ (resp.\ $\mathrm{ps}_2(C)$), as
the minimal number $s$ such that $C$ is birational to (resp.\ has birational morphism onto) a component-wise non-degenerate plane curve $\overline{C}\subset \mathbb{P}^2$ of degree $s$.
Here, a plane curve is {\em component-wise non-degenerate} if each irreducible component (may be non-reduced) is not a line in $\mathbb{P}^2$.
For example, $\mathrm{s}_2(\mathbb{P}^1)=\mathrm{ps}_2(\mathbb{P}^1)=2$ and $\mathrm{s}_2(2L)=\mathrm{ps}_2(2L)=2$, where $L$ is a line in $\mathbb{P}^2$.
\end{definition}

\begin{proposition} \label{invariants}
Let $C$ be a proper curve over a field $k$.
Then the following hold.

\smallskip

\noindent
$(1)$ $\mathrm{pgon}(C)< \infty$. 

\smallskip

\noindent
$(2)$ If $C$ is reduced, then $\mathrm{ps}_2(C)< \infty$.

\smallskip

\noindent
$(3)$ $\mathrm{gon}(C)$, $\mathrm{gon}_{\infty}(C)$ and $\mathrm{s}_2(C)$ are birational invariants of $C$.
 
\smallskip

\noindent
$(4)$ $\mathrm{gon}(C)\le \mathrm{pgon}(C)$, $\mathrm{gon}_{\infty}(C)\le \mathrm{pgon}_{\infty}(C)$, and $\mathrm{s}_2(C)\le \mathrm{ps}_2(C)$.

\smallskip

\noindent
$(5)$ $\mathrm{gon}(C)\le \mathrm{gon}_{\infty}(C)$ and $\mathrm{pgon}(C)\le \mathrm{pgon}_{\infty}(C)$.

\smallskip

\noindent
$(6)$ Assume that the base field $k$ is infinite.Then $\mathrm{pgon}_{\infty}(C)\le \mathrm{ps}_2(C)$ and $\mathrm{gon}_{\infty}(C)\le \mathrm{s}_2(C)$ hold.
If further assume that $C$ has an irreducible component with geometric multiplicity $1$ which has infinitely many rational points, then $\mathrm{pgon}_{\infty}(C)\le \mathrm{ps}_2(C)-1$ and
$\mathrm{gon}_{\infty}(C)\le \mathrm{s}_2(C)-1$ hold.

\smallskip

\noindent
$(7)$ If $C=\cup_{i}C_i$ is the irreducible decomposition, then 
$$
\mathrm{gon}(C)=\sum_{i}\mathrm{gon}(C_i),\quad \mathrm{gon}_{\infty}(C)=\sum_{i}\mathrm{gon}_{\infty}(C_i),\quad \mathrm{s}_2(C)=\sum_{i}\mathrm{s}_2(C_i),
$$
$$
\mathrm{pgon}(C)\ge \sum_{i}\mathrm{pgon}(C_i),\quad \mathrm{pgon}_{\infty}(C)\ge \sum_{i}\mathrm{pgon}_{\infty}(C_i),\quad \mathrm{ps}_2(C)\ge \sum_{i}\mathrm{ps}_2(C_i).
$$

\smallskip

\noindent
$(8)$ $\mathrm{gon}(C)=\min_{C'}\{\mathrm{pgon}(C')\}$ and $\mathrm{s}_2(C)=\min_{C'}\{\mathrm{ps}_2(C')\}$, where $C'$ runs through all proper curves birational to $C$.
\end{proposition}

\begin{proof}
(1) follows from the fact that any proper curve over a field $k$ is projective (cf.\ Lemma~33.42.4 in \cite{Sta}).
This fact also implies (2) since we can show that any reduced projective curve is birational to a plane curve by taking projections from general rational points in the projective space repeatedly.
(3), (4), (5), (7) and (8) are trivial.

We will show (6).
Let us put $s:=\mathrm{ps}_2(C)$ and take a birational morphism $\pi\colon C\to \overline{C}\subset \mathbb{P}^2$ onto a component-wise non-degenerate plane curve $\overline{C}$ of degree $s$, where we may assume $s<\infty$ in order to prove (6).
Let $E\subset \mathbb{P}^2$ be an irreducible component of $\overline{C}$ of geometric multiplicity $1$ with infinitely many rational points if $C$ has such a component, or a general line otherwise.
Let $I$ be the infinite set of rational points $p$ in the regular locus of $E$ such that $p\not\in \overline{C}\setminus E$ and $\pi$ is an isomorphism over $p$.
For any $p\in I$, let $\overline{\alpha}_p\colon \overline{C}\to \mathbb{P}^1$ denote the projection from $p$ and write $\alpha_p:=\overline{\alpha}_p\circ \pi$.
Then $\alpha_p$ is a finite morphism of degree $s-1$ (resp.\ $s$) if $E\subset \overline{C}$ (resp.\ otherwise).
Now we take an infinite distinct sequence $p_1,p_2,\ldots,$ of rational points in $I$ arbitrarily.
Let $S$ be the dense subset of $C$ consisting of closed points $x$ such that $\alpha_{p_1}(x)$ is a rational point of $\mathbb{P}^1$ and $\pi$ is an isomorphism at $x$.
For any $x\in S$, let $L_1$ denote the line passing through $\pi(x)$ and $p_1$, that is, the closure of the fiber of the projection $\mathbb{P}^2\dasharrow \mathbb{P}^1$ from $p_1$ at the rational point $\lambda_1:=\alpha_{p_1}(x)$.
Since the intersection of $E$ and $L_1$ is finite, we can take $p_m$ such that the ``fat'' line $L_m$ passing through $\pi(x)$ and $p_m$ is not equal to $L_1$, where $L_m$ is defined by the closure of the fiber of the projection $\mathbb{P}^2\dasharrow \mathbb{P}^1$ from $p_m$ at the  point $\lambda_m:=\alpha_{p_m}(x)$ (when $\lambda_m$ is not a rational point, $L_m$ has exactly one rational point $p$ as a singularity).
Since $L_1$ and $L_m$ respectively belong to  $|\O_{\mathbb{P}^2}(1)|$ and $|\O_{\mathbb{P}^2}([k(\lambda_m):k])|$, we have
$$
[k(\pi(x)):k]\le (L_1L_m)_{\pi(x)}\le L_1L_m=[k(\lambda_m):k],
$$
which implies $k(\pi(x))\cong k(\lambda_m)$ and $L_1\cap L_m=\pi(x)$ scheme-theoretically,
where $(L_1L_m)_{\pi(x)}$ is the local intersection number of $L_1$ and $L_m$ at $\pi(x)$ (cf.\ Definition~9.1.1 in \cite{Liu}).
Thus we have $\alpha_{p_1}^{-1}(\lambda_1)\cap \alpha_{p_m}^{-1}(\lambda_m)=x$ scheme-theoretically.
Hence we can take $l\gg 0$ such that $(\alpha_{p_k})_{k=1}^{l}\colon C\to \prod_{k=1}^{l}\mathbb{P}^1$ is birational onto image by using Lemma~\ref{ptbir}, which implies the claim of (6) for $\mathrm{ps}_2(C)$.
The claim of (6) for $\mathrm{s}_2(C)$ follows easily by combining the claim  (6) for $\mathrm{ps}_2(C)$ with (4) and (8).
\end{proof}

\begin{lemma}\label{mono}
Let $\{\alpha_i\colon C\to \mathbb{P}^1\}_{i\in I}$ be an infinite family of finite morphisms on a proper curve $C$ satisfying the condition $(A)$ in Definition~\ref{gonmd}.
Then it satisfies the following condition $(B)$:

\smallskip

\noindent
$(B)$ For any infinite subset $J\subset I$, any proper curve $D$ and any two finite morphisms $\varphi_1, \varphi_2\colon D\to C$, the condition that $\alpha_i\circ \varphi_1=\alpha_i\circ \varphi_2$ for any $i\in J$ implies $\varphi_1=\varphi_2$.
\end{lemma}

\begin{proof}
Let $J\subset I$, $D$ and $\varphi_1, \varphi_2\colon D\to C$ be as in the condition $(B)$
and assume that $\alpha_i\circ \varphi_1=\alpha_i\circ \varphi_2$ for any $i\in J$.
Take an infinite distinct sequence $i_1,i_2,\ldots$ in $J$ arbitrarily.
Then $(\alpha_{i_k})_{k=1}^{l}\colon C\to \prod_{k=1}^{l}\mathbb{P}^1$ is birational onto image for some $l>0$ by the condition $(A)$.
Let $U\subset C$ be the dense open subset over which $(\alpha_{i_k})_{k=1}^{l}$ is an  isomorphism onto image.
Then $\varphi_1=\varphi_2$ holds on a dense open subset $\varphi_1^{-1}(U)\cap \varphi_2^{-1}(U)$ of $D$.
By the valuative criterion of separatedness (cf.\ Chapter II, Theorem~4.3 in \cite{Har}), we have $\varphi_1=\varphi_2$.
\end{proof}

\end{document}